\documentclass[a4paper, 10pt]{amsart}

\usepackage{amsmath, amssymb, amscd}
\usepackage{color}
\usepackage{tabularx}
\usepackage{comment}
\usepackage{hhline}
\usepackage[margin=2.27cm]{geometry}
\usepackage{multirow}
\usepackage{caption}
\usepackage{tabularray}
\usepackage{xcolor}
\usepackage{longtable,xtab,booktabs}
\usepackage[colorlinks]{hyperref}
\definecolor{linkblue}{RGB}{1,1,190}
\definecolor{citered}{RGB}{190,1,1}
\hypersetup{
	linkcolor=linkblue,
	urlcolor=linkblue,
	citecolor=citered
}

\def\doi#1{{\small\href{https://doi.org/#1}{\path{doi:#1}}}}
\def\arxiv#1{{\small\href{http://www.arxiv.org/abs/#1}{\path{arXiv:#1}}}}
\def\url#1{{\small\href{#1}{\path{#1}}}}

\theoremstyle{plain}
\newtheorem{theorem}{\bf Theorem}[section]

\newtheorem{lemma}[theorem]{\bf Lemma}

\theoremstyle{definition}

\newcommand{\bdot}{\boldsymbol{\cdot}}

     \DeclareMathOperator{\ord}{ord}
\DeclareMathOperator{\lcm}{lcm}

\DeclareMathOperator{\Aut}{Aut}     \DeclareMathOperator{\GL}{GL}
\DeclareMathOperator{\PSL}{PSL}     \DeclareMathOperator{\GAP}{\mathbf{GAP}}
\DeclareMathOperator{\SL}{SL}       \DeclareMathOperator{\Inn}{Inn}
\DeclareMathOperator{\M}{M}         \DeclareMathOperator{\PSU}{PSU}
\DeclareMathOperator{\id}{id}

\renewcommand{\time}{\negthinspace \times \negthinspace}

\renewcommand{\t}{\, | \,}
\newcommand{\und}{\quad \text{ and } \quad}

\numberwithin{equation}{section}

\begin{document}

\title[A classification of finite groups with small Davenport constant]{A classification of finite groups with \\ small Davenport constant}

\author{Jun Seok Oh}
\address{Department of Mathematics Education, Jeju National University, Jeju 63243, Republic of Korea}
\email{junseok.oh@jejunu.ac.kr}



\begin{abstract}
Let $G$ be a finite group.
By a sequence over $G$, we mean a finite unordered string of terms from $G$ with repetition allowed, and we say that it is a product-one sequence if its terms can be ordered so that their product is the identity element of $G$.
Then, the Davenport constant $\mathsf D (G)$ is the maximal length of a minimal product-one sequence, that is a product-one sequence which cannot be factored into two non-trivial product-one subsequences.
The Davenport constant is a combinatorial group invariant that has been studied fruitfully over several decades in additive combinatorics, invariant theory, and factorization theory, etc.
Apart from a few cases of finite groups, the precise value of the Davenport constant is unknown.
Even in the abelian case, little is known beyond groups of rank at most two.
On the other hand, for a fixed positive integer $r$, structural results characterizing which groups $G$ satisfy $\mathsf D (G) = r$ are rare.
We only know that there are finitely many such groups.
In this paper, we study the classification of finite groups based on the Davenport constant.
\end{abstract}

\maketitle


\bigskip
\section{Introduction} \label{sec:1}
\bigskip

Let $G$ be a finite group.
A sequence over $G$ means a finite unordered string of terms from $G$ with repetition allowed, and the number of terms is called the length of a sequence.
We say that a sequence is product-one if its terms can be ordered so that their product in $G$ is the identity element of $G$, and say that a product-one sequence is minimal if it cannot be factored into two non-trivial product-one sequences.
A sequence is called product-one free if no subproduct of terms equals the identity element of $G$.
The large (resp., small) Davenport constant $\mathsf D (G)$ (resp., $\mathsf d (G)$) is the maximal length of a minimal product-one (resp., product-one free) sequence over $G$.
The Davenport constant of a finite group is a classical combinatorial group invariant that has been fruitfully studied over several decades (see \cite{Ge-HK06} for the monograph).
Because the investigation of the Davenport constant has its origin in the arithmetical problem of algebraic number fields (as posted in \cite{Ol69}), earlier work often focused on the abelian setting.
Indeed, the Davenport constant $\mathsf D (G)$ of the class group $G$ of a number field $K$ is the maximal number of prime ideals occurring in the prime ideal decomposition of an irreducible element in the ring of algebraic integers of $K$ (see \cite[Chapter 9]{Na04} for more details).
With various connections to combinatorics, invariant theory, and factorization theory, the tools have now developed to the point where non-abelian groups have begun to be explored.
For recent progress, we refer to \cite{Gr13b,Ge-Gr13,Ga-Li-Pe14,Br-Ri18,Ha-Zh19,Oh-Zh20a,Oh-Zh20b,Ga-Li-Qu21,Zh21,Qu-Li-Tee22,Zhao-Zh22,Qu-Li22,Qu-Li23,Fa-Zh23,Av-Br-Ri23,An-Cz-Do-Sz25}.

Suppose that $G$ is abelian, say $G \cong C_{n_1} \times \cdots \times C_{n_r}$ with $n_1, \ldots, n_r \in \mathbb N$ and $1 < n_1 \mid \cdots \mid n_r$.
Then, a simple argument shows that
\[
  \small{\sum}_{i=1}^{r} (n_i - 1) + 1 \, \le \, \mathsf d (G) + 1 \, = \, \mathsf D (G) \,,
\]
and equality holds for $r \le 2$ and for some classes of groups with $r \ge 3$ (see \cite{Ge-Sc92, Ga00} and \cite[Corollary 4.2.13]{Ge09a}).
Moreover, in \cite{Ge-Sc92}, the authors showed that there are infinitely many groups $G$ of rank $r \ge 4$ for which this inequality is strict.
Hence, determining the precise value of the Davenport constant is classical and quite difficult problem even in the abelian setting.
However, if we know the specific value of $\mathsf D (G)$, then since both the rank and the exponent of $G$ are bounded above by $\mathsf d (G)$, it is possible to list up all possible candidates of $G$ with the fixed $\mathsf D (G)$.

Suppose that $G$ is non-abelian.
Then, the structure of the monoid $\mathcal B (G)$ of all product-one sequences is well-studied.
It is finitely generated C-monoid, whence it is a Mori monoid with non-trivial conductor to its complete integral closure and the class group of the complete integral closure is finite.
C-monoids have been introduced to study the arithmetic of higher-dimensional non-completely integrally closed Mori domains, and the monoid $\mathcal B (G)$ was the first class of C-monoids for which we have an insight into their algebraic structure.
Indeed, the monoid $\mathcal B (G)$ is seminormal if and only if the commutator subgroup of $G$ has at most two elements if and only if the class semigroup of $\mathcal B (G)$ is Clifford, i.e., it is a union of groups (see \cite[Corollary 3.12]{Oh19}).
We refer the reader to \cite{Re13,Ge-Ra-Re15,Ge-Zh19,Oh20,Oh22} for more recent work on the algebraic structure of C-monoids.
The Davenport constant $\mathsf D (G)$ is the maximal length of an irreducible element of the monoid $\mathcal B (G)$.
Thus it is the crucial (combinatorial) invariant for studying the arithmetic of the (C-)monoid $\mathcal B (G)$.
However, there are only two results so far about the exact value of the Davenport constant by Geroldinger and Grynkiewicz \cite{Ge-Gr13}, which states that $\mathsf D (G) = \mathsf d (G) + |G'| = \frac{1}{2}|G| + |G'|$ for finite non-abelian groups $G$ with a cyclic index 2 subgroup, where $G'$ denotes the commutator subgroup of $G$, and by Grynkiewicz \cite[Theorem 5.1]{Gr13b}, which states that $\mathsf D (G) = 2q$ for non-abelian groups $G$ of order $pq$ with $p$ and $q$ distinct primes such that $p \mid q-1$.
Moreover, in \cite{Gr13b}, Grynkiewicz studied the upper bound for $\mathsf D (G)$, and in \cite{Qu-Li-Tee22, Qu-Li-Tee25}, Qu, Li, and Teeuwsen studied the upper bound for $\mathsf d (G)$ for general non-abelian groups $G$.
Unlike abelian groups, we are not aware of a finite non-abelian group with $\mathsf d (G) + 1 = \mathsf D (G)$, and so it is worthwhile to mention whether $\mathsf d (G) + 1 < \mathsf D (G)$ for all finite non-abelian group $G$.
Moreover, there are no fundamental methods classifying finite groups with the specific (large) Davenport constant, and in contrast with the abelian setting, we do not even know how many candidates exist.

This combinatorial aspect of the Davenport constant has intimate connections to invariant theory and factorization theory.
To begin with invariant theory, we let $\boldsymbol{\beta} (G)$ be the Noether number of $G$.
Then, Schmid \cite{Sc91} showed that $\mathsf d (G) + 1 = \boldsymbol{\beta} (G)$ for all abelian groups $G$, and in \cite{Cz-Do14a}, Cziszter and Domokos showed that $\mathsf d (G) + 1 \le \boldsymbol{\beta} (G)$ for groups $G$ with a cyclic subgroup of index 2.
It is still an open question whether the inequality holds true for all finite groups.

To connect with factorization theory, we let $\mathcal L (G) = \{ \mathsf L (S) \mid S$ is a product-one sequence over $G \}$ be the collection of sets of lengths $\mathsf L (S)$, where $\mathsf L (S)$ is the set of all lengths $\ell \in \mathbb N$ with $S = U_1 \bdot \ldots \bdot U_{\ell}$ for some minimal product-one sequences $U_1, \ldots, U_{\ell}$.
Then, the standing conjecture is that, for finite groups $G_1$ and $G_2$, $\mathcal L (G_1) = \mathcal L (G_2)$ implies that $G_1$ and $G_2$ are isomorphic.
With a strong connection to algebraic number theory, the initial study focused only on abelian groups (see \cite{Ge-Zh20} for an overview).
However, with the development of methods for non-abelian groups, this problem has attracted wide attention in the literature.
In this direction, a first step was done for specific finite groups (see \cite[Theorem 3.14]{Oh19} for groups with the Davenport constant 6, and \cite[Corollary 6.13]{Ge-Gr-Oh-Zh22} for finite dihedral groups).
Very recently, we obtained an affirmative answer that the standing conjecture may also hold for non-abelian groups (see \cite{Ge-Oh25}).
If $\mathcal L (G_1) = \mathcal L (G_2)$, then $\mathsf D (G_1) = \mathsf D (G_2)$ (see \cite[Proposition 5.6]{Oh20}).
Therefore, a classification of finite groups having the specific Davenport constant plays a key role in the literature.

In the present paper, we address finite groups with small values of their Davenport constant, and the main result is to provide the following classification.

\smallskip
\begin{theorem} \label{thm:classify}~
Let $G$ be a finite group with $|G| \ge 2$.
\begin{enumerate}
\item If $\mathsf D (G) \le 7$, then $G$ is isomorphic to one of the groups listed in {\sc{Table}} $1$.

\smallskip
\item If $8 \le \mathsf D (G) \le 9$, then $G$ is either a non-abelian group having a proper subgroup of order $32$, or isomorphic to one of the groups listed in {\sc{Table}} $2$.
\end{enumerate}

\begin{table}[ht]  \hspace{-20pt}
\begin{minipage}{.59\textwidth}
\begin{tabular}{|c|c|c|} \hline
\rule{0pt}{13pt} $\mathsf D (G)$ & $G$ & $GAP$ \\[.5em] \hhline{|=|=|=|}
\rule{0pt}{14pt} $2$ & $C_2$ & $(2,1)$ \\[.5em] \hline
\rule{0pt}{13pt} \multirow{2}{*}{$3$} & $C_3$ & $(3,1)$ \\ & $C^{2}_2$ & $(4,2)$ \\[.5em] \hline
\rule{0pt}{13pt} \multirow{2}{*}{$4$} & $C_4$ & $(4,1)$ \\ & $C^{3}_2$ & $(8,5)$ \\[.5em] \hline
\rule{0pt}{13pt} \multirow{4}{*}{$5$} & $C_5$ & $(5,1)$ \\ & $C_2 \times C_4$ & $(8,2)$ \\ & $C^{2}_3$ & $(9,2)$ \\ & $C^{4}_2$ & $(16,14)$ \\[.5em] \hline
\rule{0pt}{13pt} \multirow{6}{*}{$6$} & $C_6$ & $(6,2)$ \\ & $C^{2}_2 \times C_4$ & $(16,10)$ \\ & $C^{5}_2$ & $(32,51)$ \\[.4em] \cline{2-3} \rule{0pt}{15pt} & $D_6$ & $(6,1)$ \\ & $D_8$ & $(8,3)$ \\ & $Q_8$ & $(8,4)$ \\[.5em] \hline
\rule{0pt}{14pt} \multirow{11}{*}{$7$} & $C_7$ & $(7,1)$ \\ & $C_2 \times C_6$ & $(12,5)$ \\ & $C^{2}_4$ & $(16,2)$ \\ & $C^{3}_3$ & $(27,5)$ \\ & $C^{3}_2 \times C_4$ & $(32,45)$ \\ & $C^{6}_2$ & $(64,267)$ \\[.5em] \cline{2-3} \rule{0pt}{15pt} & $A_4$ & $(12,3)$ \\ & $C^{2}_2 \rtimes C_4$ & $(16,3)$ \\ & $C_2 \times D_8$ & $(16,11)$ \\ & $C_2 \times Q_8$ & $(16,12)$ \\ & $(C_2 \times C_4) \rtimes C_2$ & $(16,13)$ \\[.5em] \hline
\end{tabular}
\caption{$\mathsf D (G) \le 7$ \quad \hspace{100pt} }
\end{minipage} \hspace{-75pt}
\begin{minipage}{.4\textwidth}
\begin{tabular}{|c|c|c|} \hline
\rule{0pt}{13pt} $\mathsf D (G)$ & $G$ & $GAP$ \\[.5em] \hhline{|=|=|=|}
\rule{0pt}{13pt} \multirow{14}{*}{$8$} & $C_8$ & $(8,1)$ \\ & $C_3 \times C_6$ & $(18,5)$ \\ & $C^{2}_2 \times C_6$ & $(24,15)$ \\ & $C_2 \times C^{2}_4$ & $(32,21)$ \\ & $C^{4}_2 \times C_4$ & $(64,260)$ \\ & $C^{7}_2$ & $(128,2328)$ \\[.3em] \cline{2-3} \rule{0pt}{13pt} & $C_4 \rtimes C_4$ & $(16,4)$ \\ & $H_{27}$ & $(27,3)$ \\ & $C_2 \times (C^{2}_2 \rtimes C_4)$ & $(32,22)$ \\ & $C^{2}_2 \times D_8$ & $(32,46)$ \\ & $C^{2}_2 \times Q_8$ & $(32,47)$ \\ & $C_2 \times \big( (C_4 \times C_2) \rtimes C_2\big)$ & $(32,48)$ \\ & $C^{3}_2 \rtimes C^{2}_2$ & $(32,49)$ \\ & $(C_2 \times Q_8) \rtimes C_2$ & $(32,50)$ \\[.5em] \hline
\rule{0pt}{13pt} \multirow{16}{*}{$9$} & $C_9$ & $(9,1)$ \\ & $C_2 \times C_8$ & $(16,5)$ \\ & $C^{2}_5$ & $(25,2)$ \\ & $C^{3}_2 \times C_6$ & $(48,52)$ \\ & $C^{2}_2 \times C^{2}_4$ & $(64,192)$ \\ & $C^{4}_3$ & $(81,15)$ \\ & $C^{5}_2 \times C_4$ & $(128,2319)$ \\ & $C^{8}_2$ & $(256,56092)$ \\[.3em] \cline{2-3} \rule{0pt}{13pt} & $Q_{12}$ & $(12,1)$ \\ & $D_{12}$ & $(12,4)$ \\ & $(C_2 \times C_4) \rtimes C_4$ & $(32,2)$ \\ & $C_2 \times (C_4 \rtimes C_4)$ & $(32,23)$ \\ & $C^{2}_4 \rtimes C_2$ & $(32,24)$ \\ & $C_4 \times D_8$ & $(32,25)$ \\ & $C_4 \times Q_8$ & $(32,26)$ \\ & $C^{4}_2 \rtimes C_2$ & $(32,27)$ \\[.5em] \hline
\end{tabular}
\caption{$8 \le \mathsf D (G) \le 9$ \quad \hspace{-30pt} }
\end{minipage}
\end{table}
\end{theorem}

\smallskip
In Section~\ref{sec:2}, we gather the required background, and the proof of Theorem~\ref{thm:classify} will be given in Section~\ref{sec:3} with a series of lemma.
In the proof, we substantially use the results of the computation of the Davenport constant of non-abelian groups of small order (\cite{Cz-Do-Sz18,An-Cz-Do-Sz25}) and the classification of finite perfect groups (\cite{Sa81}).
In \cite{An-Cz-Do-Sz25}, the authors computed the Davenport constants for all non-abelian groups of order at most 42, and this allows us to list up non-abelian groups of order at most 42 having the Davenport constant less than 10.
Since $\mathsf D (H) \le \mathsf D (G)$ for any subgroup $H$ of $G$, the main approach is to find the certain subgroup having the Davenport constant at least 10.
Depending on the group order, we find the desired subgroups by making use of well-known theorems in group theory, such as Feit-Thompson's Theorem, Sylow's Theorem, and Hall's Theorem.
Specifically, if $G$ is perfect, then based on the classification, we can construct a specific subgroup $H$ of order at most 42 with a sufficiently large Davenport constant.
If $G$ is not perfect, then since $G$ has at least one proper normal subgroup, we make full use of the classical method for classifying finite groups in terms of the semidirect product.
Once this is done, we obtain the specific group generators and their relations, and then we can achieve to find either the desired subgroups or a minimal product-one sequence of sufficiently large length.
However, the difficulty of classification arises in the case of non-abelian 2-groups.
There are many non-abelian 2-groups where we need to clarify the generators and relations (for example, there are 256 non-abelian groups of order 64, and 2313 of order 128, etc), and many of these non-abelian 2-groups have subgroups with relatively small values of their Davenport constant.
We note that it is possible to search for non-abelian 2-groups whose commutator subgroups have only two elements in the database of finite groups available on the website \cite{LMFDB}.
Hence, according to the characterization of the monoid $\mathcal B (G)$ being seminormal (\cite{Oh19}), we figure out that there are many non-abelian 2-groups for which their monoids of product-one sequences are seminormal.
It is also worth noting that the study of the seminormal monoid $\mathcal B (G)$ is an ongoing subject of interest in the literature.
On the other hand, it is easy to classify abelian 2-groups using the structure theorem, and then their Davenport constant can be fully determined from their structure (see \cite{Ge-Sc92}).
Thus, we first present a lower bound for the Davenport constant of non-abelian 2-groups of order at least 32 (Lemma~\ref{lem:2group}).
Then, we prove the key result of this paper, which states that every finite non-abelian group of order exceeding 42 has the Davenport constant at least 8 if it has a proper subgroup of order 32, and the Davenport constant at least 10 if it has no proper subgroup of order 32 (Lemma~\ref{lem:DC}).
This allows us to obtain the complete list of all finite groups whose Davenport constant is at most 7.


\bigskip
\section{Preliminaries} \label{sec:2}
\bigskip

Let $\mathbb N$ denote the set of positive integers and $\mathbb N_0 = \mathbb N \cup \{ 0 \}$ denote the non-negative integers.
For integers $a, b \in \mathbb Z$, $[a,b] = \{ x \in \mathbb Z \mid a \le x \le b \}$ is the discrete interval.

Let $G$ be a multiplicatively written finite group with identity element $1_G$.
For an element $g \in G$ and a subset $G_0 \subseteq G$, we denote by $\ord (g) \in \mathbb N$ the order of $g$, and by $\langle G_0 \rangle$ the subgroup of $G$ generated by $G_0$.
Moreover, $G'$ denotes the commutator subgroup of $G$, and $Z (G)$ denotes the center of $G$.
For a (suitable) positive integer $n$ and for an odd prime $p$, we denote by

\begin{itemize}
\item $C_n$ a {\it cyclic group} of order $n$,

\smallskip
\item $Q_{4n}$ a {\it dicyclic group} of order $4n$,

\smallskip
\item $D_{2n}$ a {\it dihedral group} of order $2n$,

\smallskip
\item $SD_{2^{n}}$ a {\it semidihedral group} of order $2^{n}$,

\smallskip
\item $M_{2^{n}}$ a {\it modular maximal-cyclic group} of order $2^{n}$,

\smallskip
\item $H_{p^{3}}$ a {\it a Heisenberg group} of order $p^{3}$ with exponent $p$,

\smallskip
\item $S_n$ a {\it symmetric group} of degree $n$,

\smallskip
\item $A_n$ an {\it alternating group} of degree $n$,

\smallskip
\item $\GL (n,q)$ a {\it general linear group} of degree $n$ over a finite field of order $q$,

\smallskip
\item $\SL (n,q)$ a {\it special linear group} of degree $n$ over a finite field of order $q$, and

\smallskip
\item $\PSL (n,q)$ a {\it projective special linear group} of degree $n$ over a finite field of order $q$.
\end{itemize}

\noindent
In the tables of this paper, the column `$GAP$' presents the identifier of a group $G$, where $(m, n)$ indicates the function $\texttt{SmallGroup(m,n)}$, representing the $n^{\textnormal{th}}$ group of order $m$ in the Small Groups Library of $\GAP$ (see \cite{GAP}).
The column `$G$' presents the structure description of a group $G$, as returned by the function $\texttt{StructureDescription(G)}$ in $\GAP$.
In the proof, we occasionally use basic functions, such as $\texttt{AutomorphismGroup(SmallGroup(m,n))}$ and $\texttt{Size(AutomorphismGroup(SmallGroup(m,n)))}$, in $\GAP$ to compute the order of particular groups.

A {\it sequence} over $G$ is a finite unordered string of terms from $G$, where the repetition of terms is allowed, and we consider sequences as elements of the free abelian monoid $\mathcal F (G)$ with basis $G$, i.e., it consists of all sequences over $G$ with concatenation of sequences as the operation $\bdot$.
Thus, a sequence $S \in \mathcal F (G)$ has the form
\begin{equation} \label{eq:sequence}~
  S \, = \, g_1 \bdot \ldots \bdot g_{\ell} \, = \, \small{\prod}^{\bullet}_{g \in G} g^{[\mathsf v_g (S)]} \,,
\end{equation}
where $g_1, \ldots, g_{\ell} \in G$ and $\mathsf v_g (S) = | \{ i \in [1,\ell] \mid g_i = g \} |$.
The {\it length} of $S$ is denoted by $|S| = \ell = \sum_{g \in G} \mathsf v_g (S)$.
To avoid confusion, we usually use brackets to denote exponentiation in $\mathcal F (G)$.
Thus, for $g \in G$, $S \in \mathcal F (G)$, and $n \in \mathbb N_0$, we have that
\[
  g^{n} \in G \,, \qquad g^{[n]} = \underset{n}{\underbrace{g \bdot \ldots \bdot g}} \in \mathcal F (G) \,, \und S^{[n]} = \underset{n}{\underbrace{S \bdot \ldots \bdot S}} \in \mathcal F (G) \,.
\]
If $T, S \in \mathcal F (G)$, then we say that $T$ {\it divides} $S$ in $\mathcal F (G)$, denoted by $T \mid S$, if $S = T \bdot U$ for some $U \in \mathcal F (G)$.
In this case, $T$ is called a {\it subsequence} of $S$.
Moreover, if $T^{[n]}$ divides $S$ in $\mathcal F (G)$, then $S \bdot T^{[-n]} = S \bdot \big( T^{[n]} \big)^{[-1]} \in \mathcal F (G)$ denotes the subsequence of $S$ obtained by removing the terms of $T^{[n]}$ from $S$.

Let $S \in \mathcal F (G)$ be a sequence as in (\ref{eq:sequence}).
Then,
\[
  \pi (S) = \{ g_{\sigma (1)} \cdots g_{\sigma (\ell)} \in G \mid \sigma \mbox{ is a permutation of } [1,\ell] \}
\]
is the {\it set of products} of $S$, and it is easy to see that $\pi (S)$ is contained in a $G'$-coset.
Note that $|S| = 0$ if and only if $S$ is a trivial sequence, and in that case, we use the convention that $\pi (S) = \{ 1_G \}$.
The sequence $S$ is called

\begin{itemize}
\item {\it product-one} if $1_G \in \pi (S)$,

\smallskip
\item {\it product-one free} if $1_G \notin \pi (T)$ for every non-trivial subsequence $T$ of $S$,
\end{itemize}
and any ordered product, that equals $1_G$, in $\pi (S)$ is called a {\it product-one equation} of $S$.
A {\it minimal product-one sequence} is a product-one sequence which cannot be factored into two non-trivial product-one subsequences.
Then, the set
\[
  \mathcal B (G) = \{ S \in \mathcal F (G) \mid 1_G \in \pi (S) \} \, \subseteq \, \mathcal F (G)
\]
is a submonoid of $\mathcal F (G)$, called the {\it monoid of product-one sequences} over $G$.
It is easy to see that every product-one sequence in $\mathcal B (G)$ has a factorization into minimal product-one sequences and that there are only finitely many distinct factorizations.
Moreover, we will use the following simple fact without further mention: if $g_1 \bdot \ldots \bdot g_{\ell} \in \mathcal B (G)$ with $g_1 \cdots g_{\ell} = 1_G$, then $g_{\ell} g_1 \cdots g_{\ell-1} = g_{\ell} (g_1 \cdots g_{\ell})g^{-1}_{\ell} = 1_G$, and by iterating the same argument, we obtain that
\[
  g_i \cdots g_{\ell} g_1 \cdots g_{i-1} = 1_G \,\, \mbox{ for all } \, i \in [1,\ell] \,.
\]
We denote by
\begin{itemize}
\item $\mathsf D (G) = \sup \{ |S| \mid S$ is a minimal product-one sequence$\}$ the {\it large Davenport constant} of $G$,

\smallskip
\item $\mathsf d (G) = \sup \{ |S| \mid S$ is a product-one free sequence$\}$ the {\it small Davenport constant} of $G$.
\end{itemize}

\noindent
It is easy to verify that
\[
  \mathsf d (G) + 1 \, \overset{(1)}{\le} \, \mathsf D (G) \, \overset{(2)}{\le} \, |G|
\]
(see \cite[Lemma 2.4]{Ge-Gr13}), and it is well known that (2) satisfies equality if and only if $G$ is isomorphic to a cyclic group, or to a dihedral group of order $2n$ with $n$ odd (see \cite[Corollary 3.12]{Cz-Do-Ge16}).
However, so far, we only know that (1) satisfies equality only if $G$ is abelian.
Thus, we are not aware of any example of a finite non-abelian group $G$ with $\mathsf d (G) + 1 = \mathsf D (G)$.
Moreover, we observe that every non-abelian group listed in the tables of Theorem~\ref{thm:classify} satisfies $\mathsf d (G) + 1 < \mathsf D (G)$ (see Lemma~\ref{lem:dD}).
We refer the reader to \cite{Gr13b} for a general upper bound for the Davenport constant.
We present here the results related to the Davenport constant, which we refer to frequently throughout the manuscript.

\smallskip
\begin{lemma}$($\cite[Theorem 1 and Corollary 2]{Ge-Sc92}$)$ \label{lem:abD}~
Let $G \cong C_{n_1} \times \cdots \times C_{n_r}$, with $n_1, \ldots, n_r \in \mathbb N$ and $1 < n_1 \mid \cdots \mid n_r$, be a finite abelian group.
Then, $\mathsf D (G) \ge \sum_{i=1}^{r} (n_i - 1) + 1$, and equality holds in each of the following cases:
\begin{enumerate}
\item[(a)] $G$ is a $p$-group.

\smallskip
\item[(b)] $G$ has rank $r \le 2$.

\smallskip
\item[(c)] $G \cong C^{k}_2 \times C_{2n}$ with $k \le 3$ and odd $n$.
\end{enumerate}
\end{lemma}

\smallskip
\begin{lemma}$($\cite[Theorem 1.1]{Ge-Gr13}$)$ \label{lem:nabD}~
Let $G$ be a finite non-cyclic non-abelian group having a cyclic index $2$ subgroup, and $G'$ be its commutator subgroup. Then
\[
  \mathsf D (G) = \mathsf d (G) + |G'| = \frac{1}{2} |G| + |G'| \,.
\]
\end{lemma}


\medskip
\section{Proof of Theorem~\ref{thm:classify}} \label{sec:3}
\bigskip

We start with the basic observation for the Davenport constant.
If $H$ is a proper subgroup of $G$, then $\mathsf d (H) < \mathsf d (G)$ and $\mathsf D (H) \le \mathsf D (G)$.
The latter inequality is always strict when $G$ is abelian, but equality may happen for a non-abelian group $G$ (see \cite[Example 2.1]{Ge-Gr-Oh-Zh22}).

\smallskip
\begin{lemma} \label{lem:ineqDC}~
Let $H \subset G$ be a proper subgroup, and $N \unlhd G$ be a normal subgroup.
\begin{enumerate}
\item $\mathsf d (H) + 2 \le \mathsf D (G)$, and in particular, $\mathsf D (H) + 1 \le \mathsf D (G)$ if $H$ is abelian.

\smallskip
\item $\mathsf d (N) + \mathsf d (G/N) \le \mathsf d (G)$.

\smallskip
\item If $G/N$ is abelian, then $\mathsf D (G/N) \le \mathsf D (G)$, and in particular, $\mathsf D (G/G') \le \mathsf D (G)$.
\end{enumerate}
\end{lemma}

\begin{proof}
1. Let $S = g_1 \bdot \ldots \bdot g_{\ell}$ be a product-one free sequence over $H$ of length $|S| = \ell = \mathsf d (H)$.
Then, $g = g_1 \cdots g_{\ell} \in H$, and take $h \in G \setminus H$. Then
\[
  U = g_1 \bdot \ldots \bdot g_{\ell} \bdot h^{-1} \bdot (hg^{-1}) \in \mathcal B (G) \,.
\]
Assume to the contrary that $U = U_1 \bdot U_2$ for some $U_1, U_2 \in \mathcal B (G)$ with $h^{-1} \t U_1$.
Since $S$ is product-one free, it follows that $hg^{-1} \t U_2$, so that $U_1 = h^{-1} \bdot g_1 \bdot \ldots \bdot g_k$ for some $k < \ell$, whence $h \in \pi (g_1 \bdot \ldots \bdot g_k) \subseteq H$, a contradiction to the choice of $h$.
Thus, $U$ is a minimal product-one sequence over $G$, and it follows that $\mathsf d (H) + 2 \le \mathsf D (G)$.
In particular, if $H$ is abelian, then $\mathsf D (H) = \mathsf d (H) + 1 < \mathsf d (H) + 2 \le \mathsf D (G)$.

2. See \cite[Proposition 3.9.1]{Cz-Do-Ge16}.

3. If $G$ is abelian, then $\mathsf d (G) + 1 = \mathsf D (G)$, and so it follows by 2. that
\[
  \mathsf D (G/N) = \mathsf d (G/N) + 1 \le \mathsf d (N) + \mathsf d (G/N) + 1 \le \mathsf d (G) + 1 = \mathsf D (G) \,.
\]
If $G$ is non-abelian, then for any sequence $T = g_1 N \bdot \ldots \bdot g_{\ell}N$ with $\ell = \mathsf D (G)$, the sequence $g_1 \bdot \ldots \bdot g_{\ell}$ over $G$ has a proper product-one subsequence by \cite[Lemma 4.2.1]{Oh19}, whence $T$ must have a proper product-one subsequence.
Thus, $\mathsf D (G/N) = \mathsf d (G/N) + 1 \le |T| = \mathsf D (G)$.
\end{proof}

\smallskip
In \cite{Cz-Do-Sz18}, Cziszter, Domokos, and Sz\"oll\H{o}si computed the small and large Davenport constants for all non-abelian groups of order less than 32 using computer program.
Recently, Andr\'as, together with them, has extended their calculation to all non-abelian groups of order at most 42 (see \cite[Section 8]{An-Cz-Do-Sz25}).
For the convenience of the reader, we provide the full table of the Davenport constants for all non-abelian groups of order at most 42.

\smallskip
\begin{lemma}[\cite{Cz-Do-Sz18,An-Cz-Do-Sz25}] \label{lem:dD}~
The following table represents the Davenport constants $\mathsf d (G)$ and $\mathsf D (G)$ for all non-abelian groups $G$ with $|G| \le 42$.

\begingroup
\DefTblrTemplate{caption}{default}{}    
\DefTblrTemplate{capcont}{default}{}    
\DefTblrTemplate{contfoot}{default}{}   
\begin{longtblr}{
  colspec = {|X[0.6, c]|X[1.8, c]|X[0.6, c]|X[0.6, c]||X[0.6, c]|X[1.8, c]|X[0.6, c]|X[0.6, c]|},
  rowhead = 1,
  hlines,
  row{even} = {gray!15},
}
\rule{0pt}{11pt} $GAP$ & $G$ & $\mathsf d (G)$ & $\mathsf D (G)$ & $GAP$ & $G$ & $\mathsf d (G)$ & $\mathsf D (G)$ \\
$(6,1)$   & $D_6$ & $3$ & $6$ &
    $(32,18)$  & $D_{32}$ & $16$ & $24$ \\
$(8,3)$   & $D_8$ & $4$ & $6$ &
    $(32,19)$  & $SD_{32}$ & $16$ & $24$ \\
$(8,4)$   & $Q_8$ & $4$ & $6$ &
    $(32,20)$  & $Q_{32}$ & $16$ & $24$ \\
$(10,1)$  & $D_{10}$ & $5$ & $10$ &
    $(32,22)$  & $C_2 \times (C^{2}_2 \rtimes C_4)$ & $6$ & $8$ \\
$(12,1)$  & $Q_{12}$ & $6$ & $9$ &
    $(32,23)$  & $C_2 \times (C_4 \rtimes C_4)$ & $7$ & $9$ \\
$(12,3)$  & $A_4$ & $4$ & $7$ &
    $(32,24)$  & $C^{2}_4 \rtimes C_2$ & $7$ & $9$ \\
$(12,4)$  & $D_{12}$ & $6$ & $9$ &
    $(32,25)$  & $C_4 \times D_8$ & $7$ & $9$ \\
$(14,1)$  & $D_{14}$ & $7$ & $14$ &
    $(32,26)$  & $C_4 \times Q_8$ & $7$ & $9$ \\
$(16,3)$  & $C^{2}_2 \rtimes C_4$ & $5$ & $7$ &
    $(32,27)$  & $C^{4}_2 \rtimes C_2$ & $6$ & $9$ \\
$(16,4)$  & $C_4 \rtimes C_4$ & $6$ & $8$ &
    $(32,28)$  & $(C^{2}_2 \times C_4) \rtimes C_2$ & $7$ & $10$ \\
$(16,6)$  & $M_{16}$ & $8$ & $10$ &
    $(32,29)$  & $(C_2 \times Q_8) \rtimes C_2$ & $7$ & $10$ \\
$(16,7)$  & $D_{16}$ & $8$ & $12$ &
    $(32,30)$  & $(C^{2}_2 \times C_4) \rtimes C_2$ & $7$ & $10$ \\
$(16,8)$  & $SD_{16}$ & $8$ & $12$ &
    $(32,31)$  & $C^{2}_4 \rtimes C_2$ & $7$ & $10$ \\
$(16,9)$  & $Q_{16}$ & $8$ & $12$ &
    $(32,32)$  & $C^{2}_2 \boldsymbol{.} C^{3}_2$ & $7$ & $10$ \\
$(16,11)$  & $C_2 \times D_8$ & $5$ & $7$ &
    $(32,33)$  & $C^{2}_4 \rtimes C_2$ & $7$ & $10$ \\
$(16,12)$  & $C_2 \times Q_8$ & $5$ & $7$ &
    $(32,34)$  & $C^{2}_4 \rtimes C_2$ & $7$ & $10$ \\
$(16,13)$  & $(C_2 \times C_4) \rtimes C_2$ & $5$ & $7$ &
    $(32,35)$  & $C_4 \rtimes Q_8$ & $7$ & $10$ \\
$(18,1)$  & $D_{18}$ & $9$ & $18$ &
    $(32,37)$  & $C_2 \times (C_8 \rtimes C_2)$ & $9$ & $11$ \\
$(18,3)$  & $C_3 \times D_6$ & $7$ & $10$ &
    $(32,38)$  & $(C_2 \times C_8) \rtimes C_2$ & $9$ & $11$ \\
$(18,4)$  & $C^{2}_3 \rtimes C_2$ & $5$ & $10$ &
    $(32,39)$  & $C_2 \times D_{16}$ & $9$ & $13$ \\
$(20,1)$  & $Q_{20}$ & $10$ & $15$ &
    $(32,40)$  & $C_2 \times SD_{16}$ & $9$ & $13$ \\
$(20,3)$  & $C_5 \rtimes C_4$ & $7$ & $10$ &
    $(32,41)$  & $C_2 \times Q_{16}$ & $9$ & $13$ \\
$(20,4)$  & $D_{20}$ & $10$ & $15$ &
    $(32,42)$  & $(C_2 \times C_8) \rtimes C_2$ & $9$ & $13$ \\
$(21,1)$  & $C_7 \rtimes C_3$ & $8$ & $14$ &
    $(32,43)$  & $C_8 \rtimes C^{2}_2$ & $9$ & $12$ \\
$(22,1)$  & $D_{22}$ & $11$ & $22$ &
    $(32,44)$  & $(C_2 \times Q_8) \rtimes C_2$ & $9$ & $12$ \\
$(24,1)$ & $C_3 \rtimes C_8$ & $12$ & $15$ &
    $(32,46)$  & $C^{2}_2 \times D_8$ & $6$ & $8$ \\
$(24,3)$ & $\SL (2,3)$ & $7$ & $13$ &
    $(32,47)$  & $C^{2}_2 \times Q_8$ & $6$ & $8$ \\
$(24,4)$ & $Q_{24}$ & $12$ & $18$ &
    $(32,48)$ & $C_2 \times \big( (C_2 \times C_4) \rtimes C_2 \big)$ & $6$ & $8$ \\
$(24,5)$ & $C_4 \times D_6$ & $12$ & $15$ &
    $(32,49)$ & $C^{3}_2 \rtimes C^{2}_2$ & $6$ & $8$ \\
$(24,6)$ & $D_{24}$ & $12$ & $18$ &
    $(32,50)$ & $(C_2 \times Q_8) \rtimes C_2$ & $6$ & $8$ \\
$(24,7)$ & $C_2 \times Q_{12}$ & $8$ & $11$ &
    $(34,1)$ & $D_{34}$ & $17$ & $34$ \\
$(24,8)$ & $C_3 \rtimes D_{8}$ & $7$ & $14$ &
    $(36,1)$ & $C_9 \rtimes C_4$ & $18$ & $27$ \\
$(24,10)$ & $C_3 \times D_{8}$ & $12$ & $14$ &
    $(36,3)$ & $C^{2}_2 \rtimes C_9$ & $10$ & $13$ \\
$(24,11)$ & $C_3 \times Q_8$ & $12$ & $14$ &
    $(36,4)$ & $D_{36}$ & $18$ & $27$ \\
$(24,12)$ & $S_4$ & $6$ & $12$ &
    $(36,6)$ & $C_3 \times (C_3 \rtimes C_4)$ & $13$ & $16$ \\
$(24,13)$ & $C_2 \times A_4$ & $7$ & $10$ &
    $(36,7)$ & $C^{2}_3 \rtimes C_4$ & $8$ & $13$ \\
$(24,14)$ & $C_2 \times D_{12}$ & $7$ & $10$ &
    $(36,9)$ & $C^{2}_3 \rtimes C_4$ & $7$ & $12$ \\
$(26,1)$ & $D_{26}$ & $13$ & $26$ &
    $(36,10)$ & $D^{2}_6$ & $8$ & $14$ \\
$(27,3)$ & $H_{27}$ & $6$ & $8$ &
    $(36,11)$ & $C_3 \times A_4$ & $8$ & $11$ \\
$(27,4)$ & $C_9 \rtimes C_3$ & $10$ & $12$ &
    $(36,12)$ & $C_6 \times D_6$ & $10$ & $13$ \\
$(28,1)$ & $Q_{28}$ & $14$ & $21$ &
    $(36,13)$ & $C_2 \times (C^{2}_3 \rtimes C_2)$ & $8$ & $13$ \\
$(28,3)$ & $D_{28}$ & $14$ & $21$ &
    $(38,1)$ & $D_{38}$ & $19$ & $38$ \\
$(30,1)$ & $C_5 \times D_6$ & $15$ & $18$ &
    $(39,1)$ & $C_{13} \rtimes C_3$ & $14$ & $26$ \\
$(30,2)$ & $C_3 \times D_{10}$ & $15$ & $20$ &
    $(40,1)$ & $C_5 \rtimes C_8$ & $20$ & $25$ \\
$(30,3)$ & $D_{30}$ & $15$ & $30$ &
    $(40,3)$ & $C_5 \rtimes C_8$ & $12$ & $15$ \\
$(32,2)$ & $(C_2 \times C_4) \rtimes C_4$ & $7$ & $9$ &
    $(40,4)$ & $C_5 \rtimes Q_8$ & $20$ & $30$ \\
$(32,4)$ & $C_8 \rtimes C_4$ & $10$ & $12$ &
    $(40,5)$ & $C_4 \times D_{10}$ & $20$ & $25$ \\
$(32,5)$ & $(C_2 \times C_8) \rtimes C_2$ & $9$ & $11$ &
    $(40,6)$ & $D_{40}$ & $20$ & $30$ \\
$(32,6)$ & $C^{3}_2 \rtimes C_4$ & $7$ & $10$ &
    $(40,7)$ & $C_2 \times (C_5 \rtimes C_4)$ & $12$ & $17$ \\
$(32,7)$ & $(C_8 \rtimes C_2) \rtimes C_2$ & $9$ & $12$ &
    $(40,8)$ & $(C_2 \times C_{10}) \rtimes C_2$ & $11$ & $22$ \\
$(32,8)$ & $C^{2}_2 \boldsymbol{.} (C_2 \times C_4)$ & $9$ & $12$ &
    $(40,10)$ & $C_5 \times D_8$ & $20$ & $22$ \\
$(32,9)$  & $(C_2 \times C_8) \rtimes C_2$ & $9$ & $13$ &
    $(40,11)$ & $C_5 \times Q_8$ & $20$ & $22$ \\
$(32,10)$  & $Q_8 \rtimes C_4$ & $9$ & $13$ &
    $(40,12)$ & $C_2 \times (C_5 \rtimes C_4)$ & $12$ & $15$ \\
$(32,11)$  & $C^{2}_4 \rtimes C_2$ & $9$ & $14$ &
    $(40,13)$ & $C^{2}_2 \times D_{10}$ & $11$ & $16$ \\
$(32,12)$  & $C_4 \rtimes C_8$ & $10$ & $12$ &
    $(42,1)$ & $C_7 \rtimes C_6$ & $11$ & $14$ \\
$(32,13)$  & $C_8 \rtimes C_4$ & $10$ & $14$ &
    $(42,2)$ & $C_2 \times (C_7 \rtimes C_3)$ & $15$ & $21$ \\
$(32,14)$  & $C_8 \rtimes C_4$ & $10$ & $14$ &
    $(42,3)$ & $C_7 \times D_6$ & $21$ & $24$ \\
$(32,15)$  & $C_4 \boldsymbol{.} D_8$ & $10$ & $14$ &
    $(42,4)$ & $C_3 \times D_{14}$ & $21$ & $28$ \\
$(32,17)$  & $C_{16} \rtimes C_2$ & $16$ & $18$ &
    $(42,5)$ & $D_{42}$ & $21$ & $42$ \\
\end{longtblr}
\endgroup

\end{lemma}

\smallskip
Based on Lemma~\ref{lem:dD}, we can obtain the complete list of non-abelian groups of order at most 42 having the Davenport constant less than 10.
Thus, the following series of lemmas deals with a lower bound for the Davenport constant of finite non-abelian groups of order exceeding 42.

We first present a lower bound for the Davenport constant of non-abelian 2-groups as follows.

\smallskip
\begin{lemma} \label{lem:2group}~
Let $G$ be a non-abelian $2$-group with $|G| = 2^{n}$ for $n \ge 5$.
\begin{enumerate}
\item If $n=5$, then either $\mathsf D (G) \ge 10$, or $\mathsf D (G) \in \{ 8, 9 \}$ and $G$ is isomorphic to one of the following groups.
      \begin{table}[h]
      \hspace{1.5cm}
      \begin{minipage}{.60\textwidth}
      \begin{tabular}{|c|c|c|} \hline
      \rule{0pt}{13pt} $\mathsf D (G)$ & $G$ & $GAP$                                                   \\[.5em] \hhline{|=|=|=|}
      \rule{0pt}{13pt} \multirow{6}{*}{$8$} & $C_2 \times (C^{2}_2 \rtimes C_4)$ & $(32,22)$ \\ & $C^{2}_2 \times D_8$ & $(32,46)$ \\ & $C^{2}_2 \times Q_8$ & $(32,47)$ \\ & $C_2 \times \big((C_2 \times C_4) \rtimes C_2\big)$ & $(32,48)$ \\ & $C^{3}_2 \rtimes C^{2}_2$ & $(32,49)$ \\ & $(C_2 \times Q_8) \rtimes C_2$ & $(32,50)$                         \\[.5em] \hline
      \end{tabular}
      \end{minipage} \hspace{-2.5cm}
      \begin{minipage}{.40\textwidth}
      \begin{tabular}{|c|c|c|} \hline
      \rule{0pt}{13pt} $\mathsf D (G)$ & $G$ & $GAP$                                                  \\[.5em] \hhline{|=|=|=|}
      \rule{0pt}{13pt} \multirow{6}{*}{$9$} & $(C_2 \times C_4) \rtimes C_4$ & $(32,2)$ \\ & $C_2 \times (C_4 \rtimes C_4)$ & $(32,23)$ \\ & $C^{2}_4 \rtimes C_2$ & $(32,24)$ \\ & $C_4 \times D_8$ & $(32,25)$ \\ & $C_4 \times Q_8$ & $(32,26)$ \\ & $C^{4}_2 \rtimes C_2$ & $(32,27)$  \\[.5em] \hline
      \end{tabular}
      \end{minipage}
      \end{table}

\item If $n \ge 6$, then we obtain $\mathsf D (G) \ge 8$.
\end{enumerate}
\end{lemma}

\begin{proof}
1. See the table in Lemma~\ref{lem:dD}.

2. Since $\mathsf D (G) \ge \mathsf D (H)$ for any subgroup $H$, it suffices to show that $\mathsf D (G) \ge 8$ for the case where $n = 6$.
Since $G$ is a 2-group, we infer that $G$ has a normal subgroup $H$ with $|H| =  32$.
If $H$ is non-abelian, then item 1. ensures that $\mathsf D (G) \ge \mathsf D (H) \ge 8$.
Then, we may assume that $H$ is abelian, so that $H$ is isomorphic to one of the following groups:
\begin{equation} \label{eq:32}~
  C_{32} \,, \quad C_4 \times C_8 \,, \quad C_2 \times C_{16} \,, \quad C_2 \times C^{2}_4 \,, \quad C^{2}_2 \times C_8 \,, \quad C^{3}_2 \times C_4 \,, \quad C^{5}_2 \,.
\end{equation}
By Lemma~\ref{lem:abD}.(a), we obtain that $\mathsf D (H) = 6$ if $H \cong C^{5}_2$, and $\mathsf D (H) \ge 7$ if otherwise.
Thus, Lemma~\ref{lem:ineqDC}.1 ensures that $\mathsf D (G) \ge \mathsf D (H) + 1 \ge 8$ if $H \ncong C^{5}_2$.

Suppose now that $H \cong C^{5}_2$.
Note that $H \subset G$. If every element in $G \setminus H$ is of order 2, then all elements of $G$ are of order 2, whence $G$ is abelian, a contradiction.
Thus, there exists at least one element $f \in G \setminus H$ with $\ord (f) = 2^{i}$ for some $i \ge 2$, and we set $K = \langle f \rangle$.
If $i \ge 4$, then $K$ is cyclic of order $|K| \ge 16$, and thus $\mathsf D (G) \ge \mathsf D (K) \ge 16$.
If $i = 3$, then $K$ is cyclic of order 8.
Since $H$ is normal in $G$, it follows that $HK$ is a subgroup of $G$.
Since $64 = |G| \ge |HK| = \frac{|H||K|}{|H \cap K|}$, we infer that $H \cap K$ is a subgroup of both $H$ and $K$ of order at least 4, i.e., $H \cap K$ is a cyclic subgroup of $H$ of order at least 4, a contradiction. 
Thus, we can assume that $i = 2$.
Then, we must have that $|H \cap K| = 2$, and so $G = HK$ with $H \cap K = \langle f^{2} \rangle$.
Let $N \cong C^{4}_2 \cong \langle a \rangle \times \langle b \rangle \times \langle c \rangle \times \langle d \rangle$, and $H \cong N \times C_2 \cong N \times \langle f^{2} \rangle$.
Since $H$ is normal in $G$, conjugation of $H$ by $f$ induces an automorphism of $H$, which maps $f^{2}$ to itself, and thus $N$ is mapped to itself under this conjugation.
Since $G = HK$, we infer that $N$ is a normal subgroup of $G$, and $N$ intersects trivially with $K$, whence $G \cong N \rtimes K$.
Let $\varphi \colon K \to \Aut (N)$ be a homomorphism induced from the action of $K$ on $N$ by conjugation.
Note that $f^{2}$ commutes with all elements in $N$, so that $\varphi (f)^{2} = \varphi (f^{2})$ is the identity map on $N$.
Since $G$ is non-abelian, it follows that $\varphi [K] \le \Aut (N)$ is a non-trivial subgroup of order 2.
Note that $\Aut (H) \cong \GL (4,2)$ (see \cite[Proposition 4.17]{Du-Fo04}) and $|\GL (4,2)| = \prod_{i=0}^{3}(2^{4} - 2^{i}) = 2^{6} 3^{2} 5^{1} 7^{1}$.
Let $A \in \GL (4,2)$ be such that $A^{2} = I$.
Then, the minimal polynomial of $A$ divides $x^{2} - 1$, and $x^{2} - 1 = (x+1)^{2} \in \mathbb Z_2 [x]$.
Under the constraints of the rational canonical form, we obtain the following permissible lists of invariant factors:
\begin{enumerate}
\item[(i)] $(x+1)$, $(x+1)$, $(x+1)$, $(x+1)$

\item[(ii)] $(x+1)$, $(x+1)$, $(x+1)^{2}$

\item[(iii)] $(x+1)^{2}$, $(x+1)^{2}$
\end{enumerate}
Then, (i) corresponds to the identity matrix in $\GL (4,2)$, and (ii) (resp., (iii)) corresponds to the matrix $B = \begin{bmatrix} 1 & & & \\ & 1 & & \\ & & 0 & 1 \\ & & 1 & 0 \end{bmatrix}$ $\Bigg($ resp., $C = \begin{bmatrix}  0 & 1 & & \\ 1 & 0 & & \\ & & 0 & 1 \\ & & 1 & 0 \end{bmatrix} \Bigg)$.
Recall that $N \cong \langle a \rangle \times \langle b \rangle \times \langle c \rangle \times \langle d \rangle$, $K \cong \langle f \rangle$, and $\varphi [K] \le \Aut (N)$ is a subgroup of order 2, so that $\varphi [K]$ is either $\langle B \rangle$, or $\langle C \rangle$.

If $\varphi [K] = \langle B \rangle$, then $G \cong \langle a, b, c, d, f \mid a^{2} = b^{2} = c^{2} = d^{2} = f^{4} = 1_G, \mbox{ and } ab = ba, ac = ca, ad = da, bc = cb, bd = db, cd = dc, fa = af, fb = bf, fcf^{-1} = d, fdf^{-1} = c \rangle$.
Let $x = cd \in G$.
Then, it is easy to see that $\ord (x) = 2$ and $fx = xf$.
This means that $G$ has an abelian subgroup $L \cong \langle a, b, x, f \rangle \cong C^{3}_2 \times C_4$ of order 32, and it follows by Lemmas~\ref{lem:ineqDC}.1 and \ref{lem:abD}.(a) that $\mathsf D (G) \ge \mathsf D (L) + 1 = 8$.

If $\varphi [K] = \langle C \rangle$, then $G \cong \langle a, b, c, d, f \mid a^{2} = b^{2} = c^{2} = d^{2} = f^{4} = 1_G, \mbox{ and } ab = ba, ac = ca, ad = da, bc = cb, bd = db, cd = dc, faf^{-1} = b, fbf^{-1} = a, fcf^{-1} = d, fdf^{-1} = c \rangle$.
Let $x = ab$ and $y = cd \in G$.
Then, it is easy to see that $\ord (x) = 2$, $\ord (y) = 2$, $fx = xf$, and $fy = yf$.
Thus, $G$ has a subgroup $L \cong \langle x, a, y, f \mid x^{2} = a^{2} = y^{2} = f^{4} = 1_G, \mbox{ and } xa = ax, xy = yx, xf = fx, ay = ya, faf^{-1} = ax, fy = yf \rangle \cong C_2 \times (C^{2}_2 \rtimes C_4)$ of order 32, and hence it follows by item 1. that $\mathsf D (G) \ge \mathsf D (L) = 8$.
\end{proof}

\smallskip
The following results present a lower bound for the Davenport constants of finite groups that have a specific subgroup.

\smallskip
\begin{lemma} \label{lem:order9}~
If a finite group $G$ has an element $g$ such that either $\ord (g) \ge 10$ or $\ord (g) = 9$ with $\langle g \rangle \subsetneq G$, then $\mathsf D (G) \ge 10$.
\end{lemma}

\begin{proof}
Let $g \in G$ be an element.
If $\ord (g) \ge 10$, then $g^{[\ord (g)]}$ is a minimal product-one sequence, and so $\mathsf D (G) \ge \ord (g) \ge 10$.
If $\ord (g) = 9$ and $\langle g \rangle$ is a proper subgroup of $G$, then we can take $h \in G \setminus \langle g \rangle$, and thus $g^{[8]} \bdot h$ is a product-one free sequence, implying that $\mathsf D (G) \ge \mathsf d (G) + 1 \ge 10$.
\end{proof}

\smallskip
\begin{lemma} \label{lem:35}~
If $G$ is a group with $|G| \in \{ 70, 105, 140 \}$, then $G$ has an element of order $35$.
\end{lemma}

\begin{proof}
Since every group of order 35 is cyclic, it suffices to show that $G$ has a subgroup of order 35.
We first suppose that $|G| = 70 = 2^{1} 5^{1} 7^{1}$ or $|G| = 140 = 2^{2} 5^{1} 7^{1}$.
Then, by Sylow's Theorem, it is easy to see that $G$ has a unique Sylow 5-subgroup $H$, and a Sylow 7-subgroup $K$.
Since $H$ intersects trivially with $K$, $HK$ is a subgroup of order 35.

If $|G| = 105 = 3^{1} 5^{1} 7^{1}$, then again by Sylow's Theorem, we infer that either a Sylow 5-subgroup or a Sylow 7-subgroup of $G$ is unique.
In either case, it follows by the same argument as above that $G$ must have a subgroup of order 35.
\end{proof}

\smallskip
\begin{lemma} \label{lem:subgroup}~
If a finite group $G$ has either
\begin{enumerate}
\item[(a)] a normal subgroup $N$ with $G/N$ abelian and $|G/N| \in \mathcal A \setminus \{ 25 \}$, or

\smallskip
\item[(b)] a proper subgroup $H$ of order $|H| \in \mathcal A$,
\end{enumerate}
where $\mathcal A = \{ 10, 14, 15, 20, 21, 25, 28, 30, 35, 36, 40, 42, 45, 70, 105, 140 \}$, then $\mathsf D (G) \ge 10$.
\end{lemma}

\begin{proof}
(a) Let $N$ be a normal subgroup of $G$ such that $G/N$ is abelian.
If $|G/N| \in \{ 10, 14, 15, 21, 30, 35, 42 \}$, then $G/N$ must be a cyclic group, and so $\mathsf D (G/N) = |G/N| \ge 10$.
If $|G/N| \in \{ 70, 105, 140 \}$, then Lemma~\ref{lem:35} ensures that $G/N$ has an element of order 35.
Moreover, if $|G/N| \in \{ 20, 28, 40, 45 \}$, then it is easy to see that $G/N$ has an element of order at least 10.
Thus, it follows by Lemma~\ref{lem:order9} that $\mathsf D (G/N) \ge 10$.
If $|G/N| = 36$, then $G/N \cong C_{36}$, or $C_2 \times C_{18}$, or $C_3 \times C_{12}$, or $C^{2}_6$.
If $G/N \cong C^{2}_6$, then Lemma~\ref{lem:abD}.(b) implies that $\mathsf D (G/N) = 11$.
If $G/N \ncong C^{2}_6$, then $G/N$ must have an element of order at least 10, and so Lemma~\ref{lem:order9} ensures that $\mathsf D (G/N) \ge 10$.
In either case, it follows by Lemma~\ref{lem:ineqDC}.3 that $\mathsf D (G) \ge \mathsf D (G/N) \ge 10$.

(b) Let $H$ be a proper subgroup of $G$. Note that $\mathsf D (G) \ge \mathsf D (H)$.
Suppose first that $|H| = 25$.
If $H \cong C_{25}$, then Lemma~\ref{lem:order9} ensures that $\mathsf D (G) \ge \mathsf D (H) \ge 10$.
If $H \cong C^{2}_5$, then Lemma~\ref{lem:abD}.(a) ensures that $\mathsf D (H) = 9$, whence it follows by Lemma~\ref{lem:ineqDC}.1 that $\mathsf D (G) \ge \mathsf D (H) + 1 = 10$.
If $|H| \in \{ 70, 105, 140 \}$, then Lemma~\ref{lem:35} implies that $\mathsf D (H) \ge 35$.
Suppose that $|H| \in \mathcal A \setminus \{ 25, 70, 105, 140 \}$.
If $H$ is abelian, then by replacing $G/N$ with $H$ in (a), we directly obtain that $\mathsf D (G) \ge \mathsf D (H) \ge 10$.
If $H$ is non-abelian, then it follows by Lemma~\ref{lem:dD} that $\mathsf D (G) \ge \mathsf D (H) \ge 10$.
\end{proof}

\smallskip
We are now ready to state the key lemma of this paper, and its proof relies heavily on the classical method for classifying finite groups and finite perfect groups.
Recall that a group is said to be {\it perfect} if it equal to its commutator subgroup.
In \cite{Sa81}, the author provides a complete list of perfect groups of order less than $10^{4}$, along with their structural information.
For the convenience of the reader, we provide the list of them describing some characteristic features of each of these groups.

\smallskip
\begin{lemma}[\cite{Sa81}] \label{lem:perfect}~
Let $G$ be a finite group of order less than $10^{4}$.
Then, $G$ is perfect if and only if $G$ has a proper normal subgroup $N$ such that $G/N$ is a non-abelian simple group, as described below:

\begingroup
\DefTblrTemplate{caption}{default}{}    
\DefTblrTemplate{capcont}{default}{}    
\DefTblrTemplate{contfoot}{default}{}   
\begin{longtblr}{
  colspec = {|X[1.8, c]|X[1.2, c]|X[0.6, c]||X[1.6, c]|X[1.2, c]|X[0.6, c]|},
  rowhead = 1,
  hlines,
  row{even} = {gray!15},
}
\rule{0pt}{11pt} $N$ & $G/N$ & $|G|$ & $N$ & $G/N$ & $|G|$ \\
$\langle 1_G \rangle$   & $A_5$ & $60$ &
    $C^{4}_2 \times C_4$ & $A_5$ & $3840$ \\
$C_2$   & $A_5$ & $120$ &
    $C_2 \times \big( (C_2 \times Q_8) \rtimes C_2 \big)$ & $A_5$ & $3840$ \\
$\langle 1_G \rangle$   & $\PSL (2,7)$ & $168$ &
    $P_1$ & $A_5$ & $3840$ \\
$C_2$  & $\PSL (2,7)$ & $336$ &
    $\langle 1_G \rangle$  & $\PSL (2,16)$ & $4080$ \\
$\langle 1_G \rangle$  & $A_6$ & $360$ &
    $C^{4}_3$ & $A_5$ & $4860$ \\
$\langle 1_G \rangle$  & $\PSL (2,8)$ & $504$ &
    $C_2$  & $\PSL (2,17)$ & $4896$ \\
$\langle 1_G \rangle$  & $\PSL (2,11)$ & $660$ &
    $C_2$  & $A_7$ & $5040$ \\
$C_2$  & $A_6$ & $720$ &
    $C^{5}_2$ & $\PSL (2,7)$ & $5376$ \\
$C^{4}_2$ & $A_5$ & $960$ &
    $\langle 1_G \rangle$  & $\PSL (3,3)$ & $5616$ \\
$C_3$  & $A_6$ & $1080$ &
    $C^{4}_2$ & $A_6$ & $5760$ \\
$\langle 1_G \rangle$  & $\PSL (2,13)$ & $1092$ &
    $\langle 1_G \rangle$  & $\PSU (3,9)$ & $6048$ \\
$C_2$  & $\PSL (2,11)$ & $1320$ &
    $\langle 1_G \rangle$  & $\PSL (2,23)$ & $6072$ \\
$C^{3}_2$ & $\PSL (2,7)$ & $1344$ &
    $C_2$  & $\PSL (2,19)$ & $6840$ \\
$C^{5}_2$ & $A_5$ & $1920$ &
    $\SL (2,5)$ & $A_5$ & $7200$ \\
$(C_2 \times Q_8) \rtimes C_2$ & $A_5$ & $1920$ &
    $C^{3}_5$ & $A_5$ & $7500$ \\
$C_6$  & $A_6$ & $2160$ &
    $C_3$  & $A_7$ & $7560$ \\
$C_2$  & $\PSL (2,13)$ & $2184$ &
    $C^{7}_2$ & $A_5$ & $7680$ \\
$\langle 1_G \rangle$  & $\PSL (2,17)$ & $2448$ &
    $C_2 \times P_1$ & $A_5$ & $7680$ \\
$\langle 1_G \rangle$  & $A_7$ & $2520$ &
    $\langle 1_G \rangle$  & $\PSL (2,25)$ & $7800$ \\
$C^{4}_2$ & $\PSL (2,7)$ & $2688$ &
    $\langle 1_G \rangle$ & $\M_{11}$ & $7920$ \\
$C^{2}_5 \rtimes C_2$ & $A_5$ & $3000$ &
    $C^{3}_3 \times C_6$ & $A_5$ & $9720$ \\
$\langle 1_G \rangle$  & $\PSL (2,19)$ & $3420$ &
    $C^{4}_3 \rtimes C_2$ & $A_5$ & $9720$ \\
$A_5$ & $A_5$ & $3600$ &
    $\langle 1_G \rangle$ & $\PSL (2,27)$ & $9828$ \\
$C^{6}_2$ & $A_5$ & $3840$ &
    \\
\end{longtblr}
\endgroup
\noindent
where $\PSU (n,q)$ is a projective special unitary group over a finite field of order $q$, $\M_{11}$ is a Mathieu group, and $P_1 = \langle a, b, c, d, e \mid a^{2} = b^{2} = c^{2} = d^{2} = e^{4} = 1_G, \textnormal{ and } ab = ba, ac = ca, cd = dc, ae = ea, be = eb, ce = ec, de = ed, dad = e^{2}d, bcb = e^{2}c, dbd = e^{2}b \rangle$.
\end{lemma}

\smallskip
Based on this information, if our group is perfect, then we find a specific subgroup with a sufficiently large Davenport constant.
Otherwise, we construct the group using the classical method for the semidirect product, and we present concrete group generators and their relations.
We then find either a subgroup as in the perfect case, or a minimal product-one sequence of sufficiently large length.
Recall by Lemma~\ref{lem:dD} that $\mathsf D (G) \ge 10$ for a non-abelian group $G$, that is not isomorphic to the groups listed in the tables of the main Theorem~\ref{thm:classify}, with $|G| \le 42$.

\smallskip
\begin{lemma} \label{lem:DC}~
Let $G$ is a finite non-abelian group with $|G| > 42$.
\begin{enumerate}
\item If $G$ has a proper subgroup of order $32$, then $\mathsf D (G) \ge 8$.

\smallskip
\item If $G$ has no proper subgroup of order $32$, then $\mathsf D (G) \ge 10$.
\end{enumerate}
\end{lemma}

\begin{proof}
Let $G$ be a finite non-abelian group with $|G| > 42$.

1. Suppose that $G$ has a proper subgroup $H$ with $|H| = 32$.
If $H$ is a proper subgroup of a Sylow $2$-subgroup $P$ of $G$, then $P$ has the order at least 64, and we may assume that $|P| = 64$.
If $P$ is non-abelian, then Lemma~\ref{lem:2group}.2 ensures that $\mathsf D (G) \ge \mathsf D (P) \ge 8$.
If $P$ is abelian, then by Lemma~\ref{lem:abD}.(a) together with the structure theorem for finite abelian group, we obtain that $\mathsf D (G) \ge \mathsf D (P) \ge 8$ if $P \ncong C^{6}_2$.
If $P \cong C^{6}_2$, then $\mathsf D (G) \ge \mathsf D (P) + 1 = 8$ by Lemma~\ref{lem:ineqDC}.1.

Suppose now that $H$ is a Sylow $2$-subgroup of $G$.
If $H$ is non-abelian, then Lemma~\ref{lem:2group}.1 ensures that $\mathsf D (G) \ge \mathsf D (H) \ge 8$.
Suppose that $H$ is abelain, so that $H$ is isomorphic to one of the groups listed in (\ref{eq:32}).
If $H$ is not isomorphic to $\{ C^{3}_2 \times C_4, C^{5}_2 \}$, then by Lemma~\ref{lem:abD}.(a), we obtain that $\mathsf D (G) \ge \mathsf D (H) \ge 8$.
If $H \cong C^{3}_2 \times C_4$, then by Lemma~\ref{lem:ineqDC}.1, we obtain that $\mathsf D (G) \ge \mathsf D (H) + 1 = 8$.
If $H \cong C^{5}_2$, then  since $H$ is proper, there exists an element $g \in G \setminus H$.
Since $H$ is a Sylow $2$-subgroup of $G$, it follows that $\ord (g)$ is not a power of 2, i.e., $\ord (g) \ge 3$.
Let $\{ e_1, e_2, e_3, e_4, e_5 \}$ be a minimal generating set of $H$.
Then, it is easy to see that the sequence
\[
  S = e_1 \bdot e_2 \bdot e_3 \bdot e_4 \bdot e_5 \bdot (e_1 e_2 e_3 e_4 e_5 g) \bdot g^{[\ord (g) - 1]} \in \mathcal F (G)
\]
is a minimal product-one sequence over $G$, and thus $\mathsf D (G) \ge |S| = 6 + (\ord (g) - 1) \ge 8$.

\medskip
2. Suppose that $G$ has no proper subgroup of order 32.
If $G$ has an element $g$ with $\ord (g) \ge 9$, then Lemma~\ref{lem:order9} ensures that $\mathsf D (G) \ge 10$.
Thus, we now suppose that
\begin{equation} \label{eq:order}~
  \ord (g) \in [2,8] \quad \mbox{ for all } \,\, g \in G \setminus \{ 1_G \} \,,
\end{equation}
and since $G$ is non-abelian, there must be at least one element of order greater than 2.

We first handle the case when $G$ is a $p$-group for a prime $p$.
Then, $G$ has a normal subgroup of each order dividing $|G|$, and since $G$ is not a 2-group, we infer that $p \in \{ 3, 5, 7 \}$.

If $p = 3$, then $|G| = 3^{n}$ with $n \ge 4$, and it follows by (\ref{eq:order}) that $\exp (G) = \lcm \{ \ord (g) \mid g \in G \} = 3$. Since $\mathsf D (G) \ge \mathsf D (H)$ for any subgroup $H$, it suffices to show that $\mathsf D (G) \ge 10$ for the case where $n = 4$.
Now, let $|G| = 81$ and $K \unlhd G$ with $|K| = 9$.
Since $K$ is abelian, we must have that $K \cong C_3 \times C_3$.
Then, $G/K$ acts on $K$ by conjugation via $gK \bdot x := gxg^{-1}$, and this action induces a group homomorphsm $\varphi \colon G/K \to \Aut (K)$ given by $gK \mapsto \tau_{gK}$, where $\tau_{gK} \colon K \to K$ is an automorphism defined by $\tau_{gK} (x) = gK \bdot x = gxg^{-1}$.
Hence, $\ker (\varphi) = \{ gK \mid gxg^{-1} = x$ for all $x \in K \} = \{ gK \mid g \in C_{G} (K) \} = C_G (K)/K$.
Since $\Aut (K) \cong \GL (2,3)$ and $|\GL (2,3)| = 2^{4} 3$, it follows that $\ker (\varphi)$ must be non-trivial, whence $K$ must be a proper subgroup of $C_G (K)$.
This means that there exists an element $a \in C_G (K) \setminus K$, and thus $N := K \langle a \rangle$ is an abelian normal subgroup of $G$ of order 27.
By (\ref{eq:order}), we infer that $N \cong C^{3}_3$.
Let $\alpha \in G \setminus N$.
Then, $H = \langle \alpha \rangle$ is a subgroup of order 3, and it follows that $H \cap N = \{ 1_G \}$, whence $G \cong N \rtimes H$.
Note that $\Aut (N) \cong \GL (3,3)$ (see \cite[Proposition 4.17]{Du-Fo04}) and $|\GL (3,3)| = 2^{5}3^{3}13$.
Let $A \in \GL (3,3)$ be such that $A^{3} = I$.
Then, the minimal polynomial of $A$ divides $x^{3}-1$, and $x^{3}-1 = (x-1)^{3} \in \mathbb Z_3 [x]$.
Under the constraints of the rational canonical form, we obtain the following permissible lists of invariant factors:
\begin{enumerate}
\item[(i)] $(x+2)$, $(x+2)$, $(x+2)$

\item[(ii)] $(x+2)$, $(x+2)^{2}$

\item[(iii)] $(x+2)^{3}$
\end{enumerate}
Then, (i) corresponds to the identity matrix in $\GL (3,3)$, and (ii) (resp., (iii)) corresponds to the matrix $B = \begin{bmatrix} 1 & & \\ & 0 & 2 \\ & 1 & 2 \end{bmatrix} \Bigg($ resp., $C = \begin{bmatrix} 0 & 0 & 1 \\ 1 & 0 & 0 \\ 0 & 1 & 0 \end{bmatrix} \Bigg)$.
Let $N \cong C^{3}_3 \cong \langle a \rangle \times \langle b \rangle \times \langle c \rangle$, $H \cong C_3 \cong \langle d \rangle$, and $\varphi \colon H \to \Aut (N)$ be a homomorphism induced from the action of $H$ on $N$ by conjugation.
Since $G$ is non-abelian, it follows that $\varphi [H] \le \Aut (N)$ is a non-trivial subgroup of order 3.

If $\varphi [H] = \langle C \rangle$, then we infer that $G \cong \langle a, b, c, d \mid a^{3} = b^{3} = c^{3} = d^{3} = 1_G, \mbox{ and } ab = ba, ac = ca, bc = cb, dad^{-1} = c, dbd^{-1} = a, dcd^{-1} = b \rangle$.
Then, it is easy to see that $\ord (cd) = 9$, a contradiction to (\ref{eq:order}).

If $\varphi [H] = \langle B \rangle$, then we infer that $G \cong \langle a, b, c, d \mid a^{3} = b^{3} = c^{3} = d^{3} = 1_G, \mbox{ and } ab = ba, ac = ca, ad = da, bc = cb, dbd^{-1} = c^{-1}, dcd^{-1} = bc^{-1} \rangle$.
Let $x = b^{-1}c^{-1}$. Then $\ord (x) = 3$, $xc = cx$, $xd = dx$, and $dcd^{-1} = cx^{-1}$.
Thus, we obtain that $H_{27} \cong \langle x, c, d \rangle$, and since $a \in Z(G)$, it follows that $G \cong \langle a \rangle \times \langle x, c, d \rangle \cong C_3 \times H_{27}$.
Now, we consider the sequence
\[
  S = (ac) \bdot c^{[2]} \bdot d^{[3]} \bdot x^{[2]} \bdot a^{[2]} \in \mathcal F (G) \,.
\]
Note that $1_G = a a (ac) d c c d d x x$, so that $S$ is a product-one sequence over $G$ of length 10.
To show that $S$ is minimal, we assume to the contrary that $S = S_1 \bdot S_2$ for some non-trivial $S_1, S_2 \in \mathcal B (G)$ with $a \mid S_1$.
Since $a$ is independent with $c, d, x$ in $G$ and $\ord (a) = 3$, we infer that $(ac) \bdot a^{[2]} \mid S_1$, so that $S_2 \mid c^{[2]} \bdot d^{[3]} \bdot x^{[2]}$.
Since $a \in Z (G)$, it follows that $T_1 := \big( S_1 \bdot a^{[-2]} \bdot (ac)^{[-1]} \big) \bdot c$ must be a product-one sequence, and hence we obtain that $c^{[3]} \bdot d^{[3]} \bdot x^{[2]} = T_1 \bdot S_2$, contradicting that $c^{[3]} \bdot d^{[3]} \bdot x^{[2]}$ is a minimal product-one sequence over $H_{27}$ (see the proof of \cite[Proposition 5.5]{Cz-Do-Sz18}).
Therefore, $S$ must be a minimal product-one sequence, whence $\mathsf D (G) \ge |S| = 10$.

If $p = 5$, then $|G| = 5^{n}$ with $n \ge 3$, and so $G$ has a proper subgroup of order 25.
Thus, it follows by Lemma~\ref{lem:subgroup} that $\mathsf D (G) \ge 10$.

If $p = 7$, then $|G| = 7^{n}$ with $n \ge 2$, and the same argument as used before ensures that it suffices to show that $\mathsf D (G) \ge 10$ for the case where $n = 2$.
Then, $G$ must be abelian, and it follows by (\ref{eq:order}) that $G \cong C^{2}_7$, whence $\mathsf D (G) = 13$ by Lemma~\ref{lem:abD}.(a).

Hence, we may assume that $G$ is not a $p$-group, so that $G$ is not a prime power order and $|G| = 2^{i} 3^{j} 5^{k} 7^{\ell}$ for some $i, j, k, \ell \in \mathbb N_0$.
If $j \ge 4$ (resp, $k \ge 3$, or $\ell \ge 2$), then $G$ has a proper subgroup of order $3^{4}$ (resp, $5^{3}$, or $7^{2}$), and hence $\mathsf D (G) \ge 10$ as described above.
Moreover, if $k = 2$, then since $|G| > 42$, there exists at least one positive index among $i, j, \ell$.
By Sylow's Theorem, $G$ has a proper Sylow 5-subgroup of order 25, and so Lemma~\ref{lem:subgroup} ensures that $\mathsf D (G) \ge 10$.
Therefore, since $G$ has no proper subgroup of order 32, we suppose now that
\[
  |G| = 2^{i} 3^{j} 5^{k} 7^{\ell} \quad \mbox{ for } \,\, i \in [0,4], \,\, j \in [0,3], \,\, k \in [0,1], \,\, \ell \in [0,1] \,.
\]
We distinguish 15 cases depending on $i$, $j$, $k$, and $\ell$.

\smallskip
\noindent
{\bf CASE 1.} $j = 1$ and $k = \ell = 0$, i.e., $|G| = 2^{i} 3^{1}$.
\smallskip

Then, since $|G| > 42$, we only have the case where $i = 4$, and so $|G| = 2^{4} 3^{1} = 48$.
It is an easy consequence of Sylow's Theorem that $G$ has a non-trivial normal subgroup $H$ of order 8 or 16.
Let $K$ be a Sylow $3$-subgroup of $G$.

Suppose that $|H| = 8$.
Since $|K| = 3$, it follows that $H$ intersects trivially with $K$.
Thus, $HK \le G$ is a subgroup of order $|HK| = \frac{|H| |K|}{|H \cap K|} = 24$, and since $(G : HK) = 2$, we infer that $HK$ is a normal subgroup of $G$.
If $HK$ is non-abelian, then by Lemma~\ref{lem:dD}, $\mathsf D (G) \ge \mathsf D (HK) \ge 10$.
If $HK$ is abelian, then in view of (\ref{eq:order}), we must have that $HK \cong C^{3}_2 \times C_3$.
Then, there exists an element $e \in G \setminus HK$ with $\ord (e) = 2$, and so it is clear that $HK$ intersects trivially with $\langle e \rangle$, whence $G \cong HK \rtimes \langle e \rangle$.
Note that $\Aut (HK) \cong \Aut (C^{3}_2) \times \Aut (C_3) \cong \GL (3,2) \times \Aut (C_3)$, so that $|\Aut (HK)| = 2^{5} 3^{1} 7^{1}$.
Let $A \in \GL (3,2)$ be such that $A^{2} = I$.
Then, the minimal polynomial of $A$ divides $x^{2} - 1$, and $x^{2} - 1 = (x-1)^{2} \in \mathbb Z_2 [x]$.
Under the constraints of the rational canonical form, we obtain the following permissible lists of invariant factors:
\begin{enumerate}
\item[(i)] $(x-1)$, $(x-1)$, $(x-1)$

\smallskip
\item[(ii)] $(x-1)$, $(x-1)^{2}$
\end{enumerate}
Then, (i) corresponds to the identity matrix in $\GL (3,2)$, and (ii) corresponds to the matrix $B = \begin{bmatrix} 1 & & \\ & 0 & 1 \\ & 1 & 0 \end{bmatrix}$.
Let $HK \cong C^{3}_2 \times C_3 \cong (\langle a \rangle \times \langle b \rangle \times \langle c \rangle) \times \langle d \rangle$, and $\varphi \colon \langle e \rangle \to \Aut (HK)$ be a homomorphism induced from the action of $\langle e \rangle$ on $HK$ by conjugation.
Since $G$ is non-abelian, it follows that $\varphi [\langle e \rangle] \le \Aut (HK)$ is a non-trivial subgroup of order 2.

If $\varphi [\langle e \rangle] \cong \{ I \} \times \Aut (C_3)$, then $G \cong \langle a, b, c, d, e \mid a^{2} = b^{2} = c^{2} = d^{3} = e^{2} = 1_G, \mbox{ and } ab = ba, ac = ca, ad = da, bc = cb, bd = db, cd = dc, ea = ae, eb = be, ec = ce, ede = d^{2} \rangle$.
Let $x = cd \in G$. Then, $\ord (x) = 6$ and $exe = x^{-1}$, whence $L \cong \langle b, x, e \mid b^{2} = x^{6} = e^{2} = 1_G, \mbox{ and } bx = xb, be = eb, exe = x^{-1} \rangle \cong C_2 \times D_{12}$ is a subgroup of $G$ of order 24.
It follows by Lemma~\ref{lem:dD} that $\mathsf D (G) \ge \mathsf D (L) = 10$.

If $\varphi [\langle e \rangle] \cong \langle B \rangle \times \{ \id \}$, where $\id$ is the identity automorphism of $C_3$, then $G \cong \langle a, b, c, d, e \mid a^{2} = b^{2} = c^{2} = d^{3} = e^{2} = 1_G, \mbox{ and } ab = ba, ac = ca, ad = da, bc = cb, bd = db, cd = dc, ea = ae, ebe = c, ece = b, ed = de \rangle$.
Then, it is easy to see that $\ord (cde) = 12$, a contradiction to (\ref{eq:order}).

If $\varphi [\langle e \rangle] \cong \langle (B, f) \rangle$, where $f$ is an automorphism of $C_3$ given by $f (d) = d^{2}$, then $G \cong \langle a, b, c, d, e \mid a^{2} = b^{2} = c^{2} = d^{3} = e^{2} = 1_G, \mbox{ and } ab = ba, ac = ca, ad = da, bc = cb, bd = db, cd = dc, ea = ae, ebe = c, ece = b, ede = d^{2} \rangle$.
Let $x = bc$ and $y = ad$. Then, $\ord (x) = 2$, $\ord (y) = 6$, $xy = yx$, $xe = ex$, and $eye = y^{-1}$, whence $L \cong \langle x, y, e \rangle \cong C_2 \times D_{12}$ is a subgroup of $G$ of order 24.
It follows by Lemma~\ref{lem:dD} that $\mathsf D (G) \ge \mathsf D (L) = 10$.

Suppose now that $|H| = 16$.
Since $H$ intersects trivially with $K$, it follows that $G \cong H \rtimes K$.
If $H \cong C_2 \times C_8$, then Lemmas~\ref{lem:ineqDC}.1 and \ref{lem:abD}.(a) ensure that $\mathsf D (G) \ge \mathsf D (H) + 1 = 10$.
If $H \in \{ M_{16}, D_{16}, SD_{16}, Q_{16} \}$, then it follows by Lemma~\ref{lem:dD} that $\mathsf D (G) \ge \mathsf D (H) \ge 10$.
We continue the remaining cases for $H$.

\smallskip
\noindent
{\bf SUBCASE 1.1.} $H \cong C^{4}_2$.
\smallskip

Then, $\Aut (H) \cong \GL (4,2)$, and so $|\Aut (H)| = 2^{6} 3^{2} 5^{1} 7^{1}$.
Let $A \in \GL(4,2)$ be such that $A^{3} = I$.
Then, the minimal polynomial of $A$ divides $x^{3}-1$, and $x^{3}-1 = (x-1)(x^{2} + x + 1) \in \mathbb Z_2 [x]$.
Under the constraints of the rational canonical form, we obtain the following permissible lists of invariant factors:
\begin{enumerate}
\item[(i)] $(x-1)$, $(x-1)$, $(x-1)$, $(x-1)$

\item[(ii)] $(x-1)$, $(x-1)(x^{2} + x + 1)$

\item[(iii)] $(x^{2} + x + 1)$, $(x^{2} + x + 1)$
\end{enumerate}
Then, (i) corresponds to the identity matrix in $\GL(4,2)$, and (ii) (resp., (iii)) corresponds to the matrix $B = \begin{bmatrix} 1 & & & \\ & 0 & 0 & 1 \\ & 1 & 0 & 0 \\ & 0 & 1 & 0 \end{bmatrix}$ $\Bigg($ resp., $C = \begin{bmatrix} 0 & 1 & & \\ 1 & 1 & & \\ & & 0 & 1 \\ & & 1 & 1 \end{bmatrix} \Bigg)$.
Let $H \cong C^{4}_2 \cong \langle a \rangle \times \langle b \rangle \times \langle c \rangle \times \langle d \rangle$, $K \cong C_3 \cong \langle e \rangle$, and $\varphi \colon K \to \Aut (H)$ be a homomorphism induced from the action of $K$ on $H$ by conjugation.
Since $G$ is non-abelian, $\varphi [K] \le \Aut (H)$ is a non-trivial subgroup of order 3, and thus either $\varphi [K] = \langle B \rangle$ or $\varphi [K] = \langle C \rangle$.

If $\varphi [K] = \langle B \rangle$, then $G \cong \langle a, b, c, d, e \mid a^{2} = b^{2} = c^{2} = d^{2} = e^{3} = 1_G, \mbox{ and } ab = ba, ac = ca, ad = da, bc = cb, bd = db, cd = dc, ea = ae, ebe^{-1} = d, ece^{-1} = b, \, ede^{-1} = c \rangle$.
Let $x = bcd$, $y = cd$, and $z = bd$.
Then, $G$ has a subgroup $L \cong \langle x, y, z, e \mid x^{2} = y^{2} = z^{2} = e^{3} = 1_G, \mbox{ and } xy = yx, xz = zx, yz = zy, ex = xe, eye^{-1} = yz, eze^{-1} = y \rangle \cong C_2 \times A_4$ of order 24, and hence it follows by Lemma~\ref{lem:dD} that $\mathsf D (G) \ge \mathsf D (L) = 10$.

If $\varphi [K] = \langle C \rangle$, then $G = \langle a, b, c, d, e \mid a^{2} = b^{2} = c^{2} = d^{2} = e^{3} = 1_G, \mbox{ and } ab = ba, ac = ca, ad = da, bc = cb, bd = db, cd = dc, eae^{-1} = b, ebe^{-1} = ab, ece^{-1} = d, ede^{-1} = cd \rangle$.
Observe that $b \notin \langle c, d, e \rangle$ and $d \notin \langle a, b, e \rangle$.
Now, we consider the sequence
\[
  S = b^{[3]} \bdot d^{[3]} \bdot e^{[2]} \bdot (e^{-1})^{[2]}  \in \mathcal F (G) \,.
\]
Note that $1_G = ebebde^{-1}de^{-1}bd$, so that $S$ is a product-one sequence over $G$.
To show that $S$ is minimal, assume to the contrary that $S = S_1 \bdot S_2$ for some non-trivial $S_1, S_2 \in \mathcal B (G)$ with $b \mid S_1$.
Then, we have either $b \mid S_2$ or $b \nmid S_2$.
If $b \mid S_2$, then by symmetry, we may assume that $b^{[2]} \mid S_1$ and $b \mid S_2$.
Since $S_2$ is product-one, it follows that $b \in \pi (S_2 \bdot b^{[-1]}) = \pi \big( d^{[i]} \bdot e^{[j]} \bdot (e^{-1})^{[k]} \big) \subseteq \langle c, d, e \rangle$ for some $i \in [0,3]$ and $j, k \in [0,2]$, a contradiction.
Thus, we must have that $b \nmid S_2$, implying that $b^{[3]} \mid S_1$ and $S_2 \mid d^{[3]} \bdot e^{[2]} \bdot (e^{-1})^{[2]}$.
Then, it follows by an easy computation that $S_2$ is one of the following sequences:
\[
  d^{[2]} \,, \quad e \bdot e^{-1} \,, \quad d^{[2]} \bdot e \bdot e^{-1} \,, \quad e^{[2]} \bdot (e^{-1})^{[2]} \,, \quad d^{[2]} \bdot e^{[2]} \bdot (e^{-1})^{[2]} \,, \quad d^{[3]} \bdot e^{[2]} \bdot (e^{-1})^{[2]}
\]
(note that the last sequence is product-one because $1_G = eede^{-1}de^{-1}d$).
If $e^{[2]} \bdot (e^{-1})^{[2]} \mid S_2$, then $S_1 = b^{[3]} \bdot d^{[i]}$ for some $i \in [0,3]$.
Since $bd = db$, it follows that $S_1$ cannot be a product-one sequence, a contradiction.
If $S_2 = e \bdot e^{-1}$, then $S_1 = b^{[3]} \bdot d^{[3]} \bdot e \bdot e^{-1}$, and hence $e \in \pi (S_1 \bdot (e^{-1})^{[-1]}) = \pi (b^{[3]} \bdot d^{[3]} \bdot e)$, which leads to a contradiction.
If $S_2 = d^{[2]}$ or $S_2 = d^{[2]} \bdot e \bdot e^{-1}$, then $S_1 = b^{[3]} \bdot d \bdot e^{[2]} \bdot (e^{-1})^{[2]}$ or $S_1 = b^{[3]} \bdot d \bdot e \bdot e^{-1}$, and in either case, $d \in \pi (S_1 \bdot d^{[-1]}) = \pi (b^{[3]} \bdot e^{[i]} \bdot (e^{-1})^{[i]}) \subseteq \langle a, b, e \rangle$ for some $i \in [1,2]$, again a contradiction.
Therefore, $S$ must be a minimal product-one sequence over $G$, and thus $\mathsf D (G) \ge |S| = 10$.

\smallskip
\noindent
{\bf SUBCASE 1.2.} $H \cong C_2 \times C_2 \times C_4$.
\smallskip

It follows by \cite[Theorem 3.27]{Pa-Si-Ya18} that $| \Aut (H) | = 192 = 2^{6} 3^{1}$.
Every subgroup of $\Aut (H)$ of order 3 is a Sylow 3-subgroup, and so by Sylow's Theorem, any two subgroups of $\Aut (H)$ of order 3 are conjugate.
Hence, for any two homomorphisms $\varphi_1$, $\varphi_2 \colon K \to \Aut (H)$, we infer that $H \rtimes_{\varphi_1} K \cong H \rtimes_{\varphi_2} K$ (see \cite[Exercise 6 on p.184]{Du-Fo04} or \cite[Theorem 3.3]{Ta55}), i.e., there exists a unique group $G$ up to isomorphism.
Let $H \cong C^{2}_2 \times C_4 \cong ( \langle a \rangle \times \langle b \rangle) \times \langle c \rangle$ and $K \cong C_3 \cong \langle d \rangle$.
Now we define the map $\varphi \colon K \to \Aut (H)$ by $\varphi (x) = f_x$, where $f_x \colon H \to H$ given by $f_x (h) = xhx^{-1}$.
Clearly, $\varphi$ is a group homomorphism and $\varphi [K]$ is a subgroup of $\Aut (H)$ of order 3.
Note that $f_d \colon H \to H$ defined by $f_d (a) = b$, $f_d (b) = ab$, and $f_d (c) = c$ generates a subgroup of $\Aut (H)$ of order 3.
Thus $\varphi [K]$ and $\langle f_d \rangle$ are conjugate, and hence $G \cong H \rtimes_{\varphi} K \cong \langle a, b, c, d \mid a^{2} = b^{2} = c^{4} = d^{3} = 1_G, \mbox{ and } ab = ba, ac = ca, bc = cb, dad^{-1} = b, dbd^{-1} = ab, dcd^{-1} = c \rangle$.
Then, $G$ has a subgroup $\langle c, d \mid c^{4} = d^{3} = 1_G, \text{ and } dcd^{-1} = c \rangle \cong C_4 \times C_3 \cong C_{12}$, and hence $G$ has an element of order 12, a contradiction to (\ref{eq:order}).

\smallskip
\noindent
{\bf SUBCASE 1.3.} $H \cong C^{2}_4$.
\smallskip

It follows by \cite[Theorem 3.27]{Pa-Si-Ya18} that $| \Aut (H) | = 96 = 2^{5} 3^{1}$, and by Sylow's Theorem, any two subgroups of $\Aut (H)$ of order 3 are conjugate.
Let $\varphi \colon K \to \Aut (H)$ be a homomorphism induced from the action of $K$ on $H$ by conjugation.
Since $G$ is non-abelian, $\varphi [K]$ must be a non-trivial subgroup of $\Aut (H)$ of order 3.
Let $H \cong C^{2}_4 \cong \langle a \rangle \times \langle b \rangle$ and $K \cong C_3 \cong \langle c \rangle$.
Then, the same argument as used in {\bf SUBCASE 1.2} ensures that there exists a unique group $G$ up to isomorphism.
If we set $\varphi (c) (a) = ab^{-1}$ and $\varphi (c) (b) = a^{-1}b^{2}$, then $G \cong H \rtimes K \cong \langle a, b, c, \mid a^{4} = b^{4} = c^{3} = 1_G, \mbox{ and } ab = ba, cac^{-1} = ab^{-1}, cbc^{-1} = a^{-1}b^{2} \rangle$.
Now, we consider the sequence
\[
  S = a^{[5]} \bdot b^{[2]} \bdot c^{[3]} \in \mathcal F (G) \,.
\]
Note that $1_G = b c a a a b c c a a$, so that $S$ is a product-one sequence over $G$ of length 10.
To show that $S$ is minimal, we assume that $S = S_1 \bdot S_2$ for some $S_1$, $S_2 \in \mathcal B (G)$ with $c \mid S_1$.
Since $c$ is independent with $a, b$ in $G$ and $\ord (c) = 3$, it follows that $c^{[3]} \mid S_1$, so that $S_2 \mid a^{[5]} \bdot b^{[2]}$.
If $S_2$ is a non-trivial sequence, then, since $a, b$ are independent in $G$ and $\ord (a) = \ord (b) = 4$, we infer that $S_2 = a^{[4]}$ and $S_1 = a \bdot b^{[2]} \bdot c^{[3]}$.
Since $S_1$ is product-one, we must have that $a^{-1} \in \pi \big( b^{[2]} \bdot c^{[3]} \big)$, but $\pi \big( b^{[2]} \bdot c^{[3]} \big) = \{ a^{2}, b^{2}, ab^{2}, a^{2}b^{2}, b^{3}, a^{3}b^{3} \}$, a contradiction.
Thus, $S_2$ must be a trivial sequence, and so $S = S_1$ is a minimal product-one sequence, forcing that $\mathsf D (G) \ge |S| = 10$.

\smallskip
\noindent
{\bf SUBCASE 1.4.} $H \cong (C_2 \times C_2) \rtimes C_4$.
\smallskip

Then, from the help of $\GAP$, we have that $| \Aut (H) | = 32 = 2^{5}$, and hence $\Aut (H)$ has no subgroup of order 3.
This means that, for any homomorphism $\varphi \colon K \to \Aut (H)$, $\varphi [K]$ must be a trivial subgroup of $\Aut (H)$, and thus we infer that $G \cong H \rtimes K \cong H \times K$ (see \cite[Proposition 5.11]{Du-Fo04}).
Therefore, $G \cong \langle a, b, c, d \mid a^{2} = b^{2} = c^{4} = d^{3} = 1_G, \mbox{ and } cac^{-1} = ab = ba, ad = da, bc = cb, bd = db, cd = dc \rangle$, and $G$ has a subgroup $\langle c, d \mid c^{4} = d^{3} = 1_G, \mbox{ and } cd = dc \rangle \cong C_{12}$, i.e., $G$ has an element of order 12, a contradiction to (\ref{eq:order}).

\smallskip
\noindent
{\bf SUBCASE 1.5.} $H \cong C_2 \times D_8$.
\smallskip

Then, from the help of $\GAP$, we have that $| \Aut (H) | = 64 = 2^{6}$, and hence $\Aut (H)$ has no subgroup of order 3.
Thus, the same argument as used in {\bf SUBCASE 1.4} ensures that $G \cong H \times K \cong C_6 \times D_8$.
Then $G$ has a subgroup $L \cong C_3 \times D_8$ of order 24, and it follows by Lemma~\ref{lem:dD} that $\mathsf D (G) \ge \mathsf D (L) = 14$.

\smallskip
\noindent
{\bf SUBCASE 1.6.} $H \cong C_2 \times Q_8$.
\smallskip

Then, from the help of $\GAP$, we have that $| \Aut (H) | = 192 = 2^{6} 3^{1}$, and by Sylow's Theorem, any two subgroups of $\Aut (H)$ of order 3 are conjugate.
Let $\varphi \colon K \to \Aut (H)$ be a homomorphism induced from the action of $K$ on $H$ by conjugation.

Suppose that $\varphi [K]$ is a non-trivial subgroup of $\Aut (H)$ of order 3.
Let $H \cong C_2 \times Q_8 \cong \langle a, b, c \mid a^{2} = b^{4} = 1_G, c^{2} = b^{2}, \mbox{ and } ab = ba, ac = ca, cbc^{-1} = b^{-1} \rangle$, and $K \cong C_3 \cong \langle d \rangle$.
Then, the same argument as used in {\bf SUBCASE 1.2} ensures that there exists a unique group $G$ up to isomorphism.
If we set $\varphi (d) (a) = a$, $\varphi (d) (b) = c^{-1}$, $\varphi (d) (c) = bc^{-1}$, then $G \cong H \rtimes_{\varphi} K \cong \langle a, b, c, d \mid a^{2} = b^{4} = d^{3} = 1_G, c^{2} = b^{2}, \mbox{ and } ab = ba, ac = ca, ad = da, cbc^{-1} = b^{-1}, dbd^{-1} = c^{-1}, dcd^{-1} = bc^{-1} \rangle$.
Note that $L \cong \langle b, c, d \mid b^{4} = c^{3} = 1_G, c^{2} = b^{2}, \mbox{ and } cbc^{-1} = b^{-1}, dbd^{-1} = c^{-1}, dcd^{-1} = bc^{-1} \rangle$ is a non-abelian subgroup of $G$ of order 24.
Hence, it follows by Lemma~\ref{lem:dD} that $\mathsf D (G) \ge \mathsf D (L) \ge 10$.

Suppose that $\varphi [K]$ is a trivial subgroup of $\Aut (H)$.
Then, $G \cong H \rtimes_{\varphi} K \cong H \times K \cong C_6 \times Q_8$.
Hence $G$ has a subgroup $L \cong C_3 \times Q_8$ of order 24, and it follows by Lemma~\ref{lem:dD} that $\mathsf D (G) \ge \mathsf D (L) = 14$.

\smallskip
\noindent
{\bf SUBCASE 1.7.} $H \cong (C_2 \times C_4) \rtimes_{\phi} C_2$, where $\phi \colon C_2 \to \Aut (C_2 \times C_4)$ is given by $\phi (c) \colon \left\{\begin{aligned} & \alpha \mapsto \alpha^{-1} \\ & \beta \mapsto \alpha^{2}\beta \,. \end{aligned} \right.$
\smallskip

Then, from the help of $\GAP$, we have that $| \Aut (H) | = 48 = 2^{4} 3^{1}$, and by Sylow's Theorem, any two subgroups of $\Aut (H)$ of order 3 are conjugate.
Let $\varphi \colon K \to \Aut (H)$ be a homomorphism induced from the action of $K$ on $H$ by conjugation.

Suppose that $\varphi [K]$ is a non-trivial subgroup of $\Aut (H)$ of order 3.
Let $H \cong \langle \alpha, \beta, c \mid \alpha^{4} = \beta^{2} = c^{2} = 1_G, \mbox{ and } \alpha\beta = \beta\alpha, c\alpha c = \alpha^{-1}, c\beta c = \alpha^{2}\beta \rangle$, and $K \cong C_3 \cong \langle d \rangle$.
Let $a = \alpha^{-1}\beta$, and $b = \alpha$. Then, it is easy to see that $H \cong \langle a, b, c \mid a^{4} = c^{2} = 1_G, b^{2} = a^{2}, \mbox{ and } ab = ba, ca = ac, cbc = a^{2}b \rangle$.
Then, by \cite[Theorem 3.3]{Ta55}, there exists a unique group $G$ up to isomorphism.
If we set $\varphi (d) (a) = a$, $\varphi (d) (b) = a^{-1}bc$, and $\varphi (d) (c) = a^{-1}b$, then $G \cong H \rtimes_{\varphi} K \cong \langle a, b, c, d \mid a^{4} = c^{2} = d^{3} = 1_G, b^{2} = a^{2}, \mbox{ and } ab = ba, ac = ca, ad = da, cbc = a^{2}b, dbd^{-1} = a^{-1}bc, dcd^{-1} = a^{-1}b \rangle$.
But, $G$ has a subgroup $\langle a, d \mid a^{4} = d^{3} = 1_G, \mbox{ and } ad = da \rangle \cong C_{12}$, i.e., $G$ has an element of order 12, a contradiction to (\ref{eq:order}).

Suppose that $\varphi [K]$ is a trivial subgroup of $\Aut (H)$.
Then, $G \cong H \rtimes_{\varphi} K \cong H \times K \cong \langle a, b, c, d \mid a^{4} = c^{2} = d^{3} = 1_G, b^{2} = a^{2}, \mbox{ and } ab = ba, ac = ca, ad = da, cbc = a^{2}b, bd = db, cd = dc \rangle$, and so $G$ has a subgroup $\langle a, d \mid a^{4} = d^{3} = 1_G, \mbox{ and } ad = da \rangle \cong C_{12}$, i.e., $G$ has an element of order 12, a contradiction to (\ref{eq:order}).

\smallskip
\noindent
{\bf SUBCASE 1.8.} $H \cong C_4 \rtimes C_4$.
\smallskip

Then, from the help of $\GAP$, we have that $|\Aut (H)| = 2^{5}$, and hence $\Aut (H)$ has no subgroup of order 3.
Thus, the same argument as used in {\bf SUBCASE 1.4} ensures that $G \cong H \times K \cong \langle a, b, c \mid a^{4} = b^{4} = c^{3} = 1_G, \mbox{ and } bab = a^{-1}, ac = ca, bc = cb \rangle$.
Hence, $G$ must have an element of order 12, a contradiction to (\ref{eq:order}).

\medskip
\noindent
{\bf CASE 2.} $j = 2$ and $k = \ell = 0$, i.e., $|G| = 2^{i} 3^{2}$.
\smallskip

Then, since $|G| > 42$, it follows that $i \in [3,4]$.
Note that $\mathsf D (G) \ge 10$ for $i \in [1,2]$ by Lemma~\ref{lem:dD}.

\smallskip
\noindent
{\bf SUBCASE 2.1.} $i = 3$, i.e., $|G| = 2^{3} 3^{2} = 72$.
\smallskip

By Lemma~\ref{lem:perfect}, every group of order 72 is not a perfect group, and so such a group has a non-trivial proper commutator subgroup $G'$, whence $G/G'$ is a non-trivial abelian group.
The possible values of the pair $\big( |G'|, |G/G'| \big)$ are as follows:
\[
  \big( |G'|, |G/G'| \big) \,\, \in \,\, \big\{ (2,36), \, (3,24), \, (4,18), \, (6,12), \, (8,9), \, (9,8), \, (12,6), \, (18,4), \, (24,3), \, (36,2) \big\} \,.
\]

If $|G'| = 36$ or $|G/G'| = 36$, then the assertion follows by Lemma~\ref{lem:subgroup}.

If $|G/G'| = 24$, then $G/G'$ is isomorphic to $C_{24}$, or to $C_2 \times C_{16}$, or to $C^{2}_2 \times C_6$.
Since $|G'| = 3$, $G' \cong C_3$, and it follows by Lemma~\ref{lem:ineqDC}.2 that $\mathsf D (G) \ge \mathsf d (G) + 1 \ge \big(\mathsf d (G') + \mathsf d (G/G')\big) + 1 \ge 10$.

If $|G/G'| = 18$, then $G/G'$ is isomorphic to $C_{18}$, or to $C_3 \times C_6$.
Since $|G'| = 4$, $G'$ is isomorphic to $C_4$, or to $C^{2}_2$, and it follows by Lemma~\ref{lem:ineqDC}.2 that $\mathsf D (G) \ge \mathsf d (G) + 1 \ge \big(\mathsf d (G') + \mathsf d (G/G')\big) + 1 \ge 10$.

If $|G/G'| = 12$, then $G/G'$ is isomorphic to $C_{12}$, or to $C_2 \times C_6$.
Since $|G'| = 6$, $G'$ is isomorphic to $C_6$, or to $D_6$, and it follows by Lemmas ~\ref{lem:ineqDC}.2 and \ref{lem:dD} that $\mathsf D (G) \ge \mathsf d (G) + 1 \ge \big(\mathsf d (G') + \mathsf d (G/G')\big) + 1 \ge 10$.

If $|G/G'| = 9$, then $G/G'$ is isomorphic to $C_9$, or to $C_3 \times C_3$.
Since $|G'| = 8$, $G'$ is isomorphic to one of the groups in $\{ C_8, \, C_2 \times C_4, \, C^{3}_2, \, D_8, \, Q_8 \}$.
Let $K$ be a Sylow $3$-subgroup of $G$.
Then, $|K| = 9$, and it follows by (\ref{eq:order}) that $K \cong C^{2}_3$.
Since $G'$ intersects trivially with $K$, we obtain that $G \cong G' \rtimes K$.
Note that $\Aut (C_8) \cong C^{2}_2$ and $\Aut (C_2 \times C_4) \cong \Aut (D_8) \cong D_8$.
This means that, for $G'$ is isomorphic to $C_8$, or to $C_2 \times C_4$, or to $D_8$, any homomorphism from $K$ to $\Aut (G')$ is trivial, so that $G \cong G' \rtimes K \cong G' \times K$.
Since $G$ is non-abelian, we only have that $G \cong C^{2}_3 \times D_8$.
In this case, $L \cong C_3 \times D_8$ is a subgroup of $G$ of order 24, and Lemma~\ref{lem:dD} ensures that $\mathsf D (G) \ge \mathsf D (L) = 14$.
For $G' \cong Q_8$, we observe that $\Aut (G') \cong S_4$, so that $|\Aut (G')| = 2^{3} 3^{1}$.
Let $\varphi \colon K \to \Aut (G')$ be a homomorphism induced from the action of $K$ on $G'$ by conjugation.
If $\varphi [K]$ is a trivial subgroup of $\Aut (G')$, then $G \cong G' \rtimes K \cong C^{2}_3 \times Q_8$, and so $L \cong C_3 \times Q_8$ is a subgroup of $G$ of order 24.
It follows by Lemma~\ref{lem:dD} that $\mathsf D (G) \ge \mathsf D (L) = 14$.
If $\varphi [K]$ is a non-trivial subgroup of $\Aut (G')$, then $\ker (\varphi)$ must be a non-trivial proper subgroup of $K$, and so $|\ker (\varphi)| = 3$.
This means that $K / \ker (\varphi) \cong \varphi [K]$ is a subgroup of $\Aut (G')$ of order 3.
Let $G' \cong \langle a, b \mid a^{4} = 1_G, b^{2} = a^{2}, \mbox{ and } bab^{-1} = a^{-1} \rangle$, $K \cong C^{2}_3 \cong \langle c \rangle \times \langle d \rangle$.
Since any subgroup of $\Aut (G')$ of order 3 are conjugate, by \cite[Theorem 3.3]{Ta55}, there exists a unique group $G$ up to isomorphism.
If we set $\varphi (c) (a) = b$, $\varphi (c) (b) = ab$, $\varphi (d) (a) = a$, and $\varphi (d) (b) = b$, then $G \cong \langle a, b, c, d \mid a^{4} = c^{3} = d^{3} = 1_G, b^{2} = a^{2}, \mbox{ and } bab^{-1} = a^{-1}, cd = dc, cac^{-1} = b, cbc^{-1} = ab, da = ad, db = bd \rangle$.
Hence, $G$ has a subgroup $L \cong \langle a, b, d \rangle \cong C_3 \times Q_8$, and it follows by Lemma~\ref{lem:dD} that $\mathsf D (G) \ge \mathsf D (L) = 14$.
Suppose now that $G' \cong C^{3}_2$.
Then, $\Aut (G') \cong \GL (3,2)$ and $|\Aut (G')| = 2^{3} 3^{1} 7^{1}$.
Let $\varphi \colon K \to \Aut (G')$ be a homomorphism induced from the action of $K$ on $G'$ by conjugation.
Since any subgroup of $\Aut (G')$ of order 3 are conjugate, we infer that there exists a unique group $G$ up to isomorphism.
Under the constraints of the rational canonical form, we may assume that $\varphi [K]$ is conjugate to a subgroup $\Bigg\langle \begin{bmatrix} 0 & 0 & 1 \\ 1 & 0 & 0 \\ 0 & 1 & 0 \end{bmatrix} \Bigg\rangle$, where the generator is the rational canonical form of a matrix of order 3 in $\GL (3,2)$ corresponding to the invariant factor $x^{3} - 1$.
Let $G' \cong C^{3}_2 \cong \langle a \rangle \times \langle b \rangle \times \langle c \rangle$ and $K \cong C^{2}_3 \cong \langle d \rangle \times \langle e \rangle$.
Thus, we obtain that $G \cong \langle a, b, c, d, e \mid a^{2} = b^{2} = c^{2} = d^{3} = e^{3} = 1_G, \mbox{ and } ab = ba, ac = ca, bc = cb, de = ed, dad^{-1} = c, dbd^{-1} = a, dcd^{-1} = b, \, ea= ae, eb = be, ec = ce \rangle$.
Let $x = abc$, $y = bc$, and $z = ac$.
Then, $G$ has a subgroup $L \cong \langle x, y, z, d \mid x^{2} = y^{2} = z^{2} = d^{3} = 1_G, \mbox{ and } xy = yx, xz = zx, yz = zy, dx = xd, dyd^{-1} = yz, dzd^{-1} = y \rangle \cong C_2 \times A_4$, and it follows by Lemma~\ref{lem:dD} that $\mathsf D (G) \ge \mathsf D (L) = 10$.

If $|G/G'| = 8$, then $G/G'$ is isomorphic to $C_8$, or to $C_2 \times C_4$, or to $C^{3}_2$.
Since $|G'| = 9$, we infer by (\ref{eq:order}) that $G' \cong C^{2}_3$ is a Sylow 3-subgroup of $G$.
If $G/G' \cong C_8$, then by Lemma~\ref{lem:ineqDC}.2, we obtain that $\mathsf D (G) \ge \mathsf d (G) + 1 \ge \big( \mathsf d (G') + \mathsf d (G/G') \big) + 1 = 12$.
We may assume that $G/G' \cong C_2 \times C_4$ or $G/G' \cong C^{3}_2$.
Let $K$ be a Sylow 2-subgroup of $G$.
Then, $G'$ intersects trivially with $K$, and so $G = G' K$, i.e., $G \cong G' \rtimes K$.
Then, $K \cong K/(G' \cap K) \cong G'K/G' = G/G'$, so either $K \cong C_2 \times C_4$ or $K \cong C^{3}_2$.
In either case, we infer that $G$ has a subgroup $H \cong C^{2}_2$ of order 4.
Since $G'$ intersects trivially with $H$, it follows that $L = G'H$ is a subgroup of $G$ of order 36.
If $L$ is abelian, then it follows by (\ref{eq:order}) that $L \cong C^{2}_6$, and so $\mathsf D (G) \ge \mathsf D (L) + 1 = 12$ by Lemma~\ref{lem:ineqDC}.1.
If $L$ is non-abelian, then it follows by Lemma~\ref{lem:dD} that $\mathsf D (G) \ge \mathsf D (L) \ge 11$.

If $|G/G'| = 6$, then $G/G' \cong C_6$.
Since $|G'| = 12$, by (\ref{eq:order}), $G'$ is isomorphic to one of the groups in $\{ C_2 \times C_6, \, Q_{12},  \,A_4, \, D_{12} \}$.
Thus, it follows by Lemmas~\ref{lem:ineqDC}.2 and \ref{lem:dD} that $\mathsf D (G) \ge \mathsf d (G) + 1 \ge \big( \mathsf d (G') + \mathsf d (G/G') \big) + 1 \ge 10$.

If $|G/G'| = 4$, then $G/G'$ is isomorphic to $C_4$, or to $C^{2}_2$.
Since $|G'| = 18$, by (\ref{eq:order}), $G'$ is isomorphic to one of the groups in $\{ C_3 \times C_6, \, C_3 \times D_6, \, C^{2}_3 \rtimes C_2 \}$.
If $G'$ is abelian, then Lemmas ~\ref{lem:ineqDC}.2 and \ref{lem:abD}.(b) ensure that $\mathsf D (G) \ge \mathsf d (G) + 1 \ge \big( \mathsf d (G') + \mathsf d (G/G') \big) + 1 \ge 10$.
If $G'$ is non-abelian, then it follows by Lemma~\ref{lem:dD} that $\mathsf D (G) \ge \mathsf D (G') = 10$.

If $|G/G'| = 3$, then $G/G' \cong C_3$.
If $G'$ is abelian, then since $|G'| = 24$, it follows by (\ref{eq:order}) that $G' \cong C^{2}_2 \times C_6$.
Thus, Lemmas ~\ref{lem:ineqDC}.2 ensures that $\mathsf D (G) \ge \mathsf d (G) + 1 \ge \big( \mathsf d (G') + \mathsf d (G/G') \big) + 1 \ge 10$.
If $G'$ is non-abelian, then it follows by Lemma~\ref{lem:dD} that $\mathsf D (G) \ge \mathsf D (G') \ge 10$.

\smallskip
\noindent
{\bf SUBCASE 2.2.} $i = 4$, i.e., $|G| = 2^{4} 3^{2} = 144$.
\smallskip

By Lemma~\ref{lem:perfect}, every group of order 144 is not a perfect group, and so such a group has a non-trivial proper commutator subgroup.
The possible values of the pair $\big( |G'|, |G/G'| \big)$ are as follows:
\[
  \begin{aligned}
    \big( |G'|, |G/G'| \big) \,\, \in & \,\,\,\, \big\{ (2,72), (3,48), (4,36), (6,24), (8,18), (9,16), (12,12), (16,9), (18,8), (24,6), (36,4), \\
                                 & \quad\,\,  (48,3), (72,2) \big\} \,.
  \end{aligned}
\]

If $|G'| = 36$ or $|G/G'| = 36$, then Lemma~\ref{lem:subgroup} ensures that $\mathsf D (G) \ge 10$.

Suppose that $|G'| \in \{ 72, 48 \}$.
If $G'$ is non-abelian, then {\bf SUBCASE 2.1} and {\bf CASE 1} ensure that $\mathsf D (G) \ge \mathsf D (G') \ge 10$.
If $G'$ is abelian, then by (\ref{eq:order}), we infer that $G' \cong C_2 \times C^{2}_6$ (resp., $G' \cong C^{3}_2 \times C_6$) if $|G'| = 72$ (resp., $|G'| = 48$).
In either case, it follows by Lemma~\ref{lem:ineqDC}.1 that $\mathsf D (G) \ge \mathsf D (G') + 1 \ge 10$.

Suppose that $|G/G'| \in \{ 72, 48 \}$.
If $G/G'$ has an element of order at least 9, then by Lemmas ~\ref{lem:ineqDC}.3 and \ref{lem:order9}, $\mathsf D (G) \ge \mathsf D (G/G') \ge 10$, and so we may assume that every element of $G/G'$ has the order at most 8.
Since $G/G'$ is abelian, $G/G'$ is isomorphic to $C_2 \times C^{2}_6$, or to $C^{3}_2 \times C_6$.
In the former case, Lemma~\ref{lem:ineqDC}.3 ensures that $\mathsf D (G) \ge \mathsf D (G/G') \ge 10$.
In the latter case, since $G' \cong C_3$, it follows by Lemma~\ref{lem:ineqDC}.2 that $\mathsf D (G) \ge \mathsf d (G) + 1 \ge \big( \mathsf d (G') + \mathsf d (G/G') \big) + 1 \ge 11$.

If $|G/G'| = 24$, then $G/G'$ is isomorphic to $C_{24}$, or to $C_2 \times C_{12}$, or to $C^{2}_2 \times C_6$.
Since $|G'| = 6$, $G'$ is isomorphic to $C_6$, or to $D_6$.
By Lemmas~\ref{lem:ineqDC}.2 and \ref{lem:dD}, $\mathsf D (G) \ge \mathsf d (G) + 1 \ge \big( \mathsf d (G') + \mathsf d (G/G') \big) + 1 \ge 11$.

If $|G/G'| = 18$, then $G/G'$ is isomorphic to $C_{18}$, or to $C_3 \times C_6$.
Since $|G'| = 8$, by (\ref{eq:order}), $G'$ is isomorphic to one of the groups in $\{ C_2 \times C_4, \, C^{3}_2, \, D_8, \, Q_8 \}$.
Thus, it follows by Lemmas~\ref{lem:ineqDC}.2 and \ref{lem:dD} that $\mathsf D (G) \ge \mathsf d (G) + 1 \ge \big( \mathsf d (G') + \mathsf d (G/G') \big) + 1 \ge 11$.

If $|G/G'| = 16$, then $G/G'$ is isomorphic to $C_{16}$, or to $C^{2}_4$, or to $C_2 \times C_8$, or to $C^{2}_2 \times C_4$, or to $C^{4}_2$.
Since $|G'| = 9$, we infer by (\ref{eq:order}) that $G' \cong C^{2}_3$.
If $G/G' \ncong C^{4}_2$, then it follows by Lemmas~\ref{lem:ineqDC}.2 and \ref{lem:abD}.(a) that $\mathsf D (G) \ge \mathsf d (G) + 1 \ge \big( \mathsf d (G') + \mathsf d (G/G') \big) + 1 \ge 10$.
Thus, we may assume that $G/G' \cong C^{4}_2$.
Note that $G'$ is a Sylow 3-subgroup of $G$.
If $K$ is a Sylow 2-subgroup of $G$, then $G'$ intersects trivially with $K$, and so $G = G' K$, i.e., $G \cong G' \rtimes K$.
Then, $K \cong K/(G' \cap K) \cong G'K/G' = G/G' \cong C^{4}_2$, and so $G$ must have a subgroup $L$ of order 36, whence $\mathsf D (G) \ge \mathsf D (L) \ge 11$.

If $|G/G'| = 12$, then $G/G'$ is isomorphic to $C_{12}$, or to $C_2 \times C_6$.
Since $|G'| = 12$, by (\ref{eq:order}), $G'$ is isomorphic to one of the groups in $\{ C_2 \times C_6, \, Q_{12},\,  A_4, \, D_{12} \}$.
Thus, it follows by Lemmas~\ref{lem:ineqDC}.2 and \ref{lem:dD} that $\mathsf D (G) \ge \mathsf d (G) + 1 \ge \big( \mathsf d (G') + \mathsf d (G/G') \big) + 1 \ge 11$.

If $|G/G'| = 9$, then $G/G'$ is isomorphic to $C_9$, or to $C_3 \times C_3$.
If $G'$ is non-abelian, then by Lemmas~\ref{lem:ineqDC}.2 and \ref{lem:dD}, $\mathsf D (G) \ge \mathsf d (G) + 1 \ge \big( \mathsf d (G') + \mathsf d (G/G') \big) + 1 \ge 10$.
If $G'$ is abelian, then in view of (\ref{eq:order}), $G'$ is isomorphic to one of the groups in $\{ C^{2}_4, \, C^{2}_2 \times C_4, \, C_2 \times C_8, \, C^{4}_2 \}$.
If $G' \ncong C^{4}_2$, then it follows by Lemma~\ref{lem:ineqDC}.2 that $\mathsf D (G) \ge \mathsf d (G) + 1 \ge \big( \mathsf d (G') + \mathsf d (G/G') \big) + 1 \ge 10$.
Thus, we may assume that $G' \cong C^{4}_2$, which is a Sylow 2-subgroup of $G$.
If $K$ is a Sylow 3-subgroup of $G$, then $G = G'K$ with $G' \cap K = \{ 1_G \}$, i.e., $G \cong G' \rtimes K$.
Hence, $K \cong K/(G' \cap K) \cong G'K/G' = G/G'$, and in view of (\ref{eq:order}), we must have that $K \cong C^{2}_3$, whence $G$ has a subgroup $L = G' C_3$ of order 48.
If $L$ is abelian, then $\mathsf D (G) \ge \mathsf D (L) + 1 \ge 10$ by Lemma~\ref{lem:ineqDC}.1.
If $L$ is non-abelian, then {\bf CASE 1} ensures that $\mathsf D (G) \ge \mathsf D (L) \ge 10$.

If $|G/G'| = 8$, then $G/G'$ is isomorphic to $C_8$, or to $C_2 \times C_4$, or to $C^{3}_2$.
If $G'$ is non-abelian, then by Lemma~\ref{lem:dD}, $\mathsf D (G) \ge \mathsf D (G') \ge 10$.
If $G'$ is abelian, then by (\ref{eq:order}), we must have that $G' \cong C_3 \times C_6$, and it follows by Lemma~\ref{lem:ineqDC}.2 that $\mathsf D (G) \ge \mathsf d (G) + 1 \ge \big( \mathsf d (G') + \mathsf d (G/G') \big) + 1 \ge 11$.

If $|G/G'| = 6$, then $G/G' \cong C_6$.
If $G'$ is non-abelian, then by Lemma~\ref{lem:dD}, $\mathsf D (G) \ge \mathsf D (G') \ge 10$.
If $G'$ is abelian, then by (\ref{eq:order}), we must have that $G' \cong C^{2}_2 \times C_6$, and it follows by Lemma~\ref{lem:ineqDC}.2 that $\mathsf D (G) \ge \mathsf d (G) + 1 \ge \big( \mathsf d (G') + \mathsf d (G/G') \big) + 1 \ge 13$.

\smallskip
\noindent
{\bf CASE 3.} $j = 3$ and $k = \ell = 0$, i.e., $|G| = 2^{i} 3^{3}$.
\smallskip

Then, since $|G| > 42$, it follows that $i \in [1,4]$.

\smallskip
\noindent
{\bf SUBCASE 3.1.} $i = 1$, i.e., $|G| = 2^{1} 3^{3} = 54$.
\smallskip

Let $H$ be a Sylow 3-subgroup of $G$.
Since $(G \colon H) = 2$, $H$ must be a normal subgroup of $G$.
Let $K$ be a Sylow 2-subgroup of $G$.
Then, since $H$ intersects trivially with $K$, it follows that $G \cong H \rtimes K$.

Suppose that $H$ is abelian.
Then, by (\ref{eq:order}), we infer that $H \cong C^{3}_3$.
Note that $\Aut (H) \cong \GL (3,3)$, and $|\Aut (H)| = 2^{5} 3^{3} 13^{1}$.
Let $A \in \GL (3,3)$ be such that $A^{2} = I$.
Then, the minimal polynomial of $A$ divides $x^{2} - 1 = (x-1)(x+1) \in \mathbb Z_3 [x]$.
Under the constraints of the rational canonical form, we obtain the following permissible lists of invariant factors:
\begin{enumerate}
\item[(i)] $(x-1)$, $(x-1)$, $(x-1)$

\smallskip
\item[(ii)] $(x+1)$, $(x+1)$, $(x+1)$

\smallskip
\item[(iii)] $(x+1)$, $(x-1)(x+1)$

\smallskip
\item[(iv)] $(x-1)$, $(x-1)(x+1)$
\end{enumerate}

\noindent
Then, (i) corresponds to the identity matrix in $\GL (3,3)$, and (ii) (resp., (iii) and (iv)) corresponds to the matrix $B= \begin{bmatrix} 2 & & \\ & 2 & \\ & & 2 \end{bmatrix} \Bigg($ resp., $C = \begin{bmatrix} 2 & & \\ & 0 & 1 \\ & 1 & 0 \end{bmatrix}$, and $D = \begin{bmatrix} 1 & & \\ & 0 & 1 \\ & 1 & 0 \end{bmatrix} \Bigg)$.
Let $H \cong C^{3}_3 \cong \langle a \rangle \times \langle b \rangle \times \langle c \rangle$, $K \cong C_2 \cong \langle d \rangle$, and $\varphi \colon K \to \Aut (H)$ be a homomorphism induced from the action of $K$ on $H$ by conjugation.
Since $G$ is non-abelian, $\varphi [K] \le \Aut (H)$ is a non-trivial subgroup of order 2, and thus $\varphi [K] = \langle B \rangle$, $\varphi [K] = \langle C \rangle$, or $\varphi [K] = \langle D \rangle$.

If $\varphi [K] = \langle B \rangle$, then $G \cong \langle a, b, c, d \mid a^{3} = b^{3} = c^{3} = d^{2} = 1_G, \mbox{ and } ab = ba, ac = ca, bc = cb, dad = a^{-1}, dbd = b^{-1}, dcd = c^{-1} \rangle$, and $G$ has a subgroup $L = \langle a, b ,d \mid a^{3} = b^{3} = d^{2} = 1_G, \mbox{ and } ab = ba, dad = a^{-1}, dbd = b^{-1} \rangle$ of order 18.
Thus, it follows by Lemma~\ref{lem:dD} that $\mathsf D (G) \ge \mathsf D (L) = 10$.

If $\varphi [K] = \langle C \rangle$, then $G \cong \langle a, b, c, d \mid a^{3} = b^{3} = c^{3} = d^{2} = 1_G, \mbox{ and } ab = ba, ac = ca, bc = cb, dad = a^{-1}, dbd = c, \, dcd = b \rangle$.
Let $x = c^{-1}b$.
Then, it is easy to see that $L \cong \langle a, x ,d \mid a^{3} = x^{3} = d^{2} = 1_G, \mbox{ and } ax = xa, dad = a^{-1}, dxd = x^{-1} \rangle$ is a subgroup of $G$ of order 18, and  Lemma~\ref{lem:dD} ensures that $\mathsf D (G) \ge \mathsf D (L) = 10$.

If $\varphi [K] = \langle D \rangle$, then $G \cong \langle a, b, c, d \mid a^{3} = b^{3} = c^{3} = d^{2} = 1_G, \mbox{ and } ab = ba, ac = ca, bc = cb, da = ad, dbd = c, dcd = b \rangle$.
Let $x = c^{-1}b$.
Then, it is easy to see that $L \cong \langle a, x, d \mid a^{3} = x^{3} = d^{2} = 1_G, \mbox{ and } ax = xa, ad = da, dxd = x^{-1} \rangle$ is a subgroup of $G$ of order 18, and it follows by Lemma~\ref{lem:dD} that $\mathsf D (G) \ge \mathsf D (L) = 10$.

Suppose that $H$ is non-abelian.
Then, by (\ref{eq:order}), we infer that $H \cong H_{27}$.
Note that $\Aut (H) \cong \Inn (H) \rtimes \GL (2,3)$ (see \cite[Lemma A.20.11]{Do-Ha92}), where $Z (H) \cong C_3$ and $\Inn (H) \cong H / Z (H) \cong C^{2}_3$.
Let $H \cong \langle a, b, c \mid a^{3} = b^{3} = c^{3} = 1_G, \text{ and } ac = ca, \, bc = cb, \, bab^{-1} = ac^{-1} \rangle$, $K \cong C_2 \cong \langle d \rangle$, and $\varphi \colon K \to \Aut (H)$ be a homomorphism induced from the action of $K$ on $H$ by conjugation.

If $\varphi [K]$ is trivial, then $G \cong H \rtimes K \cong H \times K$, and since $\langle a, c \rangle$ is a normal subgroup of $H$, we infer that $G$ has a normal subgroup $L \cong \langle a, c, d \mid a^{3} = c^{3} = d^{2} = 1_G, \mbox{ and } ac = ca, ad = da, cd = dc \rangle \cong C_3 \times C_6$.
Thus, it follows by Lemma~\ref{lem:ineqDC}.2 that $\mathsf D (G) \ge \mathsf d (G) + 1 \ge \big( \mathsf d (L) + \mathsf d (G/L) \big) + 1 = 10$.

If $\varphi [K]$ is non-trivial, then $\varphi [K]$ is a subgroup of $\Aut (H)$ of order 2, and since $|\Aut (H)| = 2^{4} 3^{3}$, $\varphi [K]$ is a subgroup of a Sylow 2-subgroup of $\Aut (H)$.
Let $A = \begin{bmatrix} 0 & 1 \\ 1 & 1 \end{bmatrix}$ and $B = \begin{bmatrix} 1 & 1 \\ 0 & 2 \end{bmatrix}$ be elements of $\GL (2,3)$.
Then, $A^{8} = B^{2} = I$ and $BAB = A^{3}$, and so $\langle A, B \rangle$ is a subgroup of $\GL (2,3)$ of order 16, i.e., $\langle A, B \rangle$ is isomorphic to a Sylow 2-subgroup of $\Aut (H)$.
By Sylow's Theorem, $\varphi [K]$ is isomorphic to a conjugate of a subgroup of $\langle A, B \rangle$ of order 2.
It is easy to see that $\langle A^{4} \rangle$, $\langle B \rangle$, $\langle A^{2}B \rangle$, $\langle A^{4}B \rangle$, and $\langle A^{6}B \rangle$ are distinct five subgroups of $\langle A, B \rangle$ of order 2, where
\[
  A^{4} = \begin{bmatrix} 2 & 0 \\ 0 & 2 \end{bmatrix}, \quad A^{2}B = \begin{bmatrix} 1 & 0 \\ 1 & 2 \end{bmatrix}, \quad A^{4}B = \begin{bmatrix} 2 & 2 \\ 0 & 1 \end{bmatrix}, \und A^{6}B = \begin{bmatrix} 2 & 0 \\ 2 & 1 \end{bmatrix}.
\]
Note that $\langle A^{4} \rangle$ is the center of $\GL (2,3)$, and since
\[
  \begin{bmatrix} 2 & 1 \\ 1 & 0 \end{bmatrix} \begin{bmatrix} 1 & 1 \\ 0 & 2 \end{bmatrix} \begin{bmatrix} 2 & 1 \\ 1 & 0 \end{bmatrix}^{-1} \hspace{-5pt}= \begin{bmatrix} 1 & 0 \\ 1 & 2 \end{bmatrix} \hspace{-1pt}, \quad \begin{bmatrix} 1 & 0 \\ 1 & 2 \end{bmatrix} \begin{bmatrix} 1 & 1 \\ 0 & 2 \end{bmatrix} \begin{bmatrix} 1 & 0 \\ 1 & 2 \end{bmatrix}^{-1} \hspace{-5pt}= \begin{bmatrix} 2 & 2 \\ 0 & 1 \end{bmatrix} \hspace{-1pt}, \quad \begin{bmatrix} 0 & 1 \\ 1 & 1 \end{bmatrix} \begin{bmatrix} 1 & 1 \\ 0 & 2 \end{bmatrix} \begin{bmatrix} 0 & 1 \\ 1 & 1 \end{bmatrix}^{-1} \hspace{-5pt}= \begin{bmatrix} 2 & 0 \\ 2 & 1 \end{bmatrix} \hspace{-1pt},
\]
it follows that $B$, $A^{2}B$, $A^{4}B$, and $A^{6}B$ are conjugate in $\GL (2,3)$.
Thus, we infer that any two non-central cyclic subgroups of $\langle A, B \rangle$ of order 2 are conjugate in $\GL (2,3)$, and so it follows by \cite[Theorem 3.3]{Ta55} that $\varphi [K]$ is isomorphic to $\langle A^{4} \rangle$, or to $\langle B \rangle$.
In the former case, $G \cong \langle a, b, c, d \mid a^{3} = b^{3} = c^{3} = d^{2} = 1_G, \mbox{ and } ac = ca, bc = cb, bab^{-1} = ca^{-1}, dad = a^{-1}, dbd = b^{-1}, dcd = c \rangle$, and $G$ has a subgroup $L \cong \langle a, c, d \mid a^{3} = c^{3} = d^{2} = 1_G, \mbox{ and } ac = ca, dc = cd, dad = a^{-1} \rangle$ of order 18, whence it follows by Lemma~\ref{lem:dD} that $\mathsf D (G) \ge \mathsf D (L) = 10$.
In the latter case, $G \cong \langle a, b, c, d \mid a^{3} = b^{3} = c^{3} = d^{2} = 1_G, \mbox{ and } ac = ca, bc = cb, bab^{-1} = ac^{-1}, dad = ab, dbd = b^{-1}, dcd = c \rangle$, and $G$ has a subgroup $L \cong \langle b, c, d \mid b^{3} = c^{3} = d^{2} = 1_G, \mbox{ and } bc = cb, cd = dc, dbd = b^{-1} \rangle$ of order 18, whence it follows by Lemma~\ref{lem:dD} that $\mathsf D (G) \ge \mathsf D (L) = 10$.

\smallskip
\noindent
{\bf SUBCASE 3.2.} $i = 2$, i.e., $|G| = 2^{2} 3^{3} = 108$.
\smallskip

By Lemma~\ref{lem:perfect}, every group of order 108 is not a perfect group, and so such a group has a non-trivial proper commutator subgroup.
The possible values of the pair $\big( |G'|, |G/G'| \big)$ are as follows:
\[
  \big( |G'|, |G/G'| \big) \,\, \in \,\, \big\{ (2,54), \, (3,36), \, (4,27), \, (6,18), \, (9,12), \, (12,9), \, (18, 6), \, (27, 4), \, (36,3), \, (54,2) \big\} \,.
\]

If $|G'| = 36$ or $|G/G'| = 36$, then Lemma~\ref{lem:subgroup} ensures that $\mathsf D (G) \ge 10$.

Suppose that $|G'| = 54$.
If $G'$ is non-abelian, then {\bf SUBCASE 3.1} ensures that $\mathsf D (G) \ge \mathsf D (G') \ge 10$.
If $G'$ is abelian, then by (\ref{eq:order}), we must have that $G' \cong C^{2}_3 \time C_6$, and it follows by Lemma~\ref{lem:ineqDC}.1 that $\mathsf D (G) \ge \mathsf D (G') + 1 \ge 11$.

Suppose that $|G/G'| = 54$.
If $G/G'$ has an element of order at least 9, then by Lemmas~\ref{lem:ineqDC}.3 and \ref{lem:order9}, $\mathsf D (G) \ge \mathsf D (G/G') \ge 10$, and so we may assume that every element in $G/G'$ has the order at most 8.
Since $G/G'$ is abelian, $G/G'$ is isomorphic to $C^{2}_3 \times C_6$, and it follows by Lemma~\ref{lem:ineqDC}.3 that $\mathsf D (G) \ge \mathsf D (G/G') \ge 10$.

If $|G/G'| = 27$, then $G/G'$ is isomorphic to $C_{27}$, or to $C_3 \times C_9$, or to $C^{3}_3$.
Since $|G'| = 4$, $G'$ is isomorphic to $C_4$, or to $C_2 \times C_2$.
If $G/G' \ncong C^{3}_3$, then it follows by Lemma~\ref{lem:ineqDC}.2 that $\mathsf D (G) \ge \mathsf d (G) + 1 \ge \big( \mathsf d (G') + \mathsf d (G/G') \big) + 1 \ge 13$.
Hence, we may assume that $G/G' \cong C^{3}_3$.
Note that $G'$ is a Sylow 2-subgroup of $G$.
If $K$ is a Sylow 3-subgroup of $G$, then $G'$ intersects trivially with $K$, and so $G = G'K$, i.e., $G \cong G' \rtimes K$.
Then, $K \cong K/(G' \cap K) \cong G'K/G' = G/G'$, and thus $G$ has a subgroup $L = G' H$ of order 36, where $H \le K$ with $|H| = 9$, whence $\mathsf D (G) \ge \mathsf D (L) \ge 11$ by Lemma~\ref{lem:dD}.

If $|G/G'| = 18$, then $G/G'$ is isomorphic to $C_{18}$, or to $C_3 \times C_6$.
Since $|G'| = 6$, $G'$ is isomorphic to $C_6$, or to $D_6$.
Thus, it follow by Lemmas~\ref{lem:ineqDC}.2 and \ref{lem:dD} that $\mathsf D (G) \ge \mathsf d (G) + 1 \ge \big( \mathsf d (G') + \mathsf d (G/G') \big) + 1 \ge 11$.

If $|G/G'| = 12$, then $G/G'$ is isomorphic to $C_{12}$, or to $C_2 \times C_6$.
Since $|G'| = 9$, (\ref{eq:order}) ensures that $G' \cong C^{2}_3$.
Thus, it follows by Lemma~\ref{lem:ineqDC}.2 that $\mathsf D (G) \ge \mathsf d (G) + 1 \ge \big( \mathsf d (G') + \mathsf d (G/G') \big) + 1 \ge 11$.

If $|G/G'| = 9$, then $G/G'$ is isomorphic to $C_9$, or to $C^{2}_3$.
Since $|G'| = 12$, by (\ref{eq:order}), $G'$ is isomorphic to one of the group in $\{ C_2 \times C_6, \, D_{12}, \, Q_{12}, \, A_4 \}$.
If $G' \ncong A_4$, then by Lemmas~\ref{lem:ineqDC}.2 and \ref{lem:dD} that $\mathsf D (G) \ge \mathsf d (G) + 1 \ge \big( \mathsf d (G') + \mathsf d (G/G') \big) + 1 \ge 11$.
We may assume that $G' \cong A_4$.
Then, by passing to $G/G'$, we can take an element $g \in G \setminus G'$ such that $g^{3} \in G'$, and thus we infer by (\ref{eq:order}) that $\ord (g) \in \{ 3, 6 \}$.
If $\ord (g) = 3$, then $G'$ intersects trivially with $\langle g \rangle$, and so $L = G' \langle g \rangle$ is a proper subgroup of $G$ of order 36.
If $\ord (g) = 6$, then since $g^{3} \in G'$, it follows that $|G' \cap \langle g \rangle | = 2$, and so $L = G' \langle g \rangle$ is a proper subgroup of $G$ of order 36.
In either case, Lemma~\ref{lem:subgroup} ensures that $\mathsf D (G) \ge 10$.

If $|G/G'| = 6$, then $G/G' \cong C_6$.
Since $|G'| = 18$, by (\ref{eq:order}), $G'$ is isomorphic to one of the groups in $\{ C_3 \times C_6, \, C_3 \times D_6, \, C^{2}_3 \rtimes C_2 \}$.
Thus, it follows by Lemmas~\ref{lem:ineqDC}.2 and \ref{lem:dD} that $\mathsf D (G) \ge \mathsf d (G) + 1 \ge \big( \mathsf d (G') + \mathsf d (G/G') \big) + 1 \ge 11$.

If $|G/G'| = 4$, then $G/G'$ is isomorphic to $C_4$, or to $C_2 \times C_2$.
Since $|G'| = 27$, $G'$ is a Sylow 3-subgroup of $G$.
If $K$ is a Sylow 2-subgroup of $G$, then $G = G'K$ with $G' \cap K = \{ 1_G \}$, i.e., $G \cong G' \rtimes K$.
Hence, $K \cong K/(G' \cap K) \cong G'K/G' = G/G'$, and so $G$ has a subgroup $L = G'C_2$ of order 54.
If $L$ is abelian, then in view of (\ref{eq:order}), $L \cong C^{2}_3 \times C_6$, and so it follows by Lemma~\ref{lem:ineqDC}.1 that $\mathsf D (G) \ge \mathsf D (L) + 1 \ge 11$.
If $L$ is non-abelian, then {\bf SUBCASE 3.1} ensures that $\mathsf D (G) \ge \mathsf D (L) \ge 10$.

\smallskip
\noindent
{\bf SUBCASE 3.3.} $i = 3$, i.e., $|G| = 2^{3} 3^{3} = 216$.
\smallskip

By Lemma~\ref{lem:perfect}, every group of order 216 is not a perfect group, and so such a group has a non-trivial proper commutator subgroup.
The possible values of the pair $\big( |G'|, |G/G'| \big)$ are as follows:
\[
  \begin{aligned}
  \big( |G'|, |G/G'| \big) \,\, \in & \,\, \big\{ (2,108), \, (3,72), \, (4,54), \, (6,36), \, (8,27), \, (9,24) \, (12,18), \, (18, 12), \, (24,9), \, (27, 8), \\
                                     & \quad (36,6), \, (54,4), \, (72, 3), \, (108, 2) \big\} \,.
  \end{aligned}
\]

If $|G'| = 36$ or $|G/G'| = 36$, then Lemma~\ref{lem:subgroup} ensures that $\mathsf D (G) \ge 10$.

Suppose that $|G'| \in \{ 108, 72, 54 \}$.
If $G'$ is non-abelian, then {\bf SUBCASES 3.2}--{\bf 3.1}, and {\bf 2.1} ensure that $\mathsf D (G) \ge \mathsf D (G') \ge 10$.
If $G'$ is abelian, then we infer by (\ref{eq:order}) that $G' \cong C_3 \times C^{2}_6$ (resp., $G' \cong C_2 \times C^{2}_6$, or $G' \cong C^{2}_3 \times C_6$) if $|G'| = 108$ (resp., $|G'| = 72$, or $|G'| = 54$).
In either case, it follows by Lemma~\ref{lem:ineqDC}.1 that $\mathsf D (G) \ge \mathsf D (G') + 1 \ge 11$.

Suppose that $|G/G'| \in \{ 108, 72, 54 \}$.
If $G/G'$ has an element of order at least 9, then Lemmas~\ref{lem:ineqDC}.3 and \ref{lem:order9} ensure that $\mathsf D (G) \ge \mathsf D (G/G') \ge 10$.
So, we may assume that every element in $G/G'$ has the order at most 8.
Since $G/G'$ is abelian, $G/G'$ is isomorphic to $C_3 \times C^{2}_6$, or to $C_2 \times C^{2}_6$, or to $C^{2}_3 \times C_6$.
In either case, we infer again by Lemma~\ref{lem:ineqDC}.3 that $\mathsf D (G) \ge \mathsf D (G/G') \ge 10$.

If $|G/G'| = 27$, then $G/G'$ is isomorphic to $C_{27}$, or to $C_3 \times C_9$, or to $C^{3}_3$.
Since $|G'| = 8$, $G'$ is isomorphic to one of the groups in $\{ C_2 \times C_4, \, C^{3}_2, \, D_8, \, Q_8 \}$.
Thus, it follows by Lemmas~\ref{lem:ineqDC}.2 and \ref{lem:dD} that $\mathsf D (G) \ge \mathsf d (G) + 1 \ge \big( \mathsf d (G') + \mathsf d (G/G') \big) + 1 \ge 10$.

If $|G/G'| = 24$, then $G/G'$ is isomorphic to $C_{24}$, or to $C_2 \times C_{12}$, or to $C^{2}_2 \times C_6$.
Since $|G'| = 9$, we infer by (\ref{eq:order}) that $G' \cong C^{2}_3$.
Thus, it follows by Lemma~\ref{lem:ineqDC}.2 that $\mathsf D (G) \ge \mathsf d (G) + 1 \ge \big( \mathsf d (G') + \mathsf d (G/G') \big) + 1 \ge 12$.

If $|G/G'| = 18$, then $G/G'$ is isomorphic to $C_{18}$, or to $C_3 \times C_6$.
Since $|G'| = 12$, by (\ref{eq:order}), $G'$ is isomorphic to one of the groups in $\{ C_2 \times C_6, D_{12}, Q_{12}, A_4 \}$.
Thus, it follows by Lemmas~\ref{lem:ineqDC}.2 and \ref{lem:dD} that $\mathsf D (G) \ge \mathsf d (G) + 1 \ge \big( \mathsf d (G') + \mathsf d (G/G') \big) + 1 \ge 12$.

If $|G/G'| = 12$, then $G/G'$ is isomorphic to $C_{12}$, or to $C_2 \times C_6$.
If $G'$ is abelian, then in view of (\ref{eq:order}), $G' \cong C_3 \times C_6$, and it follows by Lemma~\ref{lem:ineqDC}.2 that $\mathsf D (G) \ge \mathsf d (G) + 1 \ge \big( \mathsf d (G') + \mathsf d (G/G') \big) + 1 \ge 14$.
If $G'$ is non-abelian, then it follows by Lemma~\ref{lem:dD} that $\mathsf D (G) \ge \mathsf D (G') = 10$.

If $|G/G'| = 9$, then $G/G'$ is isomorphic to $C_9$, or to $C_3 \times C_3$.
If $G'$ is abelian, then in view of (\ref{eq:order}), $G' \cong C^{2}_2 \times C_6$, and it follows by Lemma~\ref{lem:ineqDC}.2 that $\mathsf D (G) \ge \mathsf d (G) + 1 \ge \big( \mathsf d (G') + \mathsf d (G/G') \big) + 1 \ge 12$.
If $G'$ is non-abelian, then it follows by Lemma~\ref{lem:dD} that $\mathsf D (G) \ge \mathsf D (G') \ge 10$.

If $|G/G'| = 8$, then $G/G'$ is isomorphic to $C_8$, or to $C_2 \times C_4$, or to $C^{3}_2$.
By (\ref{eq:order}), $G'$ is isomorphic to $C^{3}_3$, or to $H_{27}$, and so it follows by Lemmas~\ref{lem:ineqDC}.2 and \ref{lem:dD} that $\mathsf D (G) \ge \mathsf d (G) + 1 \ge \big( \mathsf d (G') + \mathsf d (G/G') \big) + 1 \ge 10$.

\smallskip
\noindent
{\bf SUBCASE 3.4.} $i = 4$, i.e., $|G| = 2^{4} 3^{3} = 432$.
\smallskip

By Lemma~\ref{lem:perfect}, every group of order 432 is not a perfect group, and so such a group has a non-trivial proper commutator subgroup.
The possible values of the pair $\big( |G'|, |G/G'| \big)$ are as follows:
\[
  \begin{aligned}
  \big( |G'|, |G/G'| \big) \,\, \in & \,\, \big\{ (2,216), \, (3,144), \, (4,108), \, (6,72), \, (8,54), \, (9,48) \, (12,36), \, (16,27), \, (18, 24), \, (24,18), \\
                                     & \quad (27,16), \, (36,12), \, (48,9), \, (54,8), \, (72, 6), \, (108, 4), \, (144,3), \, (216, 2) \big\} \,.
  \end{aligned}
\]

If $|G'| = 36$ or $|G/G'| = 36$, then Lemma~\ref{lem:subgroup} ensures that $\mathsf D (G) \ge 10$.

Suppose that $|G'| \in \{ 216, 144, 108, 72, 54, 48 \}$.
If $G'$ is non-abelian, then {\bf SUBCASES 3.3}--{\bf 3.1}, {\bf 2.2}--{\bf 2.1}, and {\bf CASE 1} ensure that $\mathsf D (G) \ge \mathsf D (G') \ge 10$.
If $G'$ is abelian, then in view of (\ref{eq:order}), $G'$ is isomorphic to $C^{3}_6$ (resp., $C^{2}_2 \times C^{2}_6$, or $C_3 \times C^{2}_6$, or $C_2 \times C^{2}_6$, or $C^{2}_3 \times C_6$, or $C^{3}_2 \times C_6$) if $|G'| = 216$ (resp., $|G'| = 144$, or $|G'| = 108$, or $|G'| = 72$, or $|G'| = 54$, or $|G'| = 48$).
Thus, it follows by Lemma~\ref{lem:ineqDC}.2 that $\mathsf D (G) \ge \mathsf d (G) + 1 \ge \big( \mathsf d (G') + \mathsf d (G/G') \big) + 1 \ge 10$.

Suppose that $|G/G'| \in \{ 216, 144, 108, 72, 54, 48 \}$.
If $G/G'$ has an element of order at least 9, then by Lemmas~\ref{lem:ineqDC}.3 and \ref{lem:order9}, $\mathsf D (G) \ge \mathsf D (G/G') \ge 10$, and so we may assume that the order of any element of $G/G'$ is at most 8.
Since $G/G'$ is abelian, the same argument as the one used in the paragraph just above ensures that $\mathsf D (G) \ge \mathsf d (G) + 1 \ge \big( \mathsf d (G') + \mathsf d (G/G') \big) + 1 \ge 10$.

If $|G/G'| = 27$, then $G/G'$ is isomorphic to $C_{27}$, or to $C_3 \times C_9$, or to $C^{3}_3$.
If $G'$ is abelian, then by (\ref{eq:order}), $G'$ is isomorphic to one of the groups in $\{ C^{2}_4, \, C^{2}_2 \times C_4, \, C_2 \times C_8, \, C^{4}_2 \}$, and it follows by Lemma~\ref{lem:ineqDC}.2 that $\mathsf D (G) \ge \mathsf d (G) + 1 \ge \big( \mathsf d (G') + \mathsf d (G/G') \big) + 1 \ge 11$.
If $G'$ is non-abelian, then we obtain that $\mathsf d (G') \ge 5$ by Lemma~\ref{lem:dD}, and thus it follows again by Lemma~\ref{lem:ineqDC}.2 that $\mathsf D (G) \ge \mathsf d (G) + 1 \ge \big( \mathsf d (G') + \mathsf d (G/G') \big) + 1 \ge 12$.

If $|G/G'| = 24$, then $G/G'$ is isomorphic to $C_{24}$, or to $C_2 \times C_{12}$, or to $C^{2}_2 \times C_6$.
If $G'$ is abelian, then by (\ref{eq:order}), $G' \cong C_3 \times C_6$, and it follows by Lemma~\ref{lem:ineqDC}.2 that $\mathsf D (G) \ge \mathsf d (G) + 1 \ge \big( \mathsf d (G') + \mathsf d (G/G') \big) + 1 \ge 15$.
If $G'$ is non-abelian, then it follows by Lemma~\ref{lem:dD} that $\mathsf D (G) \ge \mathsf D (G') \ge 10$.

If $|G/G'| = 18$, then $G/G'$ is isomorphic to $C_{18}$, or to $C_3 \times C_6$.
If $G'$ is abelian, then by (\ref{eq:order}), $G' \cong C^{2}_2 \times C_6$, and so $\mathsf d (G') = 7$ by Lemma~\ref{lem:abD}.(c).
If $G'$ is non-abelian, then it follows by Lemma~\ref{lem:dD} that $\mathsf d (G') \ge 6$.
In either case, it follows by Lemma~\ref{lem:ineqDC}.2 that $\mathsf D (G) \ge \mathsf d (G) + 1 \ge \big( \mathsf d (G') + \mathsf d (G/G') \big) + 1 \ge 14$.

If $|G/G'| = 16$, then $G/G'$ is isomorphic to one of the groups in $\{ C_{16}, \, C^{2}_4, \, C_2 \times C_8, \, C^{2}_2 \times C_4, \, C^{4}_2 \}$.
By (\ref{eq:order}), $G'$ is isomorphic to $C^{3}_3$, or to $H_{27}$, and so it follows by Lemmas~\ref{lem:ineqDC}.2 and \ref{lem:dD} that $\mathsf D (G) \ge \mathsf d (G) + 1 \ge \big( \mathsf d (G') + \mathsf d (G/G') \big) + 1 \ge 11$.

\smallskip
\noindent
{\bf CASE 4.} $j = 0$, $k = 1$, and $\ell = 0$, i.e., $|G| = 2^{i} 5^{1}$.
\smallskip

Then, since $|G| > 42$, we only have the case where $i = 4$, and hence $|G| = 2^{4} 5^{1} = 80$.
By Lemma~\ref{lem:perfect}, every group of order 80 is not a perfect group, and so such a group has a non-trivial proper commutator subgroup.
The possible values of the pair $\big( |G'|, |G/G'| \big)$ are as follows:
\[
  \big( |G'|, |G/G'| \big) \,\, \in \,\, \big\{ (2,40), \, (4,20), \, (5,16), \, (8,10), \, (10,8), \, (16,5), \, (20,4), \, (40,2)  \big\} \,.
\]

If $|G'|, |G/G'| \in \{ 40, 20 \}$, then Lemma~\ref{lem:subgroup} ensures that $\mathsf D (G) \ge 10$.

If $|G/G'| = 16$, then $|G'| = 5$, and by Cauchy's Theorem, $G$ has a subgroup $K$ of order 2.
Since $G'$ intersects trivially with $K$, $L = G' K$ is a subgroup of $G$ of order 10, and by (\ref{eq:order}), we must have that $L \cong D_{10}$.
Thus, Lemma~\ref{lem:nabD} implies that $\mathsf D (G) \ge \mathsf D (L) = 10$.

If $|G/G'| = 10$, then Lemmas~\ref{lem:ineqDC}.3 and \ref{lem:order9} imply that $\mathsf D (G) \ge \mathsf D (G/G') \ge 10$.

If $|G/G'|=8$, then $|G'| = 10$, and by (\ref{eq:order}), we must have that $G' \cong D_{10}$.
Thus, Lemma~\ref{lem:nabD} ensures that $\mathsf D (G) \ge \mathsf D (G') = 10$.

If $|G/G'| = 5$, then $G/G' \cong C_5$, and since $|G'| = 16$, $G'$ is a unique Sylow 2-subgroup of $G$. whence $G'$ is isomorphic to one of the groups $H$ listed in {\bf CASE 1}.
If $G' \in \{ C_2 \times C_8, \, M_{16}, \, D_{16}, \, SD_{16}, \, Q_{16} \}$, then in view of {\bf CASE 1}, we obtain that $\mathsf D (G) \ge 10$.
Now let $K$ be a Sylow 5-subgroup of $G$.
Then, $G'$ intersects trivially with $K$, and so $G \cong G' \rtimes K$.
If $G' \in \{ C^{2}_2 \times C_4, \, C^{2}_4, \, C^{2}_2 \rtimes C_4, \, C_2 \times D_8, \, C_2 \times Q_8, \, (C_2 \times C_4) \rtimes C_2, \, C_4 \rtimes C_4 \}$, then in view of {\bf SUBCASES 1.2}--{\bf 1.8}, we note that $\Aut (G')$ has no subgroup of order 5.
This means that, for any homomorphism $\varphi$ from $K$ to $\Aut (G')$, $\varphi [K]$ is a trivial subgroup of $\Aut (G')$, and thus we infer that $G \cong G' \rtimes K \cong G' \times K$.
Hence, $G$ must have an element of order 10, a contradiction to (\ref{eq:order}).
Suppose now that $G' \cong C^{4}_2$.
Then, $\Aut (G') \cong \GL (4,2)$ and $|\Aut (G')| = 2^{6} 3^{2} 5^{1} 7^{1}$.
Let $A \in \GL (4,2)$ be such that $A^{5} = I$.
Then, the minimal polynomial of $A$ divides $x^{5} - 1$, and $x^{5}-1 = (x-1)(x^{4} + x^{3} + x^{2} + x + 1) \in \mathbb Z_2 [x]$ (note that $x^{4} + x^{3} + x^{2} + x + 1$ is irreducible over $\mathbb Z_2$).
Under the constraints of the rational canonical form, we obtain the following permissible lists of invariant factors:
\begin{enumerate}
\item[(i)] $(x-1)$, $(x-1)$, $(x-1)$, $(x-1)$

\smallskip
\item[(ii)] $(x^{4} + x^{3} + x^{2} + x + 1)$
\end{enumerate}
So, (i) corresponds to the identity matrix in $\GL (4,2)$, and (ii) corresponds to the matrix $B = \begin{bmatrix} 0 & 0 & 0 & 1 \\ 1 & 0 & 0 & 1 \\ 0 & 1 & 0 & 1 \\ 0 & 0 & 1 & 1 \end{bmatrix}$.
Let $G' \cong C^{4}_2 \cong \langle a \rangle \times \langle b \rangle \times \langle c \rangle \times \langle d \rangle$, $K \cong C_5 \cong \langle e \rangle$, and $\varphi \colon K \to \Aut (G')$ be a homomorphism induced from the action of $K$ on $G'$ by conjugation.
Since $G$ is non-abelian, $\varphi [K] \le \Aut (G')$ is a non-trivial subgroup of order 5, and thus $\varphi [K] \cong \langle B \rangle$.
Hence, $G \cong \langle a, b, c, d, e \mid a^{2} = b^{2} = c^{2} = d^{2} = e^{5} = 1_G, \mbox{ and } ab = ba, ac = ca, ad = da, bc = cb, bd = db, cd = dc, eae^{-1} = d, ebe^{-1} = ad, ece^{-1} = bd, ede^{-1} = cd \rangle$.
We now consider the sequence
\[
  S = a^{[2]} \bdot b \bdot c \bdot (bd) \bdot e^{[5]} \in \mathcal F (G) \,.
\]
Note that $1_G = e a b c e e e e a (bd)$, so that $S$ is a product-one sequence.
Suppose that $S = S_1 \bdot S_2$ for some $S_1, S_2 \in \mathcal B (G)$ with $e \mid S_1$.
Since $e$ is independent with $a, b, c, d$ in $G$ and $\ord (e) = 5$, it follows that $e^{[5]} \mid S_1$, so that $S_2 \mid a^{[2]} \bdot b \bdot c \bdot (bd)$.
If $S_2$ is non-trivial, then, in view of $S_2$ as a product-one sequence over $C^{4}_2$, we only have that $S_2 = a^{[2]}$, so that $S_1 = b \bdot c \bdot (bd) \bdot e^{[5]}$.
Since $S_1$ is product-one, we infer that $bd \in \pi \big( S_1 \bdot (bd)^{[-1]} \big) = \pi (b \bdot c \bdot e^{[5]})$.
Let $T^{*} = g_1 \bdot \ldots \bdot g_7$ be the ordered sequence of $b \bdot c \bdot e^{[5]}$ such that $\pi^{*} (T^{*}) = g_1 \cdots g_7 = bd$.
If $g_1$ is $b$ (resp., $c$), then $d$ (resp., $bcd$) belongs to $\pi (c \bdot e^{[5]}) = \{ c, bd, ac, b, ad \}$ (resp., $\pi (b \bdot e^{[5]}) = \{ b, ad, c, bd, ac \}$), a contradiction.
Thus, we must have that $g_1 = e$, and we also have by symmetry that $g_7 = e$.
Since $e g_2 \cdots g_6 e = bd$, it follows by the relations of generators that $e^{2} g_2 \cdots g_6 = ac$.
If $g_6$ is $c$ (resp., $b$), then $a$ (resp., $abc$) belongs to $\pi (b \bdot e^{[5]})$ (resp., $\pi (c \bdot e^{[5]}))$, a contradiction.
Hence, $g_6 = e$, and we obtain that $e^{3} g_2 \cdots g_5 = b$.
If $g_5$ is $b$ (resp., $c$), then $1_G$ (resp., $bc$) belongs to $\pi (c \bdot e^{[5]})$ (resp., $\pi (b \bdot e^{[5]}))$, again a contradiction.
Hence, $g_5 = e$, and we obtain that $e^{4}g_2 g_3 g_4 = ad$.
But, $ce = ead = g_2 g_3 g_4 \in \pi (b \bdot c \bdot e) = \{ bce, bde, acde, abe \}$, a contradiction.
Therefore, $S_2$ must be a trivial sequence, and so $S = S_1$ is a minimal product-one sequence, whence $\mathsf D (G) \ge |S| = 10$.

\smallskip
\noindent
{\bf CASE 5.} $j = k = 1$ and $\ell = 0$, i.e., $|G| = 2^{i} 3^{1} 5^{1}$.
\smallskip

Then, since $|G| > 42$, it follows that $i \in [2,4]$.

\smallskip
\noindent
{\bf SUBCASE 5.1.} $i = 2$, i.e., $|G| = 2^{2} 3^{1} 5^{1} = 60$.
\smallskip

By Lemma~\ref{lem:perfect}, there exists a perfect group of order 60.
If $G$ is not a perfect group, then by \cite[Proposition 4.21]{Du-Fo04}, $G$ has a unique Sylow 5-subgroup $P$.
By Cauchy's Theorem, $G$ has a subgroup $H$ of order 3.
Since $P$ intersects trivially with $H$, it follows that $G$ has a subgroup $PH$ of order 15, which must be isomorphic to $C_{15}$, a contradiction to (\ref{eq:order}).
Thus, $G$ must be a perfect group, i.e., $G \cong A_5$.
Then, $G$ has a subgroup $H$ isomorphic to $D_{10}$, and by Lemma~\ref{lem:nabD}, $\mathsf D (G) \ge \mathsf D (H) = 10$.

\smallskip
\noindent
{\bf SUBCASE 5.2.} $i = 3$, i.e., $|G| = 2^{3} 3^{1} 5^{1} = 120$.
\smallskip

By Lemma~\ref{lem:perfect}, there exists a perfect group of order 120, and thus if $G$ is such a perfect group of order 120, then $G$ has a normal subgroup $H$ of order 2 such that $G/H \cong A_5$.
Let $\varphi \colon G \to G/H \cong A_5$ be the canonical epimorphism.
Since $A_5$ has a subgroup isomorphic to $D_{10}$, we infer that $G/H$ has a subgroup $K/H$ isomorphic to $D_{10}$, where $K$ is a subgroup of $G$ with $H \subseteq K$.
Thus, $K$ is a (non-abelian) subgroup of order $|K| = |D_{10}| |H| = 20$, whence it follows by Lemma~\ref{lem:dD} that $\mathsf D (G) \ge \mathsf D (K) \ge 10$.

Suppose now that $G$ is not a perfect group, so that $G$ has a non-trivial proper subgroup $G'$.
The possible values of the pair $\big( |G'|, |G/G'| \big)$ are as follows:
\[
  \begin{aligned}
    \big( |G'|, |G/G'| \big) \,\, \in & \,\, \big\{ (2,60), \, (3,40), \, (4,30), \, (5,24), \, (6,20), \, (8,15), \, (10,12), \, (12,10), \, (15,8), \, (20,6), \\
                                       & \quad (24,5), \,(30,4), \, (40, 3), \, (60,2)  \big\} \,.
  \end{aligned}
\]

If $|G'|, |G/G'| \in \{ 40, 30, 20, 15, 10 \}$, then Lemma~\ref{lem:subgroup} ensures that $\mathsf D (G) \ge 10$.

If $|G'| = 60$, then by (\ref{eq:order}), $G'$ must be non-abelian, and hence {\bf SUBCASE 5.1} ensures that $\mathsf D (G) \ge \mathsf D (G') \ge 10$.

If $|G/G'| = 60$, then since $G/G'$ is abelian, $G/G'$ has an element of order at least 10, and hence Lemmas~\ref{lem:ineqDC}.3 and \ref{lem:order9} imply that $\mathsf D (G) \ge \mathsf D (G/G') \ge 10$.

If $|G/G'| = 24$, then  $G/G'$ is isomorphic to $C_{24}$, or to $C_2 \times C_{12}$, or to $C^{2}_2 \times C_6$.
Since $|G'| = 5$, $G' \cong C_5$, and it follows by Lemma~\ref{lem:ineqDC}.2 that $\mathsf D (G) \ge \mathsf d (G) + 1 \ge \big( \mathsf d (G') + \mathsf d (G/G') \big) + 1 \ge 12$.

If $|G/G'| = 5$, then $G/G' \cong C_5$ and $|G'| = 24$.
If $G'$ is non-abelian, then it follows by Lemma~\ref{lem:dD} that $\mathsf D (G) \ge \mathsf D (G') \ge 10$.
If $G'$ is abelian, then by (\ref{eq:order}), we must have that $G' \cong C^{2}_2 \times C_6$.
Thus, it follows by Lemma~\ref{lem:ineqDC}.2 that $\mathsf D (G) \ge \mathsf d (G) + 1 \ge \big( \mathsf d (G') + \mathsf d (G/G') \big) + 1 \ge 12$.

\smallskip
\noindent
{\bf SUBCASE 5.3.} $i = 4$, i.e., $|G| = 2^{4} 3^{1} 5^{1} = 240$.
\smallskip

By Lemma~\ref{lem:perfect}, every group of order 240 is not a perfect group, and so such a group has a non-trivial proper commutator subgroup.
The possible values of the pair $\big( |G'|, |G/G'| \big)$ are as follows:
\[
  \begin{aligned}
    \big( |G'|, |G/G'| \big) \,\, \in & \,\, \big\{ (2,120), \, (3,80), \, (4, 60), \, (5, 48), \, (6, 40), \, (8,30), \, (10, 24), \, (12, 20), \, (15, 16), \, (16,15),   \\
                                       & \quad (20,12), \, (24,10), \, (30,8), \, (40,6), \, (48,5), \, (60, 4), \, (80,3), \, (120,2) \big\} \,.
  \end{aligned}
\]

If $|G'|, |G/G'| \in \{ 40, 30, 20, 15, 10 \}$, then Lemma~\ref{lem:subgroup} ensures that $\mathsf D (G) \ge 10$.

Suppose that $|G'| \in \{ 120, 80, 60, 48 \}$.
If $G'$ is non-abelian, then {\bf SUBCASES 5.2}--{\bf 5.1}, and {\bf CASES 4}, {\bf 1} ensure that $\mathsf D (G) \ge \mathsf D (G') \ge 10$.
If $G'$ is abelian, then by (\ref{eq:order}), we only have that $G' \cong C^{3}_2 \times C_6$ with $|G'| = 48$. Since $|G/G'| = 5$, $G/G' \cong C_5$, and thus it follows by Lemma~\ref{lem:ineqDC}.2 that $\mathsf D (G) \ge \mathsf d (G) + 1 \ge \big( \mathsf d (G') + \mathsf d (G/G') \big) + 1 \ge 13$.

Suppose that $|G/G'| \in \{ 120, 80, 60, 48 \}$.
If $G/G'$ has an element of order at least 9, then by Lemmas~\ref{lem:ineqDC}.3 and \ref{lem:order9} ensure that $\mathsf D (G) \ge \mathsf D (G/G') \ge 10$.
So, we assume that every element if $G/G'$ has the order at most 8.
Since $G/G'$ is abelian, we only have that $G/G' \cong C^{3}_2 \times C_6$ with $|G/G'| = 48$. Since $|G'| = 5$, $G' \cong C_5$, and it follows by Lemma~\ref{lem:ineqDC}.2 that $\mathsf D (G) \ge \mathsf d (G) + 1 \ge \big( \mathsf d (G') + \mathsf d (G/G') \big) + 1 \ge 13$.

\smallskip
\noindent
{\bf CASE 6.} $j = 2$, $k = 1$, and $\ell = 0$, i.e., $|G| = 2^{i} 3^{2} 5^{1}$.
\smallskip

If $i = 0$, then $|G| = 3^{2} 5^{1} = 45$, and by Sylow's Theorem, we infer that $G$ has a unique Sylow 3-subgroup $P$ and a unique Sylow 5-subgroup $Q$.
Since $P$ intersects trivially with $Q$, it follows that $G \cong P \rtimes Q \cong P \times Q$.
Note that both $P$ and $Q$ are abelian, so that $G$ is also abelian.
Thus, it follows that $i \in [1,4]$.

\smallskip
\noindent
{\bf SUBCASE 6.1.} $i = 1$, i.e., $|G| = 2^{1} 3^{2} 5^{1} = 90$.
\smallskip

By Lemma~\ref{lem:perfect}, every group of order 90 is not a perfect group, and so such a group has a non-trivial proper commutator subgroup.
The possible values of the pair $\big( |G'|, |G/G'| \big)$ are as follows:
\[
  \big( |G'|, |G/G'| \big) \,\, \in \,\, \big\{ (2,45), \, (3, 30), \, (5, 18), \, (6,15), \, (9,10), \, (10,9), \, (15, 6), \, (18, 5), \, (30,3), \, (45,2) \big\} \,.
\]

If $|G'|, |G/G'| \in \{ 45, 30, 15, 10 \}$, then Lemma~\ref{lem:subgroup} ensures that $\mathsf D (G) \ge 10$.

If $|G/G'| = 18$, then $G/G'$ is isomorphic to $C_{18}$, or to $C_3 \times C_6$.
Since $|G'| = 5$, $G' \cong C_5$, and it follows by Lemma~\ref{lem:ineqDC}.2 that $\mathsf D (G) \ge \mathsf d (G) + 1 \ge \big( \mathsf d (G') + \mathsf d (G/G') \big) + 1 \ge 12$.

If $|G/G'| = 5$, then $|G'| = 18$.
If $G'$ is non-abelian, then Lemmas~\ref{lem:dD} that $\mathsf D (G) \ge \mathsf D (G') \ge 10$.
If $G'$ is abelian, then by (\ref{eq:order}), we only have that $G' \cong C_3 \times C_6$, and it follows by Lemma~\ref{lem:ineqDC}.2 that $\mathsf D (G) \ge \mathsf d (G) + 1 \ge  \big( \mathsf d (G') + \mathsf d (G/G') \big) + 1 \ge 13$.

\smallskip
\noindent
{\bf SUBCASE 6.2.} $i = 2$, i.e., $|G| = 2^{2} 3^{2} 5^{1} = 180$.
\smallskip

By Lemma~\ref{lem:perfect}, every group of order 180 is not a perfect group, and so such a group has a non-trivial proper commutator subgroup.
The possible values of the pair $\big( |G'|, |G/G'| \big)$ are as follows:
\[
  \begin{aligned}
    \big( |G'|, |G/G'| \big) \,\, \in & \,\, \big\{ (2,90), \, (3, 60), \, (4,45), \, (5, 36), \, (6,30), \, (9,20), \, (10,18), \, (12,15), \, (15, 12), \, (18,10) \\
                                       & \quad (20,9), \, (30,6), \, (36, 5), \, (45,4), \, (60,3), \, (90,2) \big\} \,.
  \end{aligned}
\]

If $|G'|, |G/G'| \in \{ 45, 36, 30, 20, 15, 10 \}$, then Lemma~\ref{lem:subgroup} ensures that $\mathsf D (G) \ge 10$.

Suppose that $|G'| \in \{ 90, 60 \}$.
If $G'$ is abelian, then it is easy to see that $G'$ has an element of order at least 10, a contradiction to (\ref{eq:order}).
Thus, $G'$ is non-abelian, and hence {\bf SUBCASES 6.1} and {\bf 5.1} ensure that $\mathsf D (G) \ge \mathsf D (G') \ge 10$.

Suppose that $|G/G'| \in \{ 90, 60 \}$.
Then, it is east to see that $G/G'$ has an element of order at least 10, and hence Lemmas~\ref{lem:ineqDC}.3 and \ref{lem:order9} imply that $\mathsf D (G) \ge \mathsf D (G/G') \ge 10$.

\smallskip
\noindent
{\bf SUBCASE 6.3.} $i = 3$, i.e., $|G| = 2^{3} 3^{2} 5^{1} = 360$.
\smallskip

By Lemma~\ref{lem:perfect}, there exists a perfect group of order 360, and thus if $G$ is such a perfect group of order 360, then we must have that $G \cong A_6$.
Hence, $G$ has a subgroup $H$ isomorphic to $A_5$, and thus it follows by {\bf SUBCASE 5.1} that $\mathsf D (G) \ge \mathsf D (A_5) \ge 10$.
Except in the case where $G \cong A_6$, we infer that $G$ has a non-trivial proper subgroup $G'$.
The possible values of the pair $\big( |G'|, |G/G'| \big)$ are as follows:
\[
  \begin{aligned}
    \big( |G'|, |G/G'| \big) \,\, \in & \,\, \big\{ (2,180), \, (3, 120), \, (4,90), \, (5,72), \, (6,60), \, (8,45), \, (9,40), \, (10,36), \, (12,30), \, (15,24), \\
                                       & \quad (18, 20), \, (20,18), \, (24,15), \, (30,12), \, (36, 10), \, (40,9), \, (45,8), \, (60,6), \, (72,5), \, (90,4), \\
                                       & \quad (120,3), \, (180, 2) \big\} \,.
  \end{aligned}
\]

If $|G'|, |G/G'| \in \{ 45, 40, 36, 30, 20, 15, 10 \}$, then Lemma~\ref{lem:subgroup} ensures that $\mathsf D (G) \ge 10$.

Suppose that $|G'| \in  \{ 180, 120, 90, 72, 60 \}$.
If $G'$ is non-abelian, then {\bf SUBCASES 6.2}--{\bf 6.1}, {\bf 5.2}--{\bf 5.1}, and {\bf 2.1} ensure that $\mathsf D (G) \ge \mathsf D (G') \ge 10$.
If $G'$ is abelian, then by (\ref{eq:order}), we must have that $G' \cong C_2 \times C^{2}_6$ with $|G'| = 72$, and hence $\mathsf D (G) \ge \mathsf D (G') \ge 12$.

Suppose that $|G/G'| \in \{ 180, 120, 90, 72, 60 \}$.
If $G/G' \cong C_2 \times C^{2}_6$ with $|G/G'| = 72$, then $\mathsf D (G) \ge 12$.
Otherwise, we infer that $G/G'$ has an element of order at least 10.
In either case, it follows by Lemmas~\ref{lem:ineqDC}.3 and \ref{lem:order9} that $\mathsf D (G) \ge \mathsf D (G/G') \ge 10$.

\smallskip
\noindent
{\bf SUBCASE 6.4.} $i = 4$, i.e., $|G| = 2^{4} 3^{2} 5^{1} = 720$.
\smallskip

By Lemma~\ref{lem:perfect}, there exists a perfect group of order 720, and thus if $G$ is such a perfect group of order 720, then $G$ has a normal subgroup $H$ of order 2 such that $G/H \cong A_6$.
Let $\varphi \colon G \to G/H \cong A_6$ be the canonical epimorphism.
Since $A_6$ has a subgroup isomorphic to $A_5$, we infer that $G/H$ has a subgroup $K/H$ isomorphic to $A_5$, where $K$ is a subgroup of $G$ with $H \subseteq K$.
Thus, $K$ is a (non-abelian) subgroup of order $|K| = |A_5| |H| = 120$, whence it follows by {\bf SUBCASE 5.2} that $\mathsf D (G) \ge \mathsf D (K) \ge 10$.
Suppose now that $G$ is not a perfect group, so that $G$ has a non-trivial proper subgroup $G'$.
The possible values of the pair $\big( |G'|, |G/G'| \big)$ are as follows:
\[
  \begin{aligned}
    \big( |G'|, |G/G'| \big) \,\, \in & \,\, \big\{ (2,360), \, (3,240), \, (4,180), \, (5,144), \, (6,120), \, (8,90), \, (9,80), \, (10,72), \, (12,60), \, (15,48), \\
                                       & \quad (16,45), \, (18,40), \, (20,36), \, (24,30), \, (30,24), \, (36,20), \, (40,18), \, (45,16), \, (48,15), \\
                                       & \quad (60,12), \, (72,10), \, (80,9), \, (90,8), \, (120, 6), \, (144,5), \, (180,4), \, (240,3), \, (360,2) \big\} \,.
  \end{aligned}
\]

If $|G'|, |G/G'| \in \{ 45, 40, 36, 30, 20, 15, 10 \}$, then Lemma~\ref{lem:subgroup} ensures that $\mathsf D (G) \ge 10$.

Suppose that $|G'| \in \{ 360, 240, 180, 144, 120, 90, 80, 60 \}$.
If $G'$ is non-abelian, then {\bf SUBCASES 6.3}--{\bf 6.1}, {\bf 5.3}--{\bf 5.1}, {\bf 2.2}, and {\bf CASE 4} ensure that $\mathsf D (G) \ge \mathsf D (G') \ge 10$.
If $G'$ is abelian, then by (\ref{eq:order}), we only have that $G' \cong C^{2}_2 \times C^{2}_6$ with $|G'| = 144$, whence $\mathsf D (G) \ge \mathsf D (G') \ge 13$.

Suppose that $|G/G'| \in \{ 360, 240, 180, 144, 120, 90, 80, 60 \}$.
If $G/G' \cong C^{2}_2 \times C^{2}_6$ with $|G/G'| = 144$, then $\mathsf D (G/G') \ge 13$.
Otherwise, $G/G'$ has an element of order at least 10.
In either case, it follows by Lemmas~\ref{lem:ineqDC}.3 and \ref{lem:order9} imply that $\mathsf D (G) \ge \mathsf D (G/G') \ge 10$.

\smallskip
\noindent
{\bf CASE 7.} $j = 3$, $k = 1$, and $\ell = 0$, i.e., $|G| = 2^{i} 3^{3} 5^{1}$.
\smallskip

Then, since $|G| > 42$, it follows that $i \in [0,4]$.

\smallskip
\noindent
{\bf SUBCASE 7.1.} $i = 0$, i.e., $|G| = 3^{3} 5^{1} = 135$.
\smallskip

By Sylow's Theorem, we infer that $G$ has a unique Sylow 3-subgroup $P$ and a unique Sylow 5-subgroup $Q$.
Since $P$ intersects trivially with $Q$, it follows that $G \cong P \rtimes Q \cong P \times Q$.
Since $Q \cong C_5$ and $G$ is non-abelian, (\ref{eq:order}) implies that $P \cong H_{27}$, so that $G \cong C_5 \times H_{27}$.
Thus, it follows by Lemmas~\ref{lem:ineqDC}.2 and \ref{lem:dD} that $\mathsf D (G) \ge \mathsf d (G) + 1 \ge \big( \mathsf d (P) + \mathsf d (G/P) \big) + 1 = 11$.

\smallskip
\noindent
{\bf SUBCASE 7.2.} $i = 1$, i.e., $|G| = 2^{1} 3^{3} 5^{1} = 270$.
\smallskip

By Lemma~\ref{lem:perfect}, every group of order 270 is not a perfect group, and so such a group has a non-trivial proper subgroup $G'$.
The possible values of the pair $\big( |G'|, |G/G'| \big)$ are as follows:
\[
  \begin{aligned}
    \big( |G'|, |G/G'| \big) \,\, \in & \,\, \big\{ (2,135), \, (3, 90), \, (5, 54), \, (6,45), \, (9,30), \, (10,27), \, (15, 18), \, (18,15), \, (27,10), \\
                                       & \quad (30,9), \, (45,6), \, (54,5), \, (90,3), \, (135,2) \big\} \,.
  \end{aligned}
\]

If $|G'|, |G/G'| \in \{ 45, 30, 15, 10 \}$, then Lemma~\ref{lem:subgroup} ensures that $\mathsf D (G) \ge 10$.

Suppose that $|G'| \in \{ 135, 90, 54 \}$.
If $G'$ is non-abelian, then {\bf SUBCASES 7.1}, {\bf 6.1}, and {\bf 3.1} ensure that $\mathsf D (G) \ge \mathsf D (G') \ge 10$.
If $G'$ is abelian, then by (\ref{eq:order}), we only have that $G' \cong C^{2}_3 \times C_6$ with $|G'| = 54$, whence $\mathsf D (G) \ge \mathsf D (G') \ge 10$.

Suppose that $|G/G'| \in \{ 135, 90, 54 \}$.
If $G/G' \cong C^{2}_3 \times C_6$ with $|G/G'| = 54$, then $\mathsf D (G/G') \ge 10$.
Otherwise, we infer that $G/G'$ has an element of order at least 10.
In either case, it follows by Lemmas~\ref{lem:ineqDC}.3 and \ref{lem:order9} that $\mathsf D (G) \ge \mathsf D (G/G') \ge 10$.

\smallskip
\noindent
{\bf SUBCASE 7.3.} $i = 2$, i.e., $|G| = 2^{2} 3^{3} 5^{1} = 540$.
\smallskip

By Lemma~\ref{lem:perfect}, every group of order 540 is not a perfect group, and so such a group has a non-trivial proper subgroup $G'$.
The possible values of the pair $\big( |G'|, |G/G'| \big)$ are as follows:
\[
  \begin{aligned}
    \big( |G'|, |G/G'| \big) \,\, \in & \,\, \big\{ (2,270), \, (3,180), \, (4,135), \, (5,108), \, (6,90), \, (9,60), \, (10,54), \, (12,45), \, (15, 36), \, (18,30), \\
                                       & \quad (20,27), \, (27,20), \, (30,18), \, (36,15), \, (45,12), \, (54,10), \, (60,9), \, (90,6), \, (108,5), \\
                                       & \quad (135,4), \, (180,3), \, (270,2) \big\} \,.
  \end{aligned}
\]

If $|G'|, |G/G'| \in \{ 45, 36, 30, 20, 15, 10 \}$, then Lemma~\ref{lem:subgroup} ensures that $\mathsf D (G) \ge 10$.

Suppose that $|G'| \in \{ 270, 180, 135, 108, 90, 60 \}$.
If $G'$ is non-abelian, then {\bf SUBCASES 7.2}--{\bf 7.1}, {\bf 6.2}--{\bf 6.1}, {\bf 5.1}, and {\bf 3.2} ensure that $\mathsf D (G) \ge \mathsf D (G') \ge 10$.
If $G'$ is abelian, then by (\ref{eq:order}), we only have that $G' \cong C_3 \times C^{2}_6$ with $|G'| = 108$, whence $\mathsf D (G) \ge \mathsf D (G') \ge 13$.

Suppose that $|G/G'| \in \{ 270, 180, 135, 108, 90, 60 \}$.
If $G/G' \cong C_3 \times C^{2}_6$ with $|G/G'| = 108$, then $\mathsf D (G/G') \ge 13$.
Otherwise, we infer that $G/G'$ has an elements of order at least 10.
In either case, it follows by Lemmas~\ref{lem:ineqDC}.3 and \ref{lem:order9} that $\mathsf D (G) \ge \mathsf D (G/G') \ge 10$.

\smallskip
\noindent
{\bf SUBCASE 7.4.} $i = 3$, i.e., $|G| = 2^{3} 3^{3} 5^{1} = 1080$.
\smallskip

By Lemma~\ref{lem:perfect}, there exists a perfect group of order 1080, and thus if $G$ is such a perfect group of order 1080, then $G$ has a subgroup $H$ of order 3 such that $G/H \cong A_6$.
Let $\varphi \colon G \to G/H \cong A_6$ be the canonical epimorphism.
Since $A_6$ has a subgroup isomorphic to $A_5$, we infer that $G/H$ has a subgroup $K/H$ isomorphic to $A_5$, where $K$ is a subgroup of $G$ with $H \subseteq K$.
Thus, $K$ is a (non-abelian) subgroup of order $|K| = |A_5||H| = 180$, whence it follows by {\bf SUBCASE 6.2} that $\mathsf D (G) \ge \mathsf D (K) \ge 10$.
Suppose now that $G$ is not a perfect group, so that $G$ has a non-trivial proper subgroup $G'$.
The possible values of the pair $\big( |G'|, |G/G'| \big)$ are as follows:
\[
  \begin{aligned}
    \big( |G'|, |G/G'| \big) \,\, \in & \,\, \big\{ (2,540), \, (3,360), \, (4,270), \, (5,216), \, (6,180), \, (8,135), \, (9,120), \, (10,108), \, (12,90), \\
                                       & \quad (15,72), \, (18,60), \, (20,54), \, (24,45), \, (27,40), \, (30,36), \, (36,30), \, (40,27), \, (45,24), \\
                                       & \quad (54,20), \, (60,18), \, (72,15), \, (90,12), \, (108,10), \, (120,9), \, (135,8), \, (180,6), \, (216,5), \\
                                       & \quad (270,4), \, (360,3), \, (540,2) \big\} \,.
  \end{aligned}
\]

If $|G'|, |G/G'| \in \{ 45, 40, 36, 30, 20, 15, 10 \}$, then Lemma~\ref{lem:subgroup} ensures that $\mathsf D (G) \ge 10$.

Suppose that $|G'| \in \{ 540, 360, 270, 216, 180, 135, 120, 90, 60 \}$.
If $G'$ is non-abelian, then {\bf SUBCASES 7.3}--{\bf 7.1}, {\bf 6.3}--{\bf 6.1}, {\bf 5.2}--{\bf 5.1}, and {\bf 3.3} ensure that $\mathsf D (G) \ge \mathsf D (G') \ge 10$.
If $G'$ is abelian, then by (\ref{eq:order}), we only have that $G' \cong C^{3}_6$ with $|G'| = 216$, whence $\mathsf D (G) \ge \mathsf D (G') \ge 16$.

Suppose that $|G/G'| \in \{ 540, 360, 270, 216, 180, 135, 120, 90, 60 \}$.
If $G/G' \cong C^{3}_6$ with $|G/G'| = 216$, then $\mathsf D (G/G') \ge 16$.
Otherwise, we infer that $G/G'$ has an element of order at least 10.
In either case, it follows by Lemmas~\ref{lem:ineqDC}.3 and \ref{lem:order9} that $\mathsf D (G) \ge \mathsf D (G/G') \ge 10$.

\smallskip
\noindent
{\bf SUBCASE 7.5.} $i = 4$, i.e., $|G| = 2^{4} 3^{3} 5^{1} = 2160$.
\smallskip

By Lemma~\ref{lem:perfect}, there exists a perfect group of order 2160, and thus if $G$ is such a perfect group of order 2160, then $G$ has a normal cyclic subgroup $H$ of order 6 such that $G/H \cong A_6$.
Let $\varphi \colon G \to G/H \cong A_6$ be the canonical epimorphism.
Since $A_6$ has a subgroup isomorphic to $A_5$, we infer that $G/H$ has a subgroup $K/H$ isomorphic to $A_5$, where $K$ is a subgroup of $G$ with $H \subseteq K$.
Thus, $K$ is a (non-abelian) subgroup of order $|K| = |A_5||H| = 360$, whence {\bf SUBCASE 6.3} ensures that $\mathsf D (G) \ge \mathsf D (K) \ge 10$.
Suppose now that $G$ is not a perfect group, so that $G$ has a non-trivial proper subgroup $G'$.
The possible values of the pair $\big( |G'|, |G/G'| \big)$ are as follows:
\[
  \begin{aligned}
    \big( |G'|, |G/G'| \big) \,\, \in & \,\, \big\{ (2,1080), \, (3,720), \, (4,540), \, (5,432), \, (6,360), \, (8,270), \, (9,240), \, (10,216), \, (12,180), \\
                                       & \quad (15,144), \, (16,135), \, (18,120), \, (20,108), \, (24,90), \, (27,80), \, (30,72), \, (36,60), \, (40,54),  \\
                                       & \quad (45,48), \, (48,45), \, (54,40), \, (60,36), \, (72,30), \, (80,27), \, (90,24), \, (108,20), \, (120,18), \\
                                       & \quad (135,16), \, (144,15), \, (180,12), \, (216,10), \, (240,9), \, (270,8), \, (360,6), \, (432,5), \, (540,4), \\
                                       & \quad (720,3), \, (1080,2) \big\} \,.
  \end{aligned}
\]

If $|G'|, |G/G'| \in \{ 45, 40, 36, 30, 20, 15, 10 \}$, then Lemma~\ref{lem:subgroup} ensures that $\mathsf D (G) \ge 10$.

Suppose that $|G'| \in \{ 1080, 720, 540, 432, 360, 270, 240, 180, 135, 120, 90, 80 \}$.
If $G'$ is non-abelian, then {\bf SUBCASES 7.4}--{\bf 7.1}, {\bf 6.4}--{\bf 6.1}, {\bf 5.3}--{\bf 5.2}, {\bf 3.4}, and {\bf CASE 4},  ensure that $\mathsf D (G) \ge \mathsf D (G') \ge 10$.
If $G'$ is abelian, then by (\ref{eq:order}), we only have that $G' \cong C_2 \times C^{3}_6$ with $|G'| = 432$, whence $\mathsf D (G) \ge \mathsf D (G') \ge 17$.

Suppose that $|G/G'| \in \{ 1080, 720, 540, 432, 360, 270, 240, 180, 135, 120, 90, 80 \}$.
If $G/G' \cong C_2 \times C^{3}_6$ with $|G/G'| = 432$, then $\mathsf D (G/G') \ge 17$.
Otherwise, we infer that $G/G'$ has an element of order at least 10.
In either case, it follows by Lemmas~\ref{lem:ineqDC}.3 and \ref{lem:order9} that $\mathsf D (G) \ge \mathsf D (G/G') \ge 10$.

\smallskip
\noindent
{\bf CASE 8.} $\ell = 1$ and $k = j = 0$, i.e., $|G| = 2^{i} 7^{1}$.
\smallskip

Then, since $|G| > 42$, it follows that $i \in [3,4]$.

\smallskip
\noindent
{\bf SUBCASE 8.1.} $i = 3$, i.e., $|G| = 2^{3} 7^{1} = 56$.
\smallskip

By Sylow's Theorem, $n_7 \in \{ 1, 8 \}$ and $n_2 \in \{ 1, 7 \}$.
If $n_7 = 1$, then $G$ has a unique Sylow 7-subgroup, and so we infer that $G$ must have a subgroup $H$ of order 14. By (\ref{eq:order}), $H \cong D_{14}$, and it follows by Lemma~\ref{lem:nabD} that $\mathsf D (G) \ge \mathsf D (H) = 14$.
If $n_7 = 8$, then any two Sylow 7-subgroups have a trivial intersection, and so $G$ contains 48 distinct elements of order 7.
Since $|G| = 56$, remaining 8 elements of $G$ that are not of order 7 must be contained in a Sylow 2-subgroup of $G$, whence $G$ has a unique Sylow 2-subgroup $K$.
Note that $\mathsf d (K) \ge 3$ if $K$ is abelian, and $\mathsf d (K) = 4$ if $K$ is non-abelian (Lemma~\ref{lem:dD}).
Since $|G/K| = 7$, it follows by Lemma~\ref{lem:ineqDC}.2 that $\mathsf D (G) \ge \mathsf d (G) + 1 \ge \big( \mathsf d (K) + \mathsf d (G/K) \big) + 1 \ge 10$.

\smallskip
\noindent
{\bf SUBCASE 8.2.} $i = 4$, i.e., $|G| = 2^{4} 7^{1} = 112$.
\smallskip

By Lemma~\ref{lem:perfect}, every group of order 112 is not a perfect group, and so such a group has a non-trivial proper subgroup $G'$.
The possible values of the pair $\big( |G'|, |G/G'| \big)$ are as follows:
\[
   \big( |G'|, |G/G'| \big) \,\, \in \,\, \big\{ (2,56), \, (4, 28), \, (7, 16), \, (8,14), \, (14,8), \, (16,7), \, (28,4), \, (56,2) \big\} \,.
\]

If $|G'|, |G/G'| \in \{ 28, 14 \}$, then Lemma~\ref{lem:subgroup} ensures that $\mathsf D (G) \ge 10$.

If $|G'| = 56$, then by (\ref{eq:order}), we infer that $G'$ must be non-abelian, and thus {\bf SUBCASE 8.1} ensures that $\mathsf D (G) \ge \mathsf D (G') \ge 10$.

If $|G/G'| = 56$, then $G/G'$ has an element of order 14, and it follows by Lemma~\ref{lem:ineqDC}.3 that $\mathsf D (G) \ge \mathsf D (G/G') \ge 14$.

If $|G/G'| = 16$, then $G/G'$ is isomorphic to one of the groups in $\{ C_{16}, C^{2}_4, C_2 \times C_8, C^{2}_2 \times C_4, C^{4}_2 \}$.
Since $|G'| = 7$, $G' \cong C_7$, and it follows by Lemma~\ref{lem:ineqDC}.2 that $\mathsf D (G) \ge \mathsf d (G) + 1 \ge \big( \mathsf d (G') + \mathsf d (G/G') \big) + 1 \ge 11$.

Suppose that $|G/G'| = 7$.
Then $|G'| = 16$, and if $G'$ is abelian, then we can easily check that $\mathsf d (G') \ge 4$.
If $G'$ is non-abelian, then by Lemma~\ref{lem:dD}, $\mathsf d (G') \ge 5$.
In either case, since $G/G' \cong C_7$, it follows by Lemma~\ref{lem:ineqDC}.2 that $\mathsf D (G) \ge \mathsf d (G) + 1 \ge \big( \mathsf d (G') + \mathsf d (G/G') \big) + 1 \ge 11$.

\smallskip
\noindent
{\bf CASE 9.} $\ell = j = 1$ and $k = 0$, i.e., $|G| = 2^{i} 3^{1} 7^{1}$.
\smallskip

Then, since $|G| > 42$, it follows that $i \in [2,4]$.

\smallskip
\noindent
{\bf SUBCASE 9.1.} $i = 2$, i.e., $|G| = 2^{2} 3^{1} 7^{1} = 84$.
\smallskip

By Sylow's Theorem, we infer that $G$ has a unique Sylow 7-subgroup $H$.
If $G/H$ is abelian, then $G/H$ is isomorphic to $C_{12}$, or to $C_2 \times C_6$, and hence $\mathsf d (G/H) \ge 6$.
If $G/H$ is non-abelian, then $\mathsf d (G/H) \ge 4$ by Lemma~\ref{lem:dD}.
In either case, it follows by Lemma~\ref{lem:ineqDC}.2 that $\mathsf D (G) \ge \mathsf d (G) + 1 \ge \big( \mathsf d (H) + \mathsf d (G/H) \big) + 1 \ge 11$.

\smallskip
\noindent
{\bf SUBCASE 9.2.} $i = 3$, i.e., $|G| = 2^{3} 3^{1} 7^{1} = 168$.
\smallskip

By Lemma~\ref{lem:perfect}, there exists a perfect group of order 168, and thus if $G$ is a perfect group of order 168, then $G \cong \PSL (2,7) \cong \GL (3,2)$ (see \cite{Br-Lo09} for the last isomorphism), and consider two elements
\[
  A = \begin{bmatrix} 1 & 0 & 1 \\ 1 & 0 & 0 \\ 0 & 1 & 0 \end{bmatrix}, \quad B = \begin{bmatrix} 1 & 0 & 0 \\ 0 & 0 & 1 \\ 0 & 1 & 1 \end{bmatrix} \, \in \, \GL (3,2) \,.
\]
Then, $G$ has a subgroup $H \cong \langle A, B \mid A^{7} = B^{3} = I, \text{ and } BAB^{-1} = A^{4} \rangle \cong C_7 \rtimes C_3$, and hence it follows by Lemma~\ref{lem:dD} that $\mathsf D (G) \ge \mathsf D (H) = 14$.
Suppose now that $G$ is not a perfect group, so that $G$ has a non-trivial proper subgroup $G'$.
The possible values of the pair $\big( |G'|, |G/G'| \big)$ are as follows:
\[
  \begin{aligned}
    \big( |G'|, |G/G'| \big) \,\, \in & \,\, \big\{ (2,84), \, (3,56), \, (4,42), \, (6,28), \, (7,24), \, (8,21), \, (12,14), \, (14,12), \, (21,8), \, (24,7), \\
                                       & \quad (28,6), \, (42,4), \, (56,3), \, (84,2) \big\} \,.
  \end{aligned}
\]

If $|G'|, |G/G'| \in \{ 42, 28,  21, 14 \}$, then Lemma~\ref{lem:subgroup} ensures that $\mathsf D (G) \ge 10$.

Suppose that $|G'| \in \{ 84, 56 \}$.
Then, by (\ref{eq:order}), we infer that $G'$ must be non-abelian, and hence {\bf SUBCASES 9.1} and {\bf 8.1} ensure that $\mathsf D (G) \ge \mathsf D (G') \ge 10$.

Suppose that $|G/G'| \in \{ 84, 56 \}$.
Then, we infer that $G/G'$ has an element of order 14, and thus it follows by Lemmas~\ref{lem:ineqDC}.3 and \ref{lem:order9} that $\mathsf D (G) \ge \mathsf D (G/G') \ge 10$.

If $|G/G'| = 24$, then $G/G'$ is isomorphic to $C_{24}$, or to $C_2 \times C_{12}$, or to $C^{2}_2 \times C_6$.
Since $G' \cong C_7$, it follows by Lemma~\ref{lem:ineqDC}.2 that $\mathsf D (G) \ge \mathsf d (G) + 1 \ge \big( \mathsf d (G') + \mathsf d (G/G') \big) + 1 \ge 14$.

If $|G/G'| = 7$, then $G/G' \cong C_7$ and $|G'| = 24$.
If $G'$ is non-abelian, then it follows by Lemma~\ref{lem:dD} that $\mathsf D (G) \ge \mathsf D (G') \ge 10$.
If $G'$ is abelian, then by (\ref{eq:order}), we only have that $G' \cong C^{2}_2 \times C_6$, and thus it follows by Lemma~\ref{lem:ineqDC}.2 that $\mathsf D (G) \ge \mathsf d (G) + 1 \ge \big( \mathsf d (G') + \mathsf d (G/G') \big) + 1 = 14$.

\smallskip
\noindent
{\bf SUBCASE 9.3.} $i = 4$, i.e., $|G| = 2^{4} 3^{1} 7^{1} = 336$.
\smallskip

By Lemma~\ref{lem:perfect}, there exists a perfect group of order 336, and thus if $G$ is a perfect group of order 336, then $G$ has a normal subgroup $H$ of order 2 such that $G/H \cong \PSL (2,7) \cong \GL (3,2)$ (see \cite{Br-Lo09} for the last isomorphism).
Let $\varphi \colon G \to G/H \cong \GL (3,2)$ be the canonical epimorphism.
As we have already seen that $\GL (3,2)$ has a subgroup isomorphic to $C_7 \rtimes C_3$, we infer that $G/H$ has a subgroup $K/H$ isomorphic to $C_7 \rtimes C_3$, where $K$ is a subgroup of $G$ with $H \subseteq K$.
Thus, $K$ is a (non-abelian) subgroup of order $|K| = |C_7 \rtimes C_3||H| = 42$, and hence Lemma~\ref{lem:dD} ensures that $\mathsf D (G) \ge \mathsf D (K) \ge 10$.
Suppose now that $G$ is not a perfect group, so that $G$ has a non-trivial proper subgroup $G'$.
The possible values of the pair $\big( |G'|, |G/G'| \big)$ are as follows:
\[
  \begin{aligned}
    \big( |G'|, |G/G'| \big) \,\, \in & \,\, \big\{ (2,168), \, (3,112), \, (4,84), \, (6,56), \, (7,48), \, (8,42), \, (12,28), \, (14,24), \, (16,21), \, (21,16), \\
                                       & \quad (24,14), \, (28,12), \, (42,8), \, (48,7), \, (56,6), \, (84,4), \, (112,3), \, (168,2) \big\} \,.
  \end{aligned}
\]

If $|G'|, |G/G'| \in \{ 42, 28, 21, 14 \}$, then Lemma~\ref{lem:subgroup} ensures that $\mathsf D (G) \ge 10$.

Suppose that $|G'| \in \{ 168, 112, 84, 56, 48 \}$.
If $G'$ is non-abelian, then {\bf SUBCASES  9.2}--{\bf 9.1}, {\bf 8.2}--{\bf 8.1}, and {\bf CASE 1} ensure that $\mathsf D (G) \ge \mathsf D (G') \ge 10$.
If $G'$ is abelian, then by (\ref{eq:order}), we only have that $G' \cong C^{3}_2 \times C_6$ with $|G'| = 48$.
Thus, it follows by Lemma~\ref{lem:ineqDC}.1 that $\mathsf D (G) \ge \mathsf D (G') + 1 \ge 10$.

Suppose that $|G/G'| \in \{ 168, 112, 84, 56, 48 \}$.
If $G/G' \cong C^{3}_2 \times C_6$ with $|G/G'| = 48$, then since $G' \cong C_7$, it follows by Lemma~\ref{lem:ineqDC}.2 that $\mathsf D (G) \ge \mathsf d (G) + 1 \ge \big( \mathsf d (G') + \mathsf d (G/G') \big) + 1 \ge 15$.
Otherwise, $G/G'$ has an element of order at least 12, and hence Lemmas~\ref{lem:ineqDC}.3 and \ref{lem:order9} that $\mathsf D (G) \ge \mathsf D (G/G') \ge 10$.

\smallskip
\noindent
{\bf CASE 10.} $\ell = 1$, $k = 0$, and $j = 2$, i.e., $|G| = 2^{i} 3^{2} 7^{1}$.
\smallskip

Then, since $|G| > 42$, it follows that $i \in [0,4]$.

\smallskip
\noindent
{\bf SUBCASE 10.1.} $i = 0$, i.e., $|G| = 3^{2} 7^{1} = 63$.
\smallskip

By Sylow's Theorem, we infer that $G$ has a unique Sylow 7-subgroup $H$.
Thus, it follows by Lemma~\ref{lem:ineqDC}.2 that $\mathsf D (G) \ge \mathsf d (G) + 1 \ge \big( \mathsf d (H) + \mathsf d (G/H) \big) + 1 \ge 11$.

\smallskip
\noindent
{\bf SUBCASE 10.2.} $i = 1$, i.e., $|G| = 2^{1} 3^{2} 7^{1} = 126$.
\smallskip

By Sylow's Theorem, we infer that $G$ has a unique Sylow 7-subgroup $H$.
If $G/H$ is abelian, then $G/H$ is isomorphic to $C_{18}$, or to $C_3 \times C_6$.
Thus, by Lemma~\ref{lem:ineqDC}.2, $\mathsf D (G) \ge \mathsf d (G) + 1 \ge \big( \mathsf d (H) + \mathsf d (G/H) \big) + 1 \ge 14$.
If $G/H$ is non-abelian, then it follows by Lemmas~\ref{lem:ineqDC}.3 and \ref{lem:dD} that $\mathsf D (G) \ge \mathsf D (G/H) \ge 10$.

\smallskip
\noindent
{\bf SUBCASE 10.3.} $i = 2$, i.e., $|G| = 2^{2} 3^{2} 7^{1} = 252$.
\smallskip

By Lemma~\ref{lem:perfect}, every group of order 252 is not a perfect group, and so such a group has a non-trivial proper subgroup $G'$.
The possible values of the pair $\big( |G'|, |G/G'| \big)$ are as follows:
\[
  \begin{aligned}
    \big( |G'|, |G/G'| \big) \,\, \in & \,\, \big\{ (2,126), \, (3,84), \, (4, 63), \, (6,42), \, (7, 36), \, (9,28), \, (12,21), \, (14,18), \, (18,14), \, (21,12), \\
                                       & \quad (28,9), \, (36,7), \, (42,6), \, (63,4), \, (84,3), \, (126,2) \big\} \,.
  \end{aligned}
\]

If $|G'|, |G/G'| \in \{ 42, 36, 28, 21, 14 \}$, then Lemma~\ref{lem:subgroup} ensures that $\mathsf D (G) \ge 10$.

Suppose that $|G'| \in \{ 126, 84, 63 \}$.
By (\ref{eq:order}), we infer that $G'$ must be non-abelian, and so {\bf SUBCASES 10.2}--{\bf 10.1}, and {\bf 9.1} ensure that $\mathsf D (G) \ge \mathsf D (G') \ge 10$.

Suppose that $|G/G'| \in \{ 126, 84, 63 \}$.
Since $G/G'$ is abelian, it is easy to see that $G/G'$ has an element of order 21, and it follows again by Lemma~\ref{lem:ineqDC}.3 that $\mathsf D (G) \ge \mathsf D (G/G') \ge 21$.

\smallskip
\noindent
{\bf SUBCASE 10.4.} $i = 3$, i.e., $|G| = 2^{3} 3^{2} 7^{1} = 504$.
\smallskip

By Lemma~\ref{lem:perfect}, there exists a perfect group of order 504, and so if $G$ is such a perfect group of order 504, then it follows by \cite{Bu99} that $G \cong \langle a, b \mid a^{7} = b^{2} = (ab)^{3} = (a^{3} b a^{5} b a^{3} b)^{2} = 1_G \rangle$ and $x = ba^{3}$ is an element of $G$ of order 9, a contradiction to (\ref{eq:order}).
Suppose now that $G$ is not a perfect group, so that $G$ has a non-trivial proper subgroup $G'$.
The possible values of the pair $\big( |G'|, |G/G'| \big)$ are as follows:
\[
  \begin{aligned}
    \big( |G'|, |G/G'| \big) \,\, \in & \,\, \big\{ (2,252), \, (3,168), \, (4,126), \, (6,84), \, (7,72), \, (8,63), \, (9,56), \, (12,42), \, (14,36), \, (18,28), \\
                                       & \quad (21,24), \, (24,21), \, (28,18), \, (36,14), \, (42,12), \, (56,9), \, (63,8), \, (72,7), \, (84,6), \, (126,4), \\
                                       & \quad (168,3), \, (252,2) \big\} \,.
  \end{aligned}
\]

If $|G'|, |G/G'| \in \{ 42, 36, 28, 21, 14 \}$, then Lemma~\ref{lem:subgroup} ensures that $\mathsf D (G) \ge 10$.

Suppose that $|G'| \in \{ 252, 168, 126, 84, 72, 63, 56 \}$.
If $G'$ is non-abelian, then {\bf SUBCASES 10.3}--{\bf 10.1}, {\bf 9.2}--{\bf 9.1}, {\bf 8.1}, and {\bf 2.1} ensure that $\mathsf D (G) \ge \mathsf D (G') \ge 10$.
If $G'$ is abelian, then by (\ref{eq:order}), we only have that $G' \cong C_2 \times C^{2}_6$ with $|G'| = 72$, whence $\mathsf D (G) \ge \mathsf D (G') \ge 12$.

Suppose that $|G/G'| \in \{ 252, 168, 126, 84, 72, 63, 56 \}$.
If $G/G' \cong C_2 \times C^{2}_6$ with $|G/G'| = 72$, then Lemma~\ref{lem:ineqDC}.3 ensures that $\mathsf D (G) \ge \mathsf D (G/G') \ge 12$.
Otherwise, we infer that $G/G'$ has an element of order at least 12, and it follows by Lemmas~\ref{lem:ineqDC}.3 and \ref{lem:order9} that $\mathsf D (G) \ge \mathsf D (G/G') \ge 10$.

\smallskip
\noindent
{\bf SUBCASE 10.5.} $i = 4$, i.e., $|G| = 2^{4} 3^{2} 7^{1} = 1008$.
\smallskip

By Lemma~\ref{lem:perfect}, every group of order 1008 is not a perfect group, and so such a group has a non-trivial proper subgroup $G'$.
The possible values of the pair $\big( |G'|, |G/G'| \big)$ are as follows:
\[
  \begin{aligned}
    \big( |G'|, |G/G'| \big) \,\, \in & \,\, \big\{ (2,504), \, (3,336), \, (4,252), \, (6,168), \, (7,144), \, (8,126), \, (9,112), \, (12,84), \, (14,72), \, (16, 63), \\
                                       & \quad (18,56), \, (21,48), \, (24,42), \, (28,36), \, (36,28), \, (42,24), \, (48,21), \, (56,18), \, (63, 16), \, (72, 14), \\
                                       & \quad (84,12), \, (112,9), \, (126,8), \, (144,7), \, (168,6), \, (252,4), \, (336,3), \, (504, 2) \big\} \,. \\
  \end{aligned}
\]

If $|G'|, |G/G'| \in \{ 42, 36, 28, 21, 14 \}$, then Lemma~\ref{lem:subgroup} ensures that $\mathsf D (G) \ge 10$.

Suppose that $|G'| \in \{ 504, 336, 252, 168, 144, 126, 112, 84, 63, 56 \}$.
If $G'$ is non-abelian, then {\bf SUBCASES 10.4}--{\bf 10.1}, {\bf 9.3}--{\bf 9.1}, {\bf 8.2}--{\bf 8.1}, and {\bf 2.2} ensure that $\mathsf D (G) \ge \mathsf D (G') \ge 10$.
If $G'$ is abelian, then by (\ref{eq:order}), we only have that $G' \cong C^{2}_2 \times C^{2}_6$ with $|G'| = 144$, and thus Lemma~\ref{lem:ineqDC}.1 that $\mathsf D (G) \ge \mathsf D (G') + 1 \ge 14$.

Suppose that $|G/G'| \in \{ 504, 336, 252, 168, 144, 126, 112, 84, 63, 56 \}$.
If $G/G' \cong C^{2}_2 \times C^{2}_6$ with $|G/G'| = 144$, then Lemma~\ref{lem:ineqDC}.3 implies that $\mathsf D (G) \ge \mathsf D (G/G') \ge 13$.
Otherwise, we infer that $G/G'$ has an element of order at least 12, and thus Lemmas~\ref{lem:ineqDC}.3 and \ref{lem:order9} imply that $\mathsf D (G) \ge \mathsf D (G/G') \ge 10$.

\smallskip
\noindent
{\bf CASE 11.} $\ell = 1$, $k = 0$, and $j = 3$, i.e., $|G| = 2^{i} 3^{3} 7^{1}$.
\smallskip

Then, since $|G| > 42$, it follows that $i \in [0,4]$.

\smallskip
\noindent
{\bf SUBCASE 11.1.} $i \in \{ 0, 1 \}$, i.e., $|G| = 3^{3} 7^{1} = 189$ or $|G| = 2^{1} 3^{3} 7^{1} = 378$.
\smallskip

By Sylow's Theorem, it is easy to see that $G$ has a unique Sylow 7-subgroup $H$.
By Cauchy's Theorem, $G$ has a subgroup $K$ of order 3, and since $H$ intersects trivially with $K$, we infer that $HK$ is a subgroup of $G$ of order 21.
In view of (\ref{eq:order}), we must have that $HK \cong C_7 \rtimes C_3$, and by Lemma~\ref{lem:dD} that $\mathsf D (G) \ge \mathsf D (HK) = 14$.

\smallskip
\noindent
{\bf SUBCASE 11.2.} $i = 2$, i.e., $|G| = 2^{2} 3^{3} 7^{1} = 756$.
\smallskip

By Lemma~\ref{lem:perfect}, every group of order 756 is not a perfect group, and so such a group has a non-trivial proper subgroup $G'$.
The possible values of the pair $\big( |G'|, |G/G'| \big)$ are as follows:
\[
  \begin{aligned}
    \big( |G'|, |G/G'| \big) \,\, \in & \,\, \big\{ (2,378), \, (3,252), \, (4,189), \, (6,126), \, (7,108), \, (9,84), \, (12,63), \, (14,54), \, (18, 42), \, (12,36), \\
                                       & \quad (27,28), \, (28,27), \, (36,21), \, (42,18), \, (54,14), \, (63,12), \, (84,9), \, (108,7), \, (126,6), \, (189,4), \\
                                       & \quad (252,3), \, (378,2) \big\} \,.
  \end{aligned}
\]

If $|G'|, |G/G'| \in \{ 42, 36, 28, 21, 14 \}$, then Lemma~\ref{lem:subgroup} ensures that $\mathsf D (G) \ge 10$.

Suppose that $|G'| \in \{ 378, 252, 189, 126, 108, 84, 63 \}$.
If $G'$ is non-abelian, then {\bf SUBCASES 11.1}, {\bf 10.3}--{\bf 10.1}, {\bf 9.1}, and {\bf 3.2} ensure that $\mathsf D (G) \ge \mathsf D (G') \ge 10$.
If $G'$ is abelian, then by (\ref{eq:order}), we only have that $G' \cong C_3 \times C^{2}_6$ with $|G'| = 108$, whence $\mathsf D (G) \ge \mathsf D (G') \ge 13$.

Suppose that $|G/G'| \in \{ 378, 252, 189, 126, 108, 84, 63 \}$.
If $G/G' \cong C_3 \times C^{2}_6$ with $|G/G'| = 108$, then $\mathsf D (G/G') \ge 13$.
Otherwise, we infer that $G/G'$ has an element of order at least 12, and so Lemma~\ref{lem:order9} implies that $\mathsf D (G/G') \ge 10$.
In either case, we infer by Lemma~\ref{lem:ineqDC}.3 that $\mathsf D (G) \ge \mathsf D (G/G') \ge 10$.

\smallskip
\noindent
{\bf SUBCASE 11.3.} $i = 3$, i.e., $|G| = 2^{3} 3^{3} 7^{1} = 1512$.
\smallskip

By Lemma~\ref{lem:perfect}, every group of order 1512 is not a perfect group, and so such a group has a non-trivial proper subgroup $G'$.
The possible values of the pair $\big( |G'|, |G/G'| \big)$ are as follows:
\[
  \begin{aligned}
    \big( |G'|, |G/G'| \big) \,\, \in & \,\, \big\{ (2,756), \, (3,504), \, (4,378), \, (6,252), \, (7,216), \, (8,189), \, (9,168), \, (12,126), \, (14,108), \\
                                       & \quad (18, 84), \, (21,72), \, (24,63), \, (27,56), \, (28,54), \, (36,42), \, (42,36), \, (54,28), \, (56,27), \\
                                       & \quad (63,24), \, (72,21), \, (84,18), \,(108,14), \, (126,12), \, (168,9), \, (189,8), \, (216,7), \, (252,6), \\
                                       & \quad (378,4), \, (504,3), \, (756,2) \big\} \,.
  \end{aligned}
\]

If $|G'|, |G/G'| \in \{ 42, 36, 28, 21, 14 \}$, then Lemma~\ref{lem:subgroup} ensures that $\mathsf D (G) \ge 10$.

Suppose that $|G'| \in \{ 756, 504, 378, 252, 216, 189, 168, 126, 84, 63, 56 \}$.
If $G'$ is non-abelian, then {\bf SUBCASES 11.2}--{\bf 11.1}, {\bf 10.4}--{\bf 10.1}, {\bf 9.2}--{\bf 9.1}, {\bf 8.1}, and {\bf 3.3} ensure that $\mathsf D (G) \ge \mathsf D (G') \ge 10$.
If $G'$ is abelian, then by (\ref{eq:order}), we only have that $G' \cong C^{3}_6$ with $|G'| = 216$, and hence $\mathsf D (G) \ge \mathsf D (G') \ge 16$.

Suppose that $|G/G'| \in \{ 756, 504, 378, 252, 216, 189, 168, 126, 84, 63, 56 \}$.
If $G/G' \cong C^{3}_6$ with $|G/G'| = 216$, then $\mathsf D (G/G') \ge 13$.
Otherwise, we infer that $G/G'$ has an element of order at least 12, and so Lemma~\ref{lem:order9} ensures that $\mathsf D (G/G') \ge 10$.
In either case, we infer by Lemma~\ref{lem:ineqDC}.3 that $\mathsf D (G) \ge \mathsf D (G/G') \ge 10$.

\smallskip
\noindent
{\bf SUBCASE 11.4.} $i = 4$, i.e., $|G| = 2^{4} 3^{3} 7^{1} = 3024$.
\smallskip

By Lemma~\ref{lem:perfect}, every group of order 3024 is not a perfect group, and so such a group has a non-trivial proper subgroup $G'$.
The possible values of the pair $\big( |G'|, |G/G'| \big)$ are as follows:
\[
  \begin{aligned}
    \big( |G'|, |G/G'| \big) \,\, \in & \,\, \big\{ (2,1512), \, (3,1008), \, (4,756), \, (6,504), \, (7,432), \, (8,378), \, (9,336), \, (12,252), \, (14,216), \\
                                       & \quad (16, 189), \, (18, 168), \, (21,144), \, (24,126), \, (27,112), \, (28,108), \, (36,84), \, (42,72), \, (48,63), \\
                                       & \quad (54,56), \, (56,54), \, (63,48), \, (72,42), \, (84,36), \, (108,28), \, (112,27), \, (126,24), \, (144,21), \\
                                       & \quad (168,18), \, (189,16), \, (216,14), \, (252,12), \, (336,9), \, (378,8), \, (432,7), \, (504,6), \, (756,4), \\
                                       & \quad (1008,3), \, (1512,2) \big\} \,.
  \end{aligned}
\]

If $|G'|, |G/G'| \in \{ 42, 36, 28, 21, 14 \}$, then Lemma~\ref{lem:subgroup} ensures that $\mathsf D (G) \ge 10$.

Suppose that $|G'| \in \{ 1512, 1008, 756, 504, 432, 378, 336, 252, 189, 168, 126, 112, 63, 56 \}$.
If $G'$ is non-abelian, then {\bf SUBCASES 11.3}--{\bf 11.1}, {\bf 10.5}--{\bf 10.1}, {\bf 9.3}--{\bf 9.2}, {\bf 8.2}--{\bf 8.1}, and {\bf 3.4} ensure that $\mathsf D (G) \ge \mathsf D (G') \ge 10$.
If $G'$ is abelian, then by (\ref{eq:order}), we only have that $G' \cong C_2 \times C^{3}_6$ with $|G'| = 432$, and thus $\mathsf D (G) \ge \mathsf D (G') \ge 17$.

Suppose that $|G/G'| \in \{ 1512, 1008, 756, 504, 432, 378, 336, 252, 189, 168, 126, 112, 63, 56 \}$.
If $G/G' \cong C_2 \times C^{3}_6$ with $|G/G'| = 432$, then $\mathsf D (G/G') \ge 17$.
Otherwise, we infer that $G/G'$ has an element of order at least 12, and so Lemma~\ref{lem:order9} ensures that $\mathsf D (G/G') \ge 10$.
In either case, it follows by Lemma~\ref{lem:ineqDC}.3 that $\mathsf D (G) \ge \mathsf D (G/G') \ge 10$.

\smallskip
\noindent
{\bf CASE 12.} $\ell = k = 1$ and $j = 0$, i.e., $|G| = 2^{i} 5^{1} 7^{1}$.
\smallskip

Then, since $|G| > 42$, it follows that $i \in [1,4]$, and if $i \in [1,2]$, then by Lemma~\ref{lem:35}, $G$ has an element of order 35.
Thus, we have that $i \in [3,4]$.

\smallskip
\noindent
{\bf SUBCASE 12.1.} $i = 3$, i.e., $|G| = 2^{3} 5^{1} 7^{1} = 280$.
\smallskip

By Lemma~\ref{lem:perfect}, every group of order 280 is not a perfect group, and so such a group has a non-trivial proper subgroup $G'$.
The possible values of the pair $\big( |G'|, |G/G'| \big)$ are as follows:
\[
  \begin{aligned}
    \big( |G'|, |G/G'| \big) \,\, \in & \,\, \big\{ (2,140), \, (4,70), \, (5,56), \, (7,40), \, (8,35), \, (10,28), \, (14,20), \, (20,14), \, (28,10), \, (35,8), \\
                                       & \quad (40,7), \, (56,5), \, (70,4), \, (140,2) \big\} \,.
  \end{aligned}
\]

If $|G'|, |G/G'| \in \{ 140, 70, 40, 35, 28, 30, 14, 10 \}$, then Lemma~\ref{lem:subgroup} ensures that $\mathsf D (G) \ge 10$.

If $|G'| = 56$, then by (\ref{eq:order}), it is easy to see that $G'$ must be non-abelian, and hence {\bf SUBCASE 8.1} ensures that $\mathsf D (G) \ge \mathsf D (G') \ge 10$.

If $|G/G'| = 56$, then we infer that $G/G'$ has an element of order 14, and thus it follows by Lemmas~\ref{lem:ineqDC}.3 and \ref{lem:order9} that $\mathsf D (G) \ge \mathsf D (G/G') \ge 10$.

\smallskip
\noindent
{\bf SUBCASE 12.2.} $i = 4$, i.e., $|G| = 2^{4} 5^{1} 7^{1} = 560$.
\smallskip

By Lemma~\ref{lem:perfect}, every group of order 560 is not a perfect group, and so such a group has a non-trivial proper subgroup $G'$.
The possible values of the pair $\big( |G'|, |G/G'| \big)$ are as follows:
\[
  \begin{aligned}
    \big( |G'|, |G/G'| \big) \,\, \in & \,\, \big\{ (2,280), \, (4,140), \, (5,112), \, (7,80), \, (8,70), \, (10,56), \, (14,40), \, (16,35), \, (20,28), \, (28,20), \\
                                       & \quad (35,16), \, (40,14), \, (56,10), \, (70,8), \, (80,7), \, (112,5), \, (140,4), \, (280,2) \big\} \,.
  \end{aligned}
\]

If $|G'|, |G/G'| \in \{ 140, 70, 40, 35, 28, 20, 14, 10 \}$, then Lemma~\ref{lem:subgroup} ensures that $\mathsf D (G) \ge 10$.

Suppose that $|G'| \in \{ 280, 112, 80 \}$.
By (\ref{eq:order}), it is easy to see that $G'$ must be non-abelian, and thus {\bf SUBCASES 12.1}, {\bf 8.2}, and {\bf CASE 4} ensure that $\mathsf D (G) \ge \mathsf D (G') \ge 10$.

Suppose that $|G/G'| \in \{ 280, 112, 80 \}$.
Then, we infer that $G/G'$ has an element of order at least 10, and thus Lemmas~\ref{lem:ineqDC}.3 and \ref{lem:order9} ensure that $\mathsf D (G) \ge \mathsf D (G/G') \ge 10$.

\smallskip
\noindent
{\bf CASE 13.} $\ell = k = j = 1$, i.e., $|G| = 2^{i} 3^{1} 5^{1} 7^{1}$.
\smallskip

Then, since $|G| > 42$, it follows that $i \in [0,4]$, and if $i = 0$, then by Lemma~\ref{lem:35}, $G$ has an element of order 35.
Thus, we have that $i \in [1,4]$.

\smallskip
\noindent
{\bf SUBCASE 13.1.} $i = 1$, i.e., $|G| = 2^{1} 3^{1} 5^{1} 7^{1} = 210$.
\smallskip

By Lemma~\ref{lem:perfect}, every group of order 210 is not a perfect group, and so such a group has a non-trivial proper subgroup $G'$.
The possible values of the pair $\big( |G'|, |G/G'| \big)$ are as follows:
\[
  \begin{aligned}
    \big( |G'|, |G/G'| \big) \,\, \in & \,\, \big\{ (2,105), \, (3,70), \, (5,42), \, (6,35), \, (7,30), \, (10,21), \, (14,15), \, (15,14), \, (21,10), \, (30,7), \\
                                       & \quad (35,6), \, (42,5), \, (70,3), \, (105,2) \big\} \,.
  \end{aligned}
\]

Then, Lemma~\ref{lem:subgroup} ensures that $\mathsf D (G) \ge 10$.

\smallskip
\noindent
{\bf SUBCASE 13.2.} $i = 2$, i.e., $|G| = 2^{2} 3^{1} 5^{1} 7^{1} = 420$.
\smallskip

By Lemma~\ref{lem:perfect}, every group of order 420 is not a perfect group, and so such a group has a non-trivial proper subgroup $G'$.
The possible values of the pair $\big( |G'|, |G/G'| \big)$ are as follows:
\[
  \begin{aligned}
    \big( |G'|, |G/G'| \big) \,\, \in & \,\, \big\{ (2,210), \, (3,140), \, (4,105), \, (5,84), \, (6,70), \, (7,60), \, (10,42), \, (12,35), \, (14,30), \, (15,28), \\
                                       & \quad (20,21), \, (21,20), \, (28,15), \,(30,14), \, (35,12), \, (42,10), \, (60,7), \, (70,6), \, (84,5), \, (105,4), \\
                                       & \quad (140,3), \, (210,2) \big\} \,.
  \end{aligned}
\]

If $|G'|, |G/G'| \in \{ 140, 105, 70, 42, 35, 30, 28, 21, 20, 15, 14, 10 \}$, then Lemma~\ref{lem:subgroup} ensures that $\mathsf D (G) \ge 10$.

Suppose that $|G'| \in \{ 210, 84, 60 \}$.
By (\ref{eq:order}), it is easy to see that $G'$ must be non-abelian, and thus {\bf SUBCASES 13.1}, {\bf 9.1}, and {\bf 5.1} ensure that $\mathsf D (G) \ge \mathsf D (G') \ge 10$.

Suppose that $|G/G'| \in \{ 210, 84, 60 \}$.
Then, we infer that $G/G'$ has an element of order at least 10, and it follows by Lemmas~\ref{lem:ineqDC}.3 and \ref{lem:order9} that $\mathsf D (G) \ge \mathsf D (G/G') \ge 10$.

\smallskip
\noindent
{\bf SUBCASE 13.3.} $i = 3$, i.e., $|G| = 2^{3} 3^{1} 5^{1} 7^{1} = 840$.
\smallskip

By Lemma~\ref{lem:perfect}, every group of order 840 is not a perfect group, and so such a group has a non-trivial proper subgroup $G'$.
The possible values of the pair $\big( |G'|, |G/G'| \big)$ are as follows:
\[
  \begin{aligned}
    \big( |G'|, |G/G'| \big) \,\, \in & \,\, \big\{ (2,420), \, (3,280), \, (4,210), \, (5,168), \, (6,140), \, (7,120), \, (8,105), \, (10,84), \, (12,70), \, (14, 60), \\
                                       & \quad (15,56), \, (20,42), \, (21,40), \, (24,35), \, (28,30), \,(30,28), \, (35,24), \, (40,21), \, (42, 20), \, (56, 15), \\
                                       & \quad (60,14), \, (70,12), \, (84,10), \, (105,8), \, (120,7), \, (140,6), \, (168,5), \, (210,4), \, (280,3), \, (420,2) \big\} \,.
  \end{aligned}
\]

If $|G'|, |G/G'| \in \{ 140, 105, 70, 42, 40, 35, 30, 28, 21, 20, 15, 14, 10 \}$, then Lemma~\ref{lem:subgroup} ensures that $\mathsf D (G) \ge 10$.

Suppose that $|G'| \in \{ 420, 280, 210, 168, 120 \}$.
By (\ref{eq:order}), it is easy to see that $G'$ must be non-abelian, and thus {\bf SUBCASES 13.2}--{\bf 13.1}, {\bf 12.1}, {\bf 9.2}, and {\bf 5.2} ensure that $\mathsf D (G) \ge \mathsf D (G') \ge 10$.

Suppose that $|G/G'| \in \{ 420, 280, 210, 168, 120 \}$.
Then, we infer that $G/G'$ has an element of order at least 10, and thus it follows by Lemmas~\ref{lem:ineqDC}.3 and \ref{lem:order9} that $\mathsf D (G) \ge \mathsf D (G/G') \ge 10$.

\smallskip
\noindent
{\bf SUBCASE 13.4.} $i = 4$, i.e., $|G| = 2^{4} 3^{1} 5^{1} 7^{1} = 1680$.
\smallskip

By Lemma~\ref{lem:perfect}, every group of order 1680 is not a perfect group, and so such a group has a non-trivial proper subgroup $G'$.
The possible values of the pair $\big( |G'|, |G/G'| \big)$ are as follows:
\[
  \begin{aligned}
    \big( |G'|, |G/G'| \big) \,\, \in & \,\, \big\{ (2,840), \, (3,560), \, (4,420), \, (5,336), \, (6,280), \, (7,240), \, (8,210), \, (10,168), \, (12,140), \\
                                       & \quad (14,120), \, (15,112), \, (16,105), \, (20,84), \, (21,80), \, (24,70), \, (28,60), \,(30,56), \, (35,48), \\
                                       & \quad (40,42), \, (42,40), \, (48,35), \, (56,30), \, (60,28), \, (70,24), \, (80,21), \, (84,20), \, (105,16), \\
                                       & \quad (112,15), \, (120,14), \, (140,12), \, (168,10), \, (210,8), \, (240,7), \, (280,6), \, (336,5), \\
                                       & \quad (420,4), \, (560,3), \, (840,2) \big\} \,.
  \end{aligned}
\]

If $|G'|, |G/G'| \in \{ 140, 105, 70, 42, 40, 35, 30, 28, 21, 20, 15, 14, 10 \}$, then Lemma~\ref{lem:subgroup} ensures that $\mathsf D (G) \ge 10$.

Suppose that $|G'| \in \{ 840, 560, 420, 336, 280, 240, 210 \}$.
By (\ref{eq:order}), it is easy to see that $G'$ must be non-abelian, and thus {\bf SUBCASES 13.3}--{\bf 13.1}, {\bf 12.2}--{\bf 12.1}, {\bf 9.3}, and {\bf 5.3} ensure that $\mathsf D (G) \ge \mathsf D (G') \ge 10$.

Suppose that $|G/G'| \in \{ 840, 560, 420, 336, 280, 240, 210 \}$.
Then, we infer that $G/G'$ has an element of order at least 10, and hence it follows by Lemmas~\ref{lem:ineqDC}.3 and \ref{lem:order9} that $\mathsf D (G) \ge \mathsf D (G/G') \ge 10$.

\smallskip
\noindent
{\bf CASE 14.} $\ell = k = 1$ and $j = 2$, i.e., $|G| = 2^{i} 3^{2} 5^{1} 7^{1}$.
\smallskip

Then, since $|G| > 42$, it follows that $i \in [0,4]$.

\smallskip
\noindent
{\bf SUBCASE 14.1.} $i = 0$, i.e., $|G| = 3^{2} 5^{1} 7^{1} = 315$.
\smallskip

By Lemma~\ref{lem:perfect}, every group of order 315 is not a perfect group, and so such a group has a non-trivial proper subgroup $G'$.
The possible values of the pair $\big( |G'|, |G/G'| \big)$ are as follows:
\[
   \big( |G'|, |G/G'| \big) \,\, \in \,\, \big\{ (3,105), \, (5,63), \, (7,45), \, (9,35), \, (15,21), \, (21,15), \, (35,9), \, (45,7), \, (63,5), \, (105,3) \big\} \,.
\]

If $|G'|, |G/G'| \in \{ 105, 45, 35, 21, 15 \}$, then Lemma~\ref{lem:subgroup} ensures that $\mathsf D (G) \ge 10$.

If $|G'| = 63$, then by (\ref{eq:order}), it is easy to see that $G'$ must be non-abelian, and thus {\bf SUBCASE 10.1} ensures that $\mathsf D (G) \ge \mathsf D (G') \ge 10$.

If $|G/G'| = 63$, then we infer that $G/G'$ has an element of order 21, whence it follows by Lemmas~\ref{lem:ineqDC}.3 and \ref{lem:order9} that $\mathsf D (G) \ge \mathsf D (G/G') \ge 21$.

\smallskip
\noindent
{\bf SUBCASE 14.2.} $i = 1$, i.e., $|G| = 2^{1} 3^{2} 5^{1} 7^{1} = 630$.
\smallskip

By Lemma~\ref{lem:perfect}, every group of order 630 is not a perfect group, and so such a group has a non-trivial proper subgroup $G'$.
The possible values of the pair $\big( |G'|, |G/G'| \big)$ are as follows:
\[
  \begin{aligned}
    \big( |G'|, |G/G'| \big) \,\, \in & \,\, \big\{ (2,315), \, (3,210), \, (5,126), \, (6,105), \, (7,90), \, (9,70), \, (10,63), \, (14,45), \, (15,42), \, (18,35), \\
                                       & \quad (21,30), \, (30,21), \, (35,18), \, (42,15), \, (45,14), \, (63,10), \, (70,9), \, (90,7), \, (105,6), \, (126,5), \\
                                       & \quad (210,3), \, (315,2)   \big\} \,.
  \end{aligned}
\]

If $|G'|, |G/G'| \in \{ 105, 70, 45, 42, 35, 30, 21, 15, 14, 10 \}$, then Lemma~\ref{lem:subgroup} ensures that $\mathsf D (G) \ge 10$.

If $|G'| \in \{ 315, 210, 126, 90 \}$, then by (\ref{eq:order}), it is easy to see that $G'$ must be non-abelian, whence {\bf SUBCASES 14.1}, {\bf 13.1}, {\bf 10.2}, and {\bf 6.1} ensure that $\mathsf D (G) \ge \mathsf D (G') \ge 10$.

If $|G/G'| \in \{ 315, 210, 126, 90 \}$, then we infer that that $G/G'$ has an element of order at least 10, and hence it follows by Lemmas~\ref{lem:ineqDC}.3 and \ref{lem:order9} that $\mathsf D (G) \ge \mathsf D (G/G') \ge 10$.

\smallskip
\noindent
{\bf SUBCASE 14.3.} $i = 2$, i.e., $|G| = 2^{2} 3^{2} 5^{1} 7^{1} = 1260$.
\smallskip

By Lemma~\ref{lem:perfect}, every group of order 1260 is not a perfect group, and so such a group has a non-trivial proper subgroup $G'$.
The possible values of the pair $\big( |G'|, |G/G'| \big)$ are as follows:
\[
  \begin{aligned}
    \big( |G'|, |G/G'| \big) \,\, \in & \,\, \big\{ (2,630), \, (3,420), \, (4,315), \, (5,252), \, (6,210), \, (7,180), \, (9,140), \, (10,126), \, (12,105), \\
                                       & \quad (14,90), \, (15,84), \, (18,70), \, (20,63), \, (21,60), \, (28,45), \, (30,42), \, (35,36), \, (36,35), \\
                                       & \quad (42,30), \, (45,28), \, (60,21), \, (63,20), \, (70,18), \, (84,15), \, (90,14), \, (105,12), \, (126,10), \\
                                       & \quad (140,9), \, (180,7), \, (210,6), \, (252,5), \, (315,4), \, (420,3), \, (630,2) \big\} \,.
  \end{aligned}
\]

If $|G'|, |G/G'| \in \{ 140, 105, 70, 45, 42, 36, 35, 30, 28, 21, 20, 15, 14, 10 \}$, then Lemma~\ref{lem:subgroup} ensures that $\mathsf D (G) \ge 10$.

If $|G'| \in \{ 630, 420, 315, 252, 210, 180 \}$, then by (\ref{eq:order}), it is easy to see that $G'$ must be non-abelian, whence {\bf SUBCASES 14.2}--{\bf 14.1}, {\bf 13.2}--{\bf 13.1}, {\bf 10.3}, and {\bf 6.2} ensure that $\mathsf D (G) \ge \mathsf D (G') \ge 10$.

If $|G/G'| \in \{ 630, 420, 315, 252, 210, 180 \}$, then we infer that $G/G'$ has an element of order at least 10, whence it follows by Lemmas~\ref{lem:ineqDC}.3 and \ref{lem:order9} that $\mathsf D (G) \ge \mathsf D (G/G') \ge 10$.

\smallskip
\noindent
{\bf SUBCASE 14.4.} $i = 3$, i.e., $|G| = 2^{3} 3^{2} 5^{1} 7^{1} = 2520$.
\smallskip

By Lemma~\ref{lem:perfect}, there exists a perfect group of order 2520, and hence if $G$ is such a perfect group of order 2520, then $G \cong A_7$, and $G$ has a subgroup $H$ isomorphic to $A_6$.
Thus it follows by {\bf SUBCASE 6.3} that $\mathsf D (G) \ge \mathsf D (A_6) \ge 10$.
Except in the case where $G \cong A_7$, we infer that $G$ has a non-trivial proper subgroup $G'$.
The possible values of the pair $\big( |G'|, |G/G'| \big)$ are as follows:
\[
  \begin{aligned}
    \big( |G'|, |G/G'| \big) \,\, \in & \,\, \big\{ (2,1260), \, (3,840), \, (4,630), \, (5,504), \, (6,420), \, (7,360), \, (8,315), \, (9,280), \, (10,252), \\
                                       & \quad (12,210), \, (14,180), \, (15,168), \, (18,140), \, (20,126), \, (21,120), \, (24,105), \, (28,90), \, (30,84), \\
                                       & \quad (35,72), \, (36,70), \, (40,63), \, (42,60), \, (45,56), \, (56,45), \, (60,42), \, (63,40), \, (70,36), \, (72,35), \\
                                       & \quad (84,30), \, (90,28), \, (105,24), \, (120,21), \, (126,20), \, (140,18), \, (168,15), \, (180,14), \, (210,12), \\
                                       & \quad (252,10), \, (280,9), \, (315,8), \, (360,7), \, (420,6), \, (504,5), \, (630,4), \, (840,3), \, (1260,2) \big\} \,.
  \end{aligned}
\]

If $|G'|, |G/G'| \in \{ 140, 105, 70, 45, 42, 40, 36, 35, 30, 28, 21, 20 15, 14, 10 \}$, then by Lemma~\ref{lem:subgroup}, $\mathsf D (G) \ge 10$.

If $|G'| \in \{ 1260, 840, 630, 504, 420, 360, 315, 280, 210 \}$, then by (\ref{eq:order}), it is easy to see that $G'$ must be non-abelian, whence {\bf SUBCASES 14.3}--{\bf 14.1}, {\bf 13.3}--{\bf 13.1}, {\bf 12.1}, {\bf 10.4}, and {\bf 6.3} ensure that $\mathsf D (G) \ge \mathsf D (G') \ge 10$.

If $|G/G'| \in \{ 1260, 840, 630, 504, 420, 360, 315, 280, 210 \}$, then we infer that $G/G'$ has an element of order at least 10, and thus it follows by Lemmas~\ref{lem:ineqDC}.3 and \ref{lem:order9} that $\mathsf D (G) \ge \mathsf D (G/G') \ge 10$.

\smallskip
\noindent
{\bf SUBCASE 14.5.} $i = 4$, i.e., $|G| = 2^{4} 3^{2} 5^{1} 7^{1} = 5040$.
\smallskip

By Lemma~\ref{lem:perfect}, there exists a perfect group of order 5040, and so if $G$ is such a perfect group of order 5040, then $G$ has a normal subgroup $H$ of order 2 such that $G/H \cong A_7$.
Let $\varphi \colon G \to G/H \cong A_7$ be the canonical epimorphism.
Since $A_6 \le A_7$, we infer that $G/H$ has a subgroup $K/H$ isomorphic to $A_6$, where $K$ is a subgroup of $G$ with $H \subseteq K$.
Thus, $K$ is a (non-abelian) subgroup of order $|K| = |A_6| |H| = 720$, and hence {\bf SUBCASE 6.4} ensures that $\mathsf D (G) \ge \mathsf D (K) \ge 10$.
Suppose now that $G$ is not a perfect group, so that $G$ has a non-trivial proper subgroup $G'$.
The possible values of the pair $\big( |G'|, |G/G'| \big)$ are as follows:
\[
  \begin{aligned}
    \big( |G'|, |G/G'| \big) \,\, \in & \,\, \big\{ (2,2520), \, (3,1680), \, (4,1260), \, (5,1008), \, (6,840), \, (7,720), \, (8,630), \, (9,560), \, (10,504), \\
                                       & \quad (12,420), \, (14,360), \, (15,336), \, (16,315), \, (18,280), \, (20,252), \, (21,240), \, (24,210), \\
                                       & \quad (28,180), \, (30,168), \, (35,144), \, (36,140), \, (40,126), \, (42,120), \, (45,112), \, (48,105), \\
                                       & \quad (56,90), \, (60,84), \, (63,80), \, (70,72), \, (72,70), \, (80,63), \, (84,60), \, (90,56), \, (105,48), \\
                                       & \quad (112,45), \, (120,42), \, (126,40), \, (140,36), \, (144,35), \, (168,30), \, (180,28), \, (210,24), \\
                                       & \quad (240,21), \, (252,20), \, (280,18), \, (315,16), \, (336,15), \, (360,14), \, (420,12), \, (504,10), \\
                                       & \quad (560,9), \, (630,8), \, (720,7), \, (840,6), \, (1008,5), \, (1260,4), \, (1680,3), \, (2520,2) \big\} \,.
  \end{aligned}
\]

If $|G'|, |G/G'| \in \{ 140, 105, 70, 45, 42, 40, 36, 35, 30, 28, 21, 20, 15, 14, 10 \}$, then by Lemma~\ref{lem:subgroup}, $\mathsf D (G) \ge 10$.

If $|G'| \in \{ 2520, 1680, 1260, 1008, 840, 720, 630, 560, 420, 315, 280, 210, 90, 84, 80 \}$, then by (\ref{eq:order}), it is easy to see that $G'$ must be non-abelian, whence {\bf SUBCASES 14.4}--{\bf 14.1}, {\bf 13.4}--{\bf 13.1}, {\bf 12.2}--{\bf 12.1}, {\bf 10.5}, {\bf 9.1}, {\bf 6.4}, {\bf 6.1}, and {\bf CASE 4} ensure that $\mathsf D (G) \ge \mathsf D (G') \ge 10$.

If $|G/G'| \in \{ 2520, 1680, 1260, 1008, 840, 720, 630, 560, 420, 315, 280, 210, 90, 84, 80 \}$, then we infer that $G/G'$ has an element of order at least 10, whence by Lemmas~\ref{lem:ineqDC}.3 and \ref{lem:order9}, $\mathsf D (G) \ge \mathsf D (G/G') \ge 10$.

\smallskip
\noindent
{\bf CASE 15.} $\ell = k = 1$ and $j = 3$, i.e., $|G| = 2^{i} 3^{3} 5^{1} 7^{1}$.
\smallskip

Then, since $|G| > 42$, it follows that $i \in [0,4]$.

\smallskip
\noindent
{\bf SUBCASE 15.1.} $i = 0$, i.e., $|G| = 3^{3} 5^{1} 7^{1} = 945$.
\smallskip

By Feit-Thompson Theorem, we infer that $G$ is a solvable group.
Since $\gcd (135,7) = 1$, it follows by \cite[Proposition II.7.14]{Hu80} that $G$ has a subgroup $H$ with $|H| = 3^{3} 5^{1} = 135$.
By (\ref{eq:order}), it is easy to see that $H$ must be non-abelian, whence {\bf SUBCASE 7.1} ensures that $\mathsf D (G) \ge \mathsf D (H) \ge 11$.

\smallskip
\noindent
{\bf SUBCASE 15.2.} $i = 1$, i.e., $|G| = 2^{1} 3^{3} 5^{1} 7^{1} = 1890$.
\smallskip

By Lemma~\ref{lem:perfect}, every group of order 1890 is not a perfect group, and so such a group has a non-trivial proper subgroup $G'$.
The possible values of the pair $\big( |G'|, |G/G'| \big)$ are as follows:
\[
  \begin{aligned}
    \big( |G'|, |G/G'| \big) \,\, \in & \,\, \big\{ (2,945), \, (3,630), \, (5,378), \, (6,315), \, (7,270), \, (9,210), \, (10,189), \, (14,135), \, (15,126), \\
                                       & \quad (18,105), \, (21,90), \, (27,70), \, (30,63), \, (35,54), \, (42,45), \, (45,42), \, (54,35), \, (63,30), \\
                                       & \quad (70,27), \, (90,21), \, (105,18), \, (126,15), \, (135,14), \, (189,10), \, (210,9), \, (270,7), \, (315,6), \\
                                       & \quad (378,5), \, (630,3), \, (945,2)  \big\} \,.
  \end{aligned}
\]

If $|G'|, |G/G'| \in \{ 105, 70, 45, 42, 35, 30, 21, 15, 14, 10 \}$, then Lemma~\ref{lem:subgroup} ensures that $\mathsf D (G) \ge 10$.

If $|G'| \in \{ 945, 630, 378, 315, 270, 210 \}$, then by (\ref{eq:order}), it is easy to see that $G'$ must be non-abelian, whence {\bf SUBCASES 15.1}, {\bf 14.2}--{\bf 14.1}, {\bf 13.1}, {\bf 11.1}, and {\bf 7.2} ensure that $\mathsf D (G) \ge \mathsf D (G') \ge 10$.

If $|G/G'| \in \{ 945, 630, 378, 315, 270, 210 \}$, then we infer that $G/G'$ has an element of order at least 10, and thus it follows by Lemmas~\ref{lem:ineqDC}.3 and \ref{lem:order9} that $\mathsf D (G) \ge \mathsf D (G/G') \ge 10$.

\smallskip
\noindent
{\bf SUBCASE 15.3.} $i = 2$, i.e., $|G| = 2^{2} 3^{3} 5^{1} 7^{1} = 3780$.
\smallskip

By Lemma~\ref{lem:perfect}, every group of order 3780 is not a perfect group, and so such a group has a non-trivial proper subgroup $G'$.
The possible values of the pair $\big( |G'|, |G/G'| \big)$ are as follows:
\[
  \begin{aligned}
    \big( |G'|, |G/G'| \big) \,\, \in & \,\, \big\{ (2,1890), \, (3,1260), \, (4,945), \, (5,756), \, (6,630), \, (7,540), \, (9,420), \, (10,378), \, (12,315), \\
                                       & \quad (14,270), \, (15,252), \, (18,210), \, (20,189), \, (21,180), \, (27,140), \, (28,135), \, (30,126), \\
                                       & \quad (35,108), \, (36,105), \, (42,90), \, (45,84), \, (54,70), \, (60,63), \, (63,60), \, (70,54), \, (84,45), \\
                                       & \quad (90,42), \, (105,36), \, (108,35), \, (126,30), \, (135,28), \, (140,27), \, (180, 21), \, (189,20), \\
                                       & \quad (210,18), \, (252,15), \, (270,14), \, (315,12), \, (378,10), \, (420,9), \, (540,7), \, (630,6), \, (756,5), \\
                                       & \quad (945,4), \, (1260,3), \, (1890,2)  \big\} \,.
  \end{aligned}
\]

If $|G'|, |G/G'| \in \{ 140, 105, 70, 45, 42, 36, 35, 30, 28, 21, 20, 15, 14, 10 \}$, then by Lemma~\ref{lem:subgroup}, $\mathsf D (G) \ge 10$.

If $|G'| \in \{ 1890, 1260, 945, 756, 630, 540, 420, 315, 210, 63 \}$, then by (\ref{eq:order}), it is easy to see that $G'$ must be non-abelian, whence {\bf SUBCASES 15.2}--{\bf 15.1}, {\bf 14.3}--{\bf 14.1}, {\bf 13.2}--{\bf 13.1}, {\bf 11.2}, {\bf 10.1}, and {\bf 7.3} ensure that $\mathsf D (G) \ge \mathsf D (G') \ge 10$.

If $|G/G'| \in \{ 1890, 1260, 945, 756, 630, 540, 420, 315, 210, 63 \}$, then we infer that $G/G'$ has an element of order at least 10, and it follows by Lemmas~\ref{lem:ineqDC}.3 and \ref{lem:order9} that $\mathsf D (G) \ge \mathsf D (G/G') \ge 10$.

\smallskip
\noindent
{\bf SUBCASE 15.4.} $i = 3$, i.e., $|G| = 2^{3} 3^{3} 5^{1} 7^{1} = 7560$.
\smallskip

By Lemma~\ref{lem:perfect}, there exists a perfect group of order 7560, and so if $G$ is such a perfect group of order 7560, then $G$ has a normal subgroup $H$ of order 3 such that $G/H \cong A_7$.
Let $\varphi \colon G \to G/H \cong A_7$ be the canonical epimorphism.
Since $A_6 \le A_7$, we infer that $G/H$ has a subgroup $K/H$ isomorphic to $A_6$, where $K$ is a subgroup of $G$ with $H \subseteq K$.
Thus, $K$ is a (non-abelian) subgroup of order $|K| = |A_6||H| = 1080$, and hence {\bf SUBCASE 7.4} ensures that $\mathsf D (G) \ge \mathsf D (K) \ge 10$.
Suppose now that $G$ is not a perfect group, so that $G$ has a non-trivial proper $G'$.
The possible values of the pair $\big( |G'|, |G/G'| \big)$ are as follows:
\[
  \begin{aligned}
    \big( |G'|, |G/G'| \big) \,\, \in & \,\, \big\{ (2,3780), \, (3,2520), \, (4,1890), \, (5,1512), \, (6,1260), \, (7,1080), \, (8,945), \, (9,840), \, (10,756), \\
                                       & \quad (12,630), \, (14,540), \, (15,504), \, (18,420), \, (20,378), \, (21,360), \, (24,315), \, (27,280), \, (28,270), \\
                                       & \quad (30,252), \, (35,216), \, (36,210), \, (40,189), \, (42,180), \, (45,168), \, (54,140), \, (56,135), \, (60,126), \\
                                       & \quad (63,120), \, (70,108), \, (72,105), \, (84,90), \, (90,84), \, (105,72), \, (108,70), \, (120,63), \, (126,60), \\
                                       & \quad (135,56), \, (140,54), \, (168,45), \, (180,42), \, (189,40), \, (210,36), \, (216,35), \, (252,30), \, (270,28), \\
                                       & \quad (280,27), \, (315,24), \, (360,21), \, (378,20), \, (420,18), \, (504,15), \, (540,14), \, (630,12), \, (756,10), \\
                                       & \quad (840,9), \, (945,8), \, (1080,7), \, (1260,6), \, (1512,5), \, (1890,4), \, (2520,3), \, (3780,2) \big\} \,.
  \end{aligned}
\]

If $|G'|, |G/G'|  \in \{ 140, 105, 70, 45, 42, 40, 36, 35, 30, 28, 21, 20, 15, 14, 10 \}$, then by Lemma~\ref{lem:subgroup}, $\mathsf D (G) \ge 10$.

If $|G'| \in \{ 3780, 2520, 1890, 1512, 1260, 1080, 945, 840, 630, 420, 315, 280, 135, 126, 120, 90 \}$, then by (\ref{eq:order}), it is easy to see that $G'$ must be non-abelian, whence {\bf SUBCASES 15.3}--{\bf 15.1}, {\bf 14.4}--{\bf 14.1}, {\bf 13.3}--{\bf 13.2}, {\bf 12.1}, {\bf 11.3}, {\bf 10.2}, {\bf 7.4}, {\bf 7.1}, {\bf 6.1}, and {\bf 5.2} ensure that $\mathsf D (G) \ge \mathsf D (G') \ge 10$.

If $|G/G'| \in \{ 3780, 2520, 1890, 1512, 1260, 1080, 945, 840, 630, 420, 315, 280, 135, 126, 120, 90 \}$, then we infer that $G/G'$ has an element of order at least 10, whence by Lemmas~\ref{lem:ineqDC}.3 and \ref{lem:order9}, $\mathsf D (G) \ge \mathsf D (G/G') \ge 10$.

\smallskip
\noindent
{\bf SUBCASE 15.5.} $i = 4$, i.e., $|G| = 2^{4} 3^{3} 5^{1} 7^{1} = 15120$.
\smallskip

Note that the Schur multiplier of $A_7$ is isomorphic to $C_6$ (see \cite[Theorem 2.11]{Ho-Hu92}).
According to the classification method used in \cite{Sa81}, we infer that there exists a perfect group of order 15120 as a Schur covering group of $A_7$ (see also \cite[Chapter 5]{Ho-Pl89}).
Let $G$ be a perfect group of order 15120.
Then, $G$ has a normal cyclic subgroup $H$ of order 6 such that $G/H \cong A_7$.
Let $\varphi \colon G \to G/H \cong A_7$ be the canonical epimorphism.
Since $A_6$ is a subgroup of $A_7$, we infer that $G/H$ has a subgroup $K/H$ isomorphic to $A_6$, where $K$ is a subgroup of $G$ with $H \subseteq K$.
Thus, $K$ is a (non-abelian) subgroup of order $|K| = |A_6| |H| = 2160$, and hence {\bf SUBCASE 7.5} ensures that $\mathsf D (G) \ge \mathsf D (K) \ge 10$.
Suppose now that $G$ is not a perfect group, so that $G$ has a non-trivial proper $G'$.
The possible values of the pair $\big( |G'|, |G/G'| \big)$ are as follows:
\[
  \begin{aligned}
    \big( |G'|, |G/G'| \big) \in & \,\, \big\{ (2,7560), \, (3,5040), \, (4,3780), \, (5,3024), \, (6,2520), \, (7,2160), \, (8,1890), (9,1680), (10,1512), \\
                                       & \quad (12,1260), \, (14,1080), \, (15,1008), \, (16,945), \, (18,840), \, (20,756), \, (21,720), \, (24,630), \\
                                       & \quad (27,560), \, (28,540), \, (30,504), \, (35,432), \, (36,420), \, (40,378), \, (42,360), \, (45,336), \, (48,315), \\
                                       & \quad (54, 280), \, (56,270), \, (60,252), \, (63,240), \, (70,216), \, (72,210), \, (80,189), \, (84,180), \, (90,168),  \\
                                       & \quad (105,144), \, (108,140), \, (112,135), \, (120,126), \, (126,120), \, (135,112), \, (140,108), \, (144,105), \\
                                       & \quad (168,90), \, (180,84), \, (189,80), \, (210,72), \, (216,70), \, (240,63), \, (252,60), \, (270,56), \, (280,54), \\
                                       & \quad (315,48), \, (336,45), \, (360,42), \, (378,40), \, (420,36), \, (432,35), \, (504,30), \, (540,28), \, (560,27), \\
                                       & \quad (630,24), \, (720,21), \, (756,20), \, (840,18), \, (945,16), \, (1008,15), \, (1080,14), \, (1260,12), \\
                                       & \quad (1512,10), (1680,9), (1890,8), (2160,7), (2520,6), (3024,5), \, (3780,4), (5040,3), (7560,2) \big\} \,.
  \end{aligned}
\]

If $|G'|, |G/G'| \in \{ 140, 105, 70, 45, 42, 40, 36, 35, 30, 28, 21, 20, 15, 14, 10 \}$, then by Lemma~\ref{lem:subgroup}, $\mathsf D (G) \ge 10$.

If $|G'| \in \{ 7560, 5040, 3780, 3024, 2520, 2160, 1890, 1680, 1260, 945, 840, 630, 560, 315, 280, 270, 252, 240, 210,$ $189, 180, 168, 135, 126 \}$, then (\ref{eq:order}), it is easy to see that $G'$ must be non-abelian, whence {\bf SUBCASES 15.4}--{\bf 15.1}, {\bf 14.5}--{\bf 14,1}, {\bf 13.4}--{\bf 13.3}, {\bf 13.1}, {\bf 12.2}--{\bf 12.1}, {\bf 11.4}, {\bf 11.1}, {\bf 10.3}--{\bf 10.2}, {\bf 9.2}, {\bf 7.5}, {\bf 7.2}--{\bf 7.1}, {\bf 6.2}, and {\bf 5.3} ensure that $\mathsf D (G) \ge \mathsf D (G') \ge 10$.

If $|G/G'| \in \{ 7560, 5040, 3780, 3024, 2520, 2160, 1890, 1680, 1260, 945, 840, 630, 560, 315, 280, 270, 252, 240,$ $210, 189, 180, 168, 135, 126 \}$, then we infer that $G/G'$ has an element of order at least 10, and thus it follows by Lemmas~\ref{lem:ineqDC}.3 and \ref{lem:order9} that $\mathsf D (G) \ge \mathsf D (G/G') \ge 10$.

\medskip
\noindent
We now cover all cases, and this completes the proof of Lemma~\ref{lem:DC}.
\end{proof}

\medskip
\begin{proof}[Proof of Theorem~\ref{thm:classify}]
Let $G$ be a non-trivial finite group with $\mathsf D (G) < 10$.
Suppose first that $G$ is non-abelian.
If $\mathsf D (G) \le 7$, then Lemma~\ref{lem:DC} ensures that $|G| \le 42$.
If $8 \le \mathsf D (G) \le 9$, then it follows again by Lemma~\ref{lem:DC} that either $|G| \le 42$, or that $G$ has a proper subgroup of order 32.
Thus, except for groups having a proper subgroup of order 32, the classification of such non-abelian groups follows from Lemma~\ref{lem:dD}.
Now, we may assume that $G$ is abelian, say $G \cong C_{n_1} \times \cdots \times C_{n_r}$ with $1 < n_1 \mid \cdots \mid n_r$.
If $n_r \ge 10$, then we obtain that $\mathsf D (G) \ge \mathsf D (C_{10}) = 10$, and hence we must have that $n_r \le 9$.

If $n_r = 2$, then $G \cong C^{r}_2$, and by Lemma~\ref{lem:abD}.(a), $r+1 = \mathsf D (G) < 10$, whence $G \cong C^{r}_2$ with $r \le 8$.

If $n_r = 3$, then $G \cong C^{r}_3$, and by Lemma~\ref{lem:abD}.(a), $2r + 1 = \mathsf D (G) < 10$, whence $G \cong C^{r}_3$ with $r \le 4$.

If $n_r = 4$, then $G \cong C^{t}_2 \times C^{\ell}_4$ for some $t \in \mathbb N_0$ and $\ell \in \mathbb N$.
Then, by Lemma~\ref{lem:abD}.(a), $t + 3\ell + 1 = \mathsf D (G) < 10$, whence $G \cong C^{t}_2 \times C^{\ell}_4$, where $\ell \in \{ 1, 2 \}$ if $0 \le t \le 2$, and $\ell = 1$ if $3 \le t \le 5$.

If $n_r = 5$, then $G \cong C^{r}_5$, and by Lemma~\ref{lem:abD}.(a), $4r + 1 = \mathsf D (G) < 10$, whence $G \cong C^{r}_5$ with $r \le 2$.

If $n_r = 6$, then either $G \cong C^{t}_2 \times C^{\ell}_6$ or $G \cong C^{i}_3 \times C^{j}_6$ for some $t, i \in \mathbb N_0$ and $\ell, j \in \mathbb N$.
Thus, either $t + 5\ell +1 \le \mathsf D (G) < 10$ or $2i + 5j + 1 \le \mathsf D (G) < 10$, and hence the possible values of either $(t, \ell)$ or $(i,j)$ are as follows:
\begin{equation} \label{eq:exp6}~
  (t, \ell) \,\, \in \,\, \{ (0,1), (1,1), (2,1), (3,1)  \} \quad \und \quad (i,j) \,\, \in \,\, \{ (0,1), (1,1) \} \,.
\end{equation}
Thus, we infer that $\mathsf D (G) = t + 5\ell + 1$ if $G \cong C^{t}_2 \times C^{\ell}_6$ with $(t,\ell)$ from (\ref{eq:exp6}) by Lemma~\ref{lem:abD}.(c), and $\mathsf D (G) = 2i + 5j + 1$ if $G \cong C^{i}_3 \times C^{j}_6$ with $(i,j)$ from (\ref{eq:exp6}) by Lemma~\ref{lem:abD}.(b).
Hence, it follows that either $G \cong C^{t}_2 \times C^{\ell}_6$ or $G \cong C^{i}_3 \times C^{j}_6$, where $(t, \ell)$ and $(i,j)$ from (\ref{eq:exp6}).

If $n_r = 7$, then $G \cong C^{r}_7$, and by Lemma~\ref{lem:abD}.(a), $6r + 1 = \mathsf D (G) < 10$, whence $G \cong C_7$.

If $n_r = 8$, then $G \cong C^{t}_2 \times C^{\ell}_4 \times C^{k}_8$ for some $t, \ell \in \mathbb N_0$ and $k \in \mathbb N$, and by Lemma~\ref{lem:abD}.(a), $t + 3\ell + 7k + 1 = \mathsf D (G) < 10$.
Thus, it follows that either $G \cong C_8$ or $G \cong C_2 \times C_8$.

If $n_r = 9$, then $G \cong C^{t}_3 \times C^{\ell}_9$ for some $t \in \mathbb N_0$ and $\ell \in \mathbb N$, and by Lemma~\ref{lem:abD}.(a), $2t + 8\ell + 1 = \mathsf D (G) < 10$, whence $G \cong C_9$.
This completes our proof of Theorem~\ref{thm:classify}.
\end{proof}


\bigskip
\noindent
{\bf Acknowledgement.} I thank Alfred Geroldinger for his helpful comments on a preliminary version of this paper. I would also like to thank the anonymous referees for their time and effort. They provided a detailed and valuable comments, which greatly helped to improve the presentation of this paper.


\bigskip
\providecommand{\bysame}{\leavevmode\hbox to3em{\hrulefill}\thinspace}
\providecommand{\MR}{\relax\ifhmode\unskip\space\fi MR }
\providecommand{\MRhref}[2]{%
  \href{http://www.ams.org/mathscinet-getitem?mr=#1}{#2}
}
\providecommand{\href}[2]{#2}

\bigskip

\end{document}